\documentclass[a4paper]{article}

\usepackage{dsfont}

\usepackage{amsmath,amsthm}
\usepackage{amsfonts}
\usepackage{graphicx}

\newcommand{\mean}{\operatorname{mean}}
\newcommand{\e}{\mathbf e_0}
\newcommand{\A}{\mathbf A}

\newcommand{\I}{\mathbf I}
\newcommand{\bb}{\mathbf b}

\newcommand{\M}{\mathbf M}

\newcommand{\nn}{\mathbf n}

\newcommand{\T}{\mathbf T}
\newcommand{\Talg}{\mathbf T_{alg}}
\newcommand{\Tss}{\mathbf{T}_{ss}}
\newcommand{\tpq}{t_{pq}}
\newcommand{\pp}{\mathbf{e}_1}
\newcommand{\qq}{\mathbf{e}_2}
\newcommand{\ppalg}{\mathbf{e}_{1, alg}}
\newcommand{\qqalg}{\mathbf{e}_{2, alg}}

\newcommand{\Ts}{\T_s}
\newcommand{\Xt}{\mathbf X_t}
\newcommand{\XM}{\mathbf X_M}
\newcommand{\X}{\mathbf X}
\newcommand{\Xs}{\mathbf X_s}
\newcommand{\Xss}{\mathbf X_{ss}}
\newcommand{\Tt}{\mathbf T_t}
\newcommand{\ez}{\mathbf e_3}

\newtheorem{teo}{Theorem}[section]

\newtheorem{lema}[teo]{Lemma}
\newtheorem{coro}[teo]{Corollary}

{\theoremstyle{remark} \newtheorem{remark}[teo]{Remark}}

\begin{document}

\title{On the Relationship between the One-Corner Problem and the $M$-Corner Problem for the Vortex Filament Equation\thanks{This work was supported by an ERCEA Advanced Grant 2014 669689 - HADE, by the MINECO projects MTM2014-53850-P and  SEV-2013-0323 and by the Basque Government project IT641-13.}}

\author{Francisco de la Hoz (UPV/EHU) and Luis Vega (UPV/EHU-BCAM)}

\maketitle

\begin{abstract}
In this paper, we give evidence that the evolution of the Vortex Filament Equation for a regular $M$-corner polygon as initial datum can be explained at infinitesimal times as the superposition of $M$ one-corner initial data. Therefore, and due to periodicity, the evolution at later times can be understood as the nonlinear interaction of infinitely many filaments, one for each corner. This interaction turns out to be some kind of nonlinear Talbot effect. We also give very strong numerical evidence of the transfer of energy and linear momentum for the $M$-corner case.
\end{abstract}

\section{Introduction}

The binormal flow,
\begin{equation*}
\Xt = \kappa\mathbf b,
\end{equation*}

\noindent where $t$ is the time, $\kappa$ the curvature, and $\mathbf b$ the binormal component of the
Frenet-Serret formulae, appeared first in 1906 \cite{darios} as an approximation of the dynamics of a vortex filament under the Euler equations, and was rederived in \cite{arms}, in an attempt to describe the evolution of the coherent structures that appear in turbulent flows. It is also known as the vortex filament equation (VFE) or the localized induction approximation (LIA). The reason for the latter is that just local effects are considered in the Biot-Savart integral that allows to compute the velocity from the vorticity. This is assumed to be a very strong hypothesis \cite{Saffman}, because, among other things, it does not allow for the possibility of the streching of the filament. Nevertheless, at the qualitative level, VFE seems to capture some of the important examples of vortex filaments, namely the straight line, the circle, and the helix (see also \cite{JS}, for some recent theoretical results).

The equation is equivalent to
\begin{equation}
\label{e:xt}\Xt = \Xs\wedge\Xss,
\end{equation}

\noindent where $\wedge$ is the usual cross-product, and $s$ is the arc-length parameter. Since the length of the tangent vector $\T = \X_s$ remains constant, we can assume, without loss of generality, that $|T| = 1$, for all $t$ (along this paper, we will simply write $|\cdot|$ instead of $\|\cdot\|_2$, in order to denote the Euclidean norm of a vector). Differentiating \eqref{e:xt} with respect to $s$, we get the Schr\"odinger map onto the sphere:
\begin{equation}
\label{e:schmap}\Tt = \T\wedge\Tss.
\end{equation}

\noindent In this work, we are interested in the evolution of \eqref{e:xt}-\eqref{e:schmap} for initial data with corners. The existence of solutions for one-corner initial data,
\begin{equation}
\label{e:1cornerproblem}
\begin{cases}
\Xt = \Xs \wedge \Xss, \qquad t > 0, \quad s\in\mathbb R,
 \cr
\X(s, 0) = \A^-\chi_{(-\infty,0]}(s) + \A^+ s\chi_{[0,+\infty)}(s), \qquad |\A^-| = |\A^+| = 1,
\end{cases}
\end{equation}

\noindent which is the simplest case, has been proven in \cite{GutierrezRivasVega2003} (see also \cite{delahoz2007} for the corresponding hyperbolic space problem); and numerical simulations of these solutions have been carried out in \cite{buttke87,DelahozGarciaCerveraVega09}. Furthermore, the fact that this kind of solutions yields a well-posed problem has been shown in a series of papers \cite{BV0,BV1,BV2,BV3}; in particular, \cite{BV3}, closes the question, because it proves that the problem with one-corner initial data is well-posed in an adequate function space.

Along this paper, all the vectors are given in column form. Since \eqref{e:xt} is rotation invariant, it can be assumed, without loss of generality, that $\A^-$ and $\A^+$ are of the form
\begin{equation}\label{e:A+A-}
\A^- = (A_1, -A_2, -A_3)^T, \qquad \A^+ = (A_1, A_2, A_3)^T.
\end{equation}

\noindent Then, using selfsimilarity arguments, we can conclude that the solution of \eqref{e:1cornerproblem} is a one-parameter family of regular curves developing a singularity at finite time. More precisely, looking for solutions satisfying 
\begin{equation}
\label{e:1selfsimilar}\X(s, t) = \sqrt t\X(s/\sqrt t, 1),\quad t>0,
\end{equation}
 it can be proved that the solution can be described by the following system of ODEs:
\begin{equation}
\label{e:1cornerproblemsystem}
\begin{split}
\Xs & = \T,
 \cr
\begin{pmatrix}
\T \cr \nn \cr \bb
\end{pmatrix}_s
& =
\begin{pmatrix}
0 & \frac{c_0}{\sqrt t} & 0
 \cr
-\frac{c_0}{\sqrt t} & 0 & \frac{s}{2t}
 \cr
0 & -\frac{s}{2t} & 0
\end{pmatrix}
\cdot
\begin{pmatrix}
\T \cr \nn \cr \bb
\end{pmatrix},
\end{split}
\end{equation}

\noindent with initial conditions
\begin{equation}
\label{e:1cornerproblemsystem0}
\begin{cases}
\X(0, t) = 2c_0\sqrt t(0, 0, 1)^T,
 \cr
\T(0, t) = (1, 0, 0)^T,
 \cr
\nn(0, t) = (1, 0, 0)^T,
 \cr
\bb(0, t) = (1, 0, 0)^T;
\end{cases}
\end{equation}

\noindent and where $A_1$ and the parameter family are related by
\begin{equation}
\label{e:A1}
A_1 = e^{-(c_0^2/2)\pi}.
\end{equation}

\noindent Even if $A_2$ and $A_3$ trivially satisfy $A_2^2 + A_3^2 = 1 - e^{-c_0^2\pi}$, expressing them individually as functions of $c_0$ requires much more involved relationships \cite{GutierrezRivasVega2003}:
\begin{equation*}
\begin{split}
A_2 & = 1 - \frac{e^{-(c_0^2/4)\pi}}{8\pi}\sinh(\pi c_0^2/2)|c_0\Gamma(ic_0^2/4) + 2^{i\pi/4}\Gamma(1/2 + ic_0^2/4)|^2,
 \cr
A_3 & = 1 - \frac{e^{-(c_0^2/4)\pi}}{8\pi}\sinh(\pi c_0^2/2)|c_0\Gamma(ic_0^2/4) - 2e^{-i\pi/4}\Gamma(1/2 + ic_0^2/4)|^2.
\end{split}
\end{equation*}

\noindent Along this paper, we use 
\begin{equation}
\label{e:selfsimilar}\X_{c_0},\quad \T_{c_0}, \quad\nn_{c_0},\quad \text{and}\quad \bb_{c_0},
\end{equation}

\noindent whenever we need to refer explicitly to the solution of the one-corner problem corresponding to a given value of $c_0$, and with $\A^-$ and $\A^+$ being of the form \eqref{e:A+A-}, i.e., the solution of \eqref{e:1cornerproblemsystem} and \eqref{e:1cornerproblemsystem0}. On the other hand, if $\theta$ denotes the inner angle of the corner, it is straightforward to check that
\begin{equation}
\label{e:thetaA1}
\cos(\theta) = 1 - 2A_1^2.
\end{equation}

\noindent In the especial case $\theta = \pi$, we have $c_0 = 0$, $(A_1, A_2, A_3)^T = (1, 0, 0)^T$, and the solution of \eqref{e:1cornerproblem} is simply $\X_{c_0}(s, t) = s(1, 0, 0)^T$, i.e., a line.

Even if the one-corner problem is well understood, the multiple-corner problem has started to receive attention only recently \cite{Didier,Didier2}. In \cite{HozVega2014}, we have studied for the first time the evolution of \eqref{e:xt}-\eqref{e:schmap}, taking a regular planar polygon of $M$ sides as the initial datum; this case will be referred to as the $M$-corner problem. Later on, in \cite{HozVega2014b}, we have shown the essentially random character of the evolution. The main ideas of \cite{HozVega2014} are as follows. In order to avoid working with the curvature $\kappa$ and the torsion $\tau$, we reformulate the Frenet-Serret formulae, without loss of generality, as
\begin{equation}
\label{e:Te1e2}
\begin{pmatrix}
\T \cr \pp \cr \qq
\end{pmatrix}_s =
\begin{pmatrix}
 0 & \alpha & \beta
 \cr
 - \alpha & 0 & 0 \cr -\beta & 0
& 0
\end{pmatrix} \cdot
\begin{pmatrix}
 \T \cr \pp \cr \qq
\end{pmatrix}.
\end{equation}

\noindent Then, the Hasimoto transformation \cite{hasimoto} adopts the form
\begin{equation*}
\psi = \alpha + i\beta,
\end{equation*}

\noindent and transforms \eqref{e:xt}-\eqref{e:schmap} into the nonlinear Schr\"odinger (NLS) equation:
\begin{equation}
\label{e:schr}
\psi_t = i\psi_{ss} + i\left(\frac{1}{2}(|\psi|^2 + A(t))\right)\psi,
\end{equation}

\noindent where $A(t)$ is a certain time-dependent real constant. The idea is to work with \eqref{e:schr}, and, at a given $t$, to recover $\X(s, t)$ and $\T(s, t)$ from $\psi(s, t)$, by integrating \eqref{e:Te1e2} up to a rigid movement that can be determined by the symmetries of the problem.

Observe that  $\int_s | \psi(s,t)|^2\,ds$ is formally preserved; as a consequence, we have that, at the formal level, the following two energies are also conserved:
\begin{equation}
\label{energy1}
\int_s |\X_t(s,t)|^2\, ds,
\end{equation}
\begin{equation}
\label{energy2}
\int_s |\T_s(s,t)|^2\, ds.
\end{equation}

\noindent The first one is related to the kinetic energy of the vortex filament \eqref{e:xt}, while the second one can be seen as an interchange energy, if we recall that \eqref{e:schmap} is related to the Landau-Lifshitz equation of ferromagnetism \cite{Ishimori1982,Lakshmanan2011}.

Given a regular planar polygon of $M$ sides as $\X(s, 0)$, there is no torsion; hence, $\psi(s, 0)$ is precisely the curvature of the polygon, which is a $2\pi/M$-periodic sum of Dirac deltas:
\begin{equation*}
\psi(s, 0) \equiv \kappa(s) = \frac{2\pi}{M}\sum_{k = -\infty}^{+\infty}\delta\left(s - \frac{2\pi k}{M}\right).
\end{equation*}

\noindent Then, bearing in mind the Galilean invariance of \eqref{e:schr} and, assuming uniqueness, we are able to obtain $\psi(s, t)$ at any rational multiple of $2\pi / M^2$. Defining $t_{pq} \equiv (2\pi/M^2)(p/q)$, $\gcd(p, q) = 1$, we show in \cite{HozVega2014} that
\begin{equation}
\label{e:hatpsistpq}
\hat\psi(k, t_{pq})  = \hat\psi(0, t_{pq})e^{-i(Mk)^2},
\end{equation}
\noindent where $\hat\psi(0, t_{pq})$, which, without loss of generality, is assumed to be real, is the mean of $\psi(s, t_{pq})$ over a period:
\begin{equation*}
\hat\psi(0, t_{pq}) = \frac{M}{2\pi}\int_0^{2\pi/M}\psi(s, t_{pq})ds.
\end{equation*}
\noindent As a consequence, we obtain that
\begin{equation}
\label{e:psistpq}
\psi(s, t_{pq})= \frac{2\pi}{Mq}\hat\psi(0, t_{pq})\sum_{k = -\infty}^{+\infty}\sum_{m = 0}^{q - 1}G(-p, m, q) \delta\left(s - \frac{2\pi k}{M} - \frac{2\pi m}{Mq}\right),
\end{equation}

\noindent where
\begin{equation*}
G(a, b, c) = \sum_{l=0}^{c - 1}e^{2\pi i (al^2 + bl)/c}
\end{equation*}

\noindent denotes a generalized quadratic Gau{\ss} sum. An important property is that $G$ can be represented as
\begin{equation*}
G(-p, m, q) =
\begin{cases}
\sqrt q e^{i \theta_m}, & \mbox{if $q\equiv 1\bmod2$},
 \\
\sqrt{2q} e^{i \theta_m}, & \mbox{if $q\equiv 0\bmod2$ $\wedge$ $q/2\equiv m\bmod 2$},
 \\
0, & \mbox{if $q\equiv 0\bmod2$ $\wedge$ $q/2\not\equiv m\bmod 2$},
\end{cases}
\end{equation*}

\noindent for certain $\theta_m$ depending also on $q$. Hence, if we define
\begin{equation}
\label{e:rhom}
\rho_m =
\begin{cases}
\dfrac{2\pi}{M\sqrt q}\hat\psi(0, t_{pq}), & \mbox{if $q\equiv1\bmod2$},
 \cr
\dfrac{2\pi}{M\sqrt {\tfrac{q}{2}}}\hat\psi(0, t_{pq}), & \mbox{if $q\equiv0\bmod2$ $\wedge$ $q/2\equiv m\bmod 2$},
 \cr
0, & \mbox{if $q\equiv0\bmod2$ $\wedge$ $q/2\not\equiv m\bmod 2$},
\end{cases}
\end{equation}

\noindent we can represent \eqref{e:psistpq} as
\begin{equation}
\label{psistpq}
\psi(s, t_{pq}) = \sum_{k=-\infty}^{+\infty}\sum_{m = 0}^{q - 1}\rho_m e^{i \theta_m} \delta\left(s - \frac{2\pi k}{M} - \frac{2\pi m}{Mq}\right).
\end{equation}

\noindent The coefficients multiplying the Dirac deltas are in general not real, except for $t = 0$ and $t_{1,2} = \pi/M^2$. Moreover, when $q$ is even, half of the $\rho_m$ are zero. Therefore, $\psi(s, t_{pq})$ does not correspond to a planar polygon, but to a skew polygon with $Mq$ (for $q$ odd) or $Mq/2$ (for $q$ even) equal-lengthed sides.

In order to recover $\X$ and $\T$ from $\psi$, we observe that every addend $\rho_m e^{i\theta_m}\delta(s - (2\pi m)/(Mq))$ in \eqref{psistpq}, with $\rho_m\not=0$, induces a rotation on $\T$, $\pp$ and $\qq$. Denoting $\rho_m\not=0$ simply as $\rho$,
\begin{equation}
\label{e:Mm}
\M_m =
\begin{pmatrix}
\cos(\rho) & \sin(\rho) \cos(\theta_m) & \sin(\rho) \sin(\theta_m)
 \\
-\sin(\rho) \cos(\theta_m) & \cos(\rho) \cos^2(\theta_m) + \sin^2(\theta_m) & (\cos(\rho) - 1)\cos(\theta_m)\sin(\theta_m)
 \\
-\sin(\rho) \sin(\theta_m) & (\cos(\rho) - 1) \cos(\theta_m)\sin(\theta_m) & \cos(\rho) \sin^2(\theta_m) + \cos^2(\theta_m)
\end{pmatrix}
\end{equation}

\noindent is the matrix such that
\begin{equation*}
\left(
 \begin{array}{c}
\T(\tfrac{2\pi m}{Mq}^+)^T
 \cr
\hline
\pp(\tfrac{2\pi m}{Mq}^+)^T
 \cr
\hline
\qq(\tfrac{2\pi m}{Mq}^+)^T
 \end{array}
\right)
 = \M_m\cdot
\left(
 \begin{array}{c}
\T(\tfrac{2\pi m}{Mq}^-)^T
 \cr
\hline
\pp(\tfrac{2\pi m}{Mq}^-)^T
 \cr
\hline
\qq(\tfrac{2\pi m}{Mq}^-)^T
 \end{array}
\right).
\end{equation*}

\noindent Notice that, when $\rho_m = 0$, $\M_m$ is the identity matrix $\I$. From \eqref{e:Mm}, it follows that $\rho$ is the angle between any two adjacent sides. Imposing that \eqref{e:psistpq} corresponds to a closed polygon, i.e., that
\begin{equation*}
\M_{Mq-1}\cdot \M_{Mq-2} \cdot \ldots \cdot \M_1 \cdot \M_0 = \I,
\end{equation*}

\noindent there is concluding evidence (see \cite{HozVega2014}) that $\rho$ is given by
\begin{equation*}
\cos(\rho) =
\begin{cases}
2\cos^{2/q}(\pi/M) - 1, & \mbox{if $q\equiv1\bmod2$},
 \cr
2\cos^{4/q}(\pi/M) - 1, & \mbox{if $q\equiv0\bmod2$};
\end{cases}
\end{equation*}

\noindent and the value of $\hat\psi(0, t_{pq})$ follows from \eqref{e:rhom}. In some cases, it will be preferable to work with $\cos(\rho/2)$, which results in a slightly simpler expression:
\begin{equation}
\label{e:cosrho}
\cos(\rho/2) =
\begin{cases}
\cos^{1/q}(\pi/M), & \mbox{if $q\equiv1\bmod2$},
 \cr
\cos^{2/q}(\pi/M), & \mbox{if $q\equiv0\bmod2$}.
\end{cases}
\end{equation}

\noindent The previous ideas suggest very strongly that $\psi(s, t)$ is also periodic in time, with period $2\pi / M^2$. Furthermore, bearing in mind the symmetries of the problem, it follows that $\T$ is also periodic in time, while $\X$ is periodic in time, up to a movement of its center of mass with constant upward velocity $c_M$. This last sentence seems to be true only for the $M$-corner problem, i.e., for regular polygons. In fact, in the case of nonregular polygons, it is natural to expect that time periodicity is lost.

Remark that \eqref{e:psistpq}, with $\hat\psi(0, t_{pq})=1$,  is the mathematical expression of the so-called Talbot effect in optics. This is a linear effect that, using Fresnel diffraction, can be described by means of the constant coefficient Schr\"odinger equation, i.e., \eqref{e:schr} without the nonlinear potential (see, for instance, \cite{Berry}). One of the consequences of this Talbot effect at the qualitative level is the so-called axis switching phenomena. In fact, it is easily seen from the values of the Gau{\ss} sums at half the period that, at that time, the same $M$-polygon reappears, but with the axis switched by an angle of $\pi/M$. This phenomenon has been observed and largely documented in the literature related to the evolution of noncircular jets (see, for example, the survey \cite{GG}). Moreover, it is also observed that, at other rational times, some more complicated structures in the shape of skew polygons appear in real fluids, when nozzles with the shape of equilateral triangles or squares are considered (see, for example,  \cite[Figure 6]{GG}, for the case of a triangle, and also \cite[Figure 10]{GGP} and \cite[p. 1492]{GGP}, where it is said ``[...]a consistent eightfold distribution pattern is also suggested[...]" for  nozzles with a squared shape).

The Talbot effect in nonlinear dispersive equations has been studied at the numerical level in \cite{CO1,CO2,Olver}, and experimentally in nonlinear optics in \cite{ZWZX}. Theoretical results on the nonlinear setting are obtained in \cite{ET1,ET2}. These are results at the subcritical level of regularity, which typically implies that the nonlinear potential is considered as an external perturbative force which is small with respect to the linear part of the equation. Hence, its contribution is obtained through Duhamel's integral, which has a smoothing effect. As a consequence, the complex behavior exhibited by the solutions is due to the linear term in Duhamel's expression. Therefore, it is very natural to try to find nonlinear Talbot effects whose complexity is not only a consequence of a linear behavior. In our case, a first hint of this is precisely the value of $\hat\psi(k, t_{pq})$, which does not remain constant at all times, as in the linear setting. Observe that, from \eqref{e:hatpsistpq}, we have that
\begin{equation}
\label{e:supnorm}|\hat\psi(k, t_{pq})|=|\hat\psi(0, t_{pq})|.
\end{equation}

\noindent Other nonlinear results have been recently obtained in \cite{BanicaVega2016}, in the case of a filament with one corner and small perturbations of it. In particular, some transfer of energy, measured in an appropriate norm suggested by \eqref{e:supnorm}, and the lack of conservation of the linear momentum are proved.  One of the important consequences obtained in this paper is that, in Sections \ref{s:energy} and \ref{s:momentum}, we give very strong numerical evidence that these results are also true in the case of the $M$-corner problem. 

Along this paper, we denote as $\X_M$ and $\T_M$ the exact solution of the $M$-corner problem; as $\X_{alg}$ and $\T_{alg}$, the algebraically constructed solution of the $M$-corner problem (where $\X_{alg}$ is constructed in such a way that $\mean(\X_{alg}) = (0, 0, 0)^T$); and, as $\X_{num}$ and $\T_{num}$, the numerical solution of the $M$-corner problem, obtained by means of a fourth-order Runge-Kutta scheme applied to \eqref{e:xt}-\eqref{e:schmap}. The details can be found in \cite{HozVega2014}.

The structure and the content of this paper is as follows. In Section \ref{s:evidence}, we offer concluding numerical evidence that the $M$-corner problem can be explained as a superposition of $M$ one-corner problems at $t = 0^+$. This has heavy implications, as, for instance, the recovery of \eqref{e:A1} in a completely novel way. Furthermore, in Section \ref{s:analytical}, we use the relationship between the one-problem and the $M$-corner problems to determine the velocity of the center of mass $c_M$:
\begin{equation}
\label{e:cM}
c_M = \frac{-2\ln(\cos(\pi/M))}{(\pi/M)\tan(\pi/M)} = \frac{\ln(1 + \tan^2(\pi/M))}{(\pi/M)\tan(\pi/M)}.
\end{equation}

\noindent This is done by reducing the calculation of $c_M$ to the computation of an integral (see \eqref{e:cMlimit2}) that appears in the one-corner problem in a natural way. We also integrate numerically \eqref{e:cMlimit2} for a large set of $M$, and the results fully agree with \eqref{e:cM}. Besides, we approximate $c_M$ directly from the numerical simulation of \eqref{e:xt}-\eqref{e:schmap}, and compare it with \eqref{e:cM}, obtaining coherent results. On the other hand, in Section \ref{s:algebraic}, under some hypotheses, we obtain \eqref{e:cM} by algebraic means, using an approach completely unrelated to that in Section \ref{s:analytical}. Therefore, there is in our opinion concluding evidence (analytical, algebraic and numerical) that \eqref{e:cM} is correct.

In Section \ref{s:energy}, we study numerically the transfer of energy for $M = 3$, measured in the norm $\|\widehat{\T_{M,s}}(t)\|_{\infty} = \max_k|k\, \widehat{\T_{M}}(k, t)|$.  The lack of continuity proved in \cite{BanicaVega2016} is clearly seen at all the rational times (see Figure \ref{f:maxhatTsq200006}), due to the fact that the creation of corners happens at those times. More interestingly, the jumps do not seem to be bounded. In fact, looking carefully at the symmetries of the problem, we rigorously simplify the expression of the tangent vector and basically reduce it to computing two discrete Fourier transforms of $q$ or $q/2$ elements, with $q$ the denominator of the rational time (see \eqref{e:TsMk1} and \eqref{e:TsMk}). The final conclusion is the logarithmic fitting found in \eqref{logg}, which gives a very strong numerical evidence that these jumps are indeed unbounded. We consider this latter fact a far reaching one. As far as we know, there is no theoretical result in this direction for periodic solutions of \eqref{e:schmap}.

In Section \ref{s:momentum}, we study numerically the transfer of linear momentum, which exhibits an intermittent behavior (see Figure \ref{f:momentum}), very reminiscent of the so-called Riemman's nondifferentiable function \cite{Ja}:
\begin{equation}
\label{e:riemann1}
\phi(x) = \sum_{n = 1}^\infty\frac{\sin(\pi n^2 x)}{n^2}.
\end{equation}

\noindent In particular, in Figure \ref{f:fingerprintriemann}, we approximate the Fourier coefficients of the second component of the momentum, for $M = 3$, and show that the leading terms are multiples of perfect squares, in analogy with \eqref{e:riemann1}. 

In Section \ref{s:observations}, we study briefly the case of a nonregular polygon from a numerical point of view, and, based on the results, we conjecture how \eqref{e:cosrho} is to be generalized. Moreover, as in the regular case, the corners do not see each other at infinitesimal times; but, unlike in the regular case, the periodicity in time seems to be lost.

Finally, in Section \ref{s:Conclusions}, we draw the main conclusions.

\section{Numerical relationship between the $M$-corner problem and the one-corner problem}

\label{s:evidence}

We claim that the $M$-corner problem can be understood as a superposition of $M$ one-corner problems at $t = t_{1,q}$, $q\gg1$, or, in other words, that at infinitesimal times, the corners ``do not see'' one another. In order to compare both cases, we integrate \eqref{e:1cornerproblemsystem} at $t = t_{1,q}$, $q\gg1$, and in such a way that its orientation is in agreement with the $M$-corner problem. Therefore, since the inner angle between two adjacent sides of an $M$-sided regular polygon is
\begin{equation}
\label{e:theta}
\theta = \pi - 2\pi/M,
\end{equation}

\noindent it follows from \eqref{e:A1} and \eqref{e:thetaA1} that, for a given $M$, we have to choose
\begin{equation}
\label{e:c_0}
c_0 = \left[-\frac{2}{\pi}\ln\left(\cos\left(\frac{\pi}{M}\right)\right)\right]^{1/2}.
\end{equation}

\noindent Moreover, bearing in mind \eqref{e:A+A-}, i.e.,
\begin{equation*}
\lim_{s\to-\infty}\T_{c_0}(s) = \A^- = (A_1, -A_2, -A_3)^T, \qquad \lim_{s\to\infty}\T_{c_0}(s) = \A^+ = (A_1, A_2, A_3)^T,
\end{equation*}

\noindent we have to rotate $\X_{c_0}$, $\T_{c_0}$, $\nn_{c_0}$ and $\bb_{c_0}$ by means of a rotation matrix $\M$, in such a way that $\X_{rot}\equiv\M\cdot\X_{c_0}$, $\T_{rot}\equiv\M\cdot\T_{c_0}$, etc., where the subscript $rot$ indicates a one-corner problem solution rotated in order to match the $M$-corner problem. The matrix $\M$ is determined by imposing
\begin{equation*}
\lim_{s\to-\infty}\T_{rot}(s) = (\cos(2\pi/M),-\sin(2\pi/M),0)^T, \qquad \lim_{s\to\infty}\T_{rot}(s) = (1, 0, 0)^T,
\end{equation*}

\noindent i.e., $\lim_{s\to\pm\infty}\T_{rot}(s)$ takes the values of the tangent vector of an $M$-sided regular polygon at $s = 0^\pm$, $t = 0$. This can be achieved by defining
\begin{equation}
\label{e:rotationM}
\M =
\begin{pmatrix}
\cos(\tfrac{\pi}{M}) & \sin(\tfrac{\pi}{M}) & 0
 \\
-\sin(\tfrac{\pi}{M}) & \cos(\tfrac{\pi}{M}) & 0
 \\
0 & 0 & 1
\end{pmatrix}
\cdot
\begin{pmatrix}
1 & 0 & 0
 \\
0 & \frac{A_2}{\sqrt{A_2^2+A_3^2}} & \frac{A_3}{\sqrt{A_2^2+A_3^2}}
 \\
0 & \frac{-A_3}{\sqrt{A_2^2+A_3^2}} & \frac{A_2}{\sqrt{A_2^2+A_3^2}}
\end{pmatrix}.
\end{equation}

\noindent Hence,
\begin{equation}
\label{e:rotatedXT}
\begin{pmatrix}
X_{rot,1}
 \\
X_{rot,2}
 \\
X_{rot,3}
\end{pmatrix}
=
\begin{pmatrix}
-\pi/M
 \\
-\frac{\pi/M}{\tan(\pi/M)}
 \\
0
\end{pmatrix}
+
\M
\cdot
\begin{pmatrix}
X_{c_0, 1}
 \\
X_{c_0, 2}
 \\
X_{c_0, 3}
\end{pmatrix},
 \qquad
\begin{pmatrix}
T_{rot,1}
 \\
T_{rot,2}
 \\
T_{rot,3}
\end{pmatrix}
=
\M
\cdot
\begin{pmatrix}
T_{c_0, 1}
 \\
T_{c_0, 2}
 \\
T_{c_0, 3}
\end{pmatrix},
\end{equation}

\noindent etc., where $(-\pi/M, -\pi/(M\tan(\pi/M)), 0)^T$ is the corner of an $M$-sided regular polygon at $s = 0$, $t = 0$.

In our numerical experiments, we have distinguished three cases, according to whether $q\equiv1\bmod2$, $q\equiv0\bmod4$, or $q\equiv2\bmod4$. On the one hand, we have computed $\T_{alg}(s, t_{1,q})$ at the values $s=s_j$ corresponding to the middle points of the sides of the skew polygon, where $\T_{alg}$ is continuous: $s_j = \pi (2j-1)/(Mq)$, $j = -(q - 1) / 2, \ldots, (q + 1) / 2$, if $q\equiv1\bmod2$; $s_j = 2\pi(2j - 1) / (Mq)$, $j = -q/4+1, \ldots, q/4$, if $q\equiv0\bmod4$; and $s_j = 4\pi j/(Mq)$, $j = -(q - 2) / 4, \ldots, (q - 2) / 4$, if $q\equiv2\bmod4$. On the other hand, we have approximated numerically the corresponding $\T_{rot}(s, t_{1,q})$, at those same $s = s_j$, by integrating numerically \eqref{e:1cornerproblemsystem}-\eqref{e:1cornerproblemsystem0} at $t = t_{1,q}$ by means of a fourth-order Runge-Kutta with $\Delta s = \pi/(M^2q)$, and rotating the resulting $\T_{c_0}$ according to \eqref{e:rotatedXT}.
\begin{table}[htb!]
\centering
\begin{tabular}{|c|c||c|c||c|c|}
\hline $q$ & $\max|\T_{alg} - \T_{rot}|$ & $q$ & $\max|\T_{alg} - \T_{rot}|$ & $q$ & $\max|\T_{alg} - \T_{rot}|$
\\
\hline $1001$ & $2.2653\cdot10^{-2}$ & $1000$ & $1.0636\cdot10^{-2}$ & $1002$ & $1.0625\cdot10^{-2}$
\\
\hline $2001$ & $1.6034\cdot10^{-2}$ & $2000$ & $7.5322\cdot10^{-3}$ & $2002$ & $7.5285\cdot10^{-3}$
\\
\hline $4001$ & $1.1343\cdot10^{-2}$ & $4000$ & $5.3287\cdot10^{-3}$ & $4002$ & $5.3273\cdot10^{-3}$
\\
\hline $8001$ & $8.0209\cdot10^{-3}$ & $8000$ & $3.7669\cdot10^{-3}$ & $8002$ & $3.7664\cdot10^{-3}$
\\
\hline $16001$ & $5.6700\cdot10^{-3}$ & $16000$ & $2.6606\cdot10^{-3}$ & $16002$ & $2.6604\cdot10^{-3}$
\\
\hline $32001$ & $4.0068\cdot10^{-3}$ & $32000$ & $1.8773\cdot10^{-3}$ & $32002$ & $1.8772\cdot10^{-3}$
\\
\hline $64001$ & $2.8307\cdot10^{-3}$ & $64000$ & $1.3236\cdot10^{-3}$ & $64002$ & $1.3235\cdot10^{-3}$
\\
\hline $128001$ & $2.0019\cdot10^{-3}$ & $128000$ & $9.3638\cdot10^{-4}$ & $128002$ & $9.3638\cdot10^{-4}$
\\
\hline
\end{tabular}
\caption{Maximum of $|\T_{alg}(s_j, t_{1,q}) - \T_{rot}(s_j, t_{1,q})|$, for $M = 5$. First case: $q\equiv1\bmod2$; $s_j = \pi (2j-1)/(Mq)$, $j = -(q - 1) / 2, \ldots, (q + 1) / 2$. Second case: $q\equiv0\bmod4$; $s_j = 2\pi(2j - 1) / (Mq)$, $j = -q/4+1, \ldots, q/4$. Third case: $q\equiv2\bmod4$; $s_j = 4\pi j/(Mq)$, $j = -(q - 2) / 4, \ldots, (q - 2) / 4$. In the three cases, the maximum Euclidean distance between $\T_{alg}$ and $\T_{rot}$ clearly decreases as $\mathcal O(1 / \sqrt q) = \mathcal O(\sqrt{t_{1,q}})$, so there is convergence between both approaches (see also Figure \ref{f:comparisonnorm}).}\label{t:XalgVsXrot}
\end{table}
	
In Table \ref{t:XalgVsXrot}, we give $\max_j|\T_{alg}(s_j, t_{1,q}) - \T_{rot}(s_j, t_{1,q})|$, for $M = 5$, and a number of different $q$. Here, $c_0 = 0.3673\ldots$. Observe that, in any of the three cases, when $q$ is (approximately) doubled, the maximum Euclidean distance between $\T_{alg}$ and $\T_{rot}$ is divided by approximately the square root of two, i.e., it is of the order of $\mathcal O(1 / \sqrt q) = \mathcal O(\sqrt{t_{1,q}})$. In Figure \ref{f:comparisonnorm}, using those same values of $q$, we plot $|\T_{alg}(s_j, t_{1,q}) - \T_{rot}(s_j, t_{1,q})|$ as a function of $s_j$; again, the agreement between $\T_{alg}$ and $\T_{rot}$ clearly improves, as $q$ increases. Furthermore, when $q$ is even, the best agreement happens at the smallest $s=|s_j|$ (where it is extremely high), and decreases monotonically as $s=|s_j|$ grows up. However, when, $q$ is odd, the plot of $|\T_{alg}(s_j, t_{1,q}) - \T_{rot}(s_j, t_{1,q})|$ seems to yield two curves that intersect near $s = 0$. An immediate explanation to this apparently strange behavior is given in Figure \ref{f:Tq1001q4001}. On the left-hand side, we have plotted $\{\T_{alg}(s_j, t_{1, 1001})\}$. Remark that $\{\T_{alg}(s_j, t_{1,q})\}$ is not an actual curve, but a collection of $Mq = 5\times1001$ points; on the one hand, if we plot them together with their joining segments, we get an annoying saw teeth effect; on the other hand, if we join each point with the second next one, we get two smooth curves between which the saw teeth are sandwiched. As $q$ odd grows, the two smooth curves that contain the saw teeth become more and more close, until they converge into a single one. On the right-hand side, we have plotted $\T_{alg}(s_j, t_{1, 4001})$; since $q$ is approximately four times as large, the size of the teeth is approximately one half.

\begin{figure}[!htb]
\centering
\includegraphics[width=0.325\textwidth, clip=true]{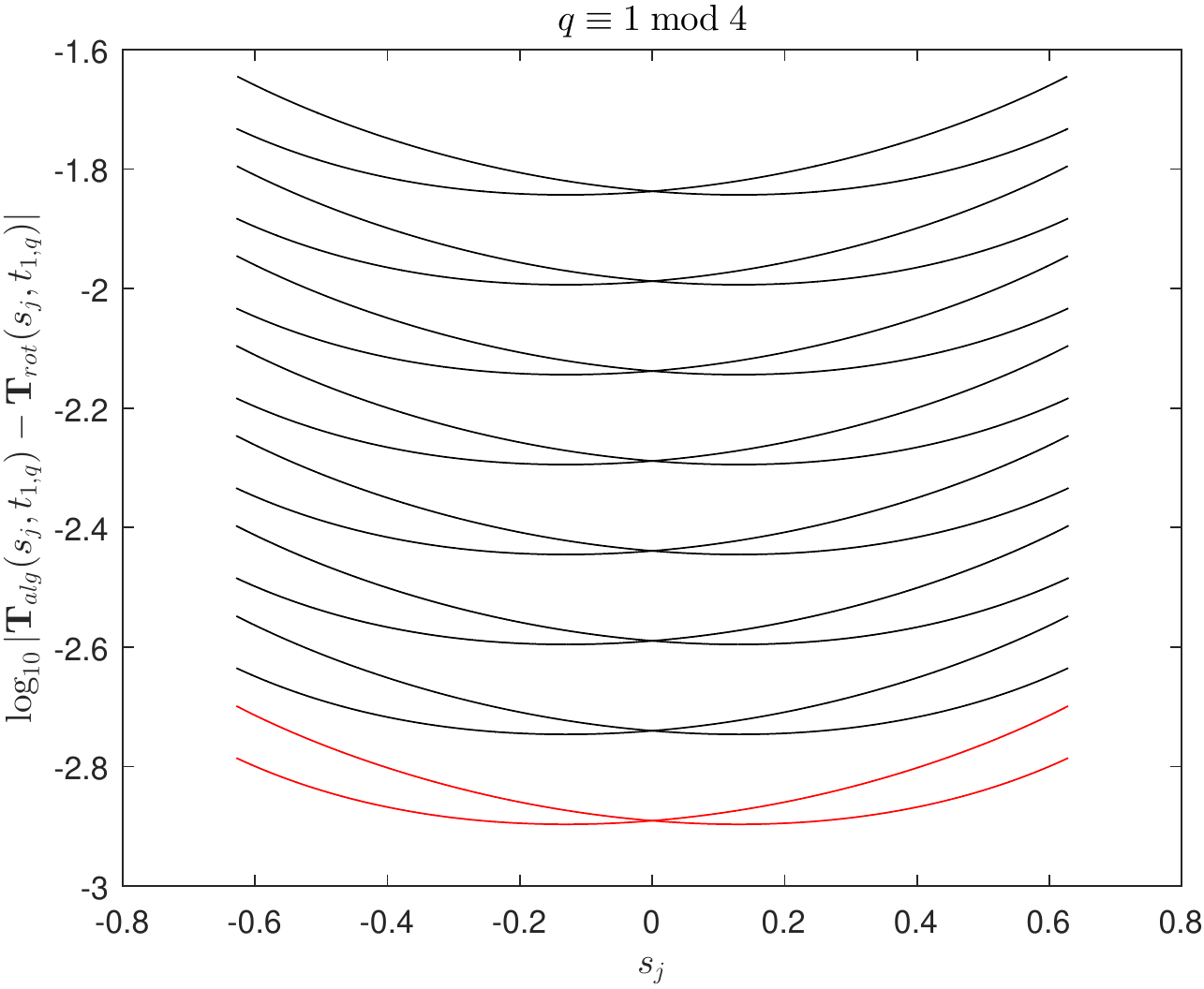}
\includegraphics[width=0.325\textwidth, clip=true]{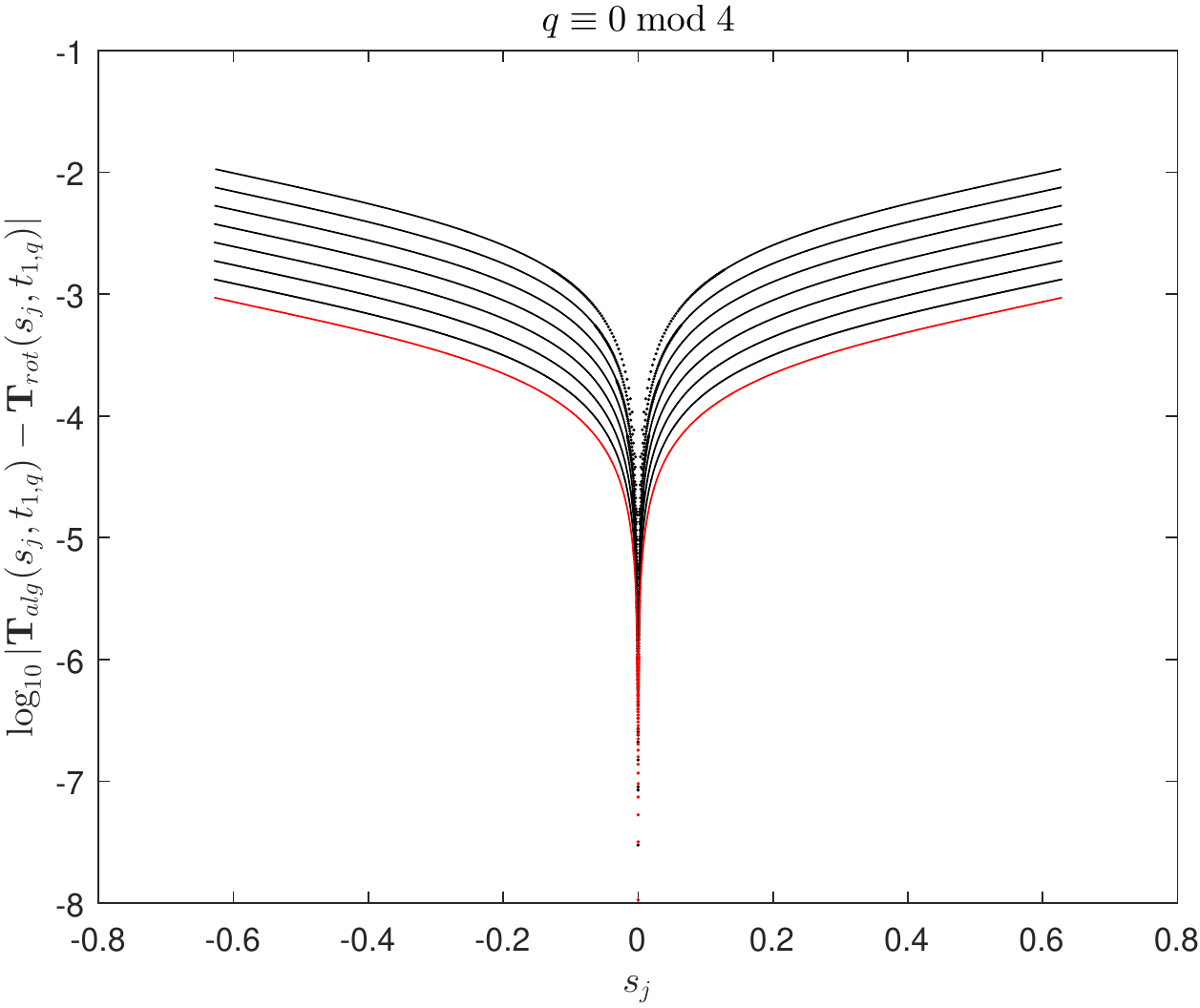}
\includegraphics[width=0.325\textwidth, clip=true]{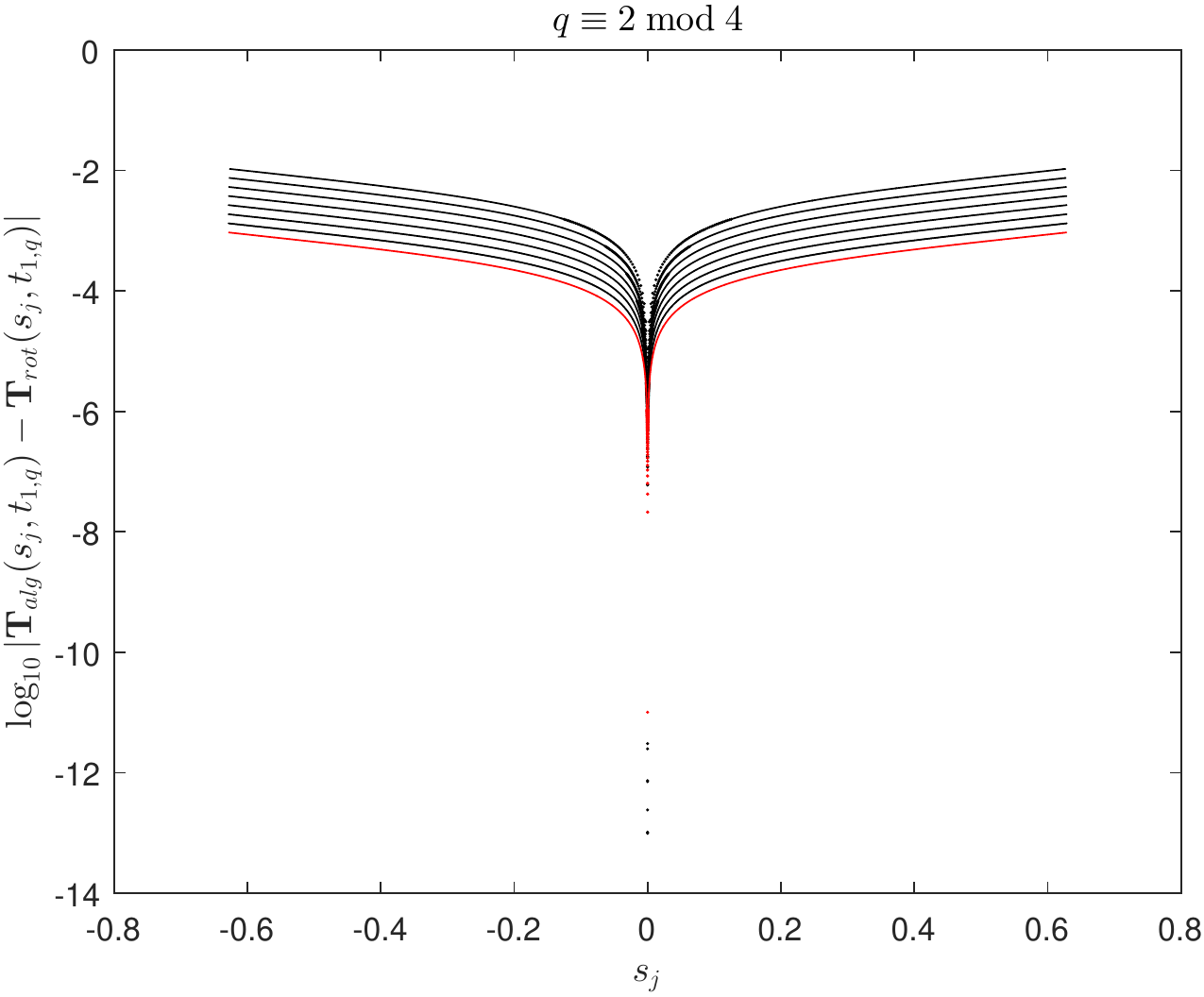}
\caption{Plots of $\log_{10}|\T_{alg}(s_j, t_{1,q}) - \T_{rot}(s_j, t_{1,q})|$, for the values of $q$ considered in Table \ref{t:XalgVsXrot}. In general, the agreement improves as $q$ grows up, so the best results, in red, correspond to $q = 128001$, $q = 128000$ and $q = 128002$, respectively.}\label{f:comparisonnorm}
\end{figure}

 \begin{figure}[!htb]
\centering
\includegraphics[width=0.49\textwidth, clip=true]{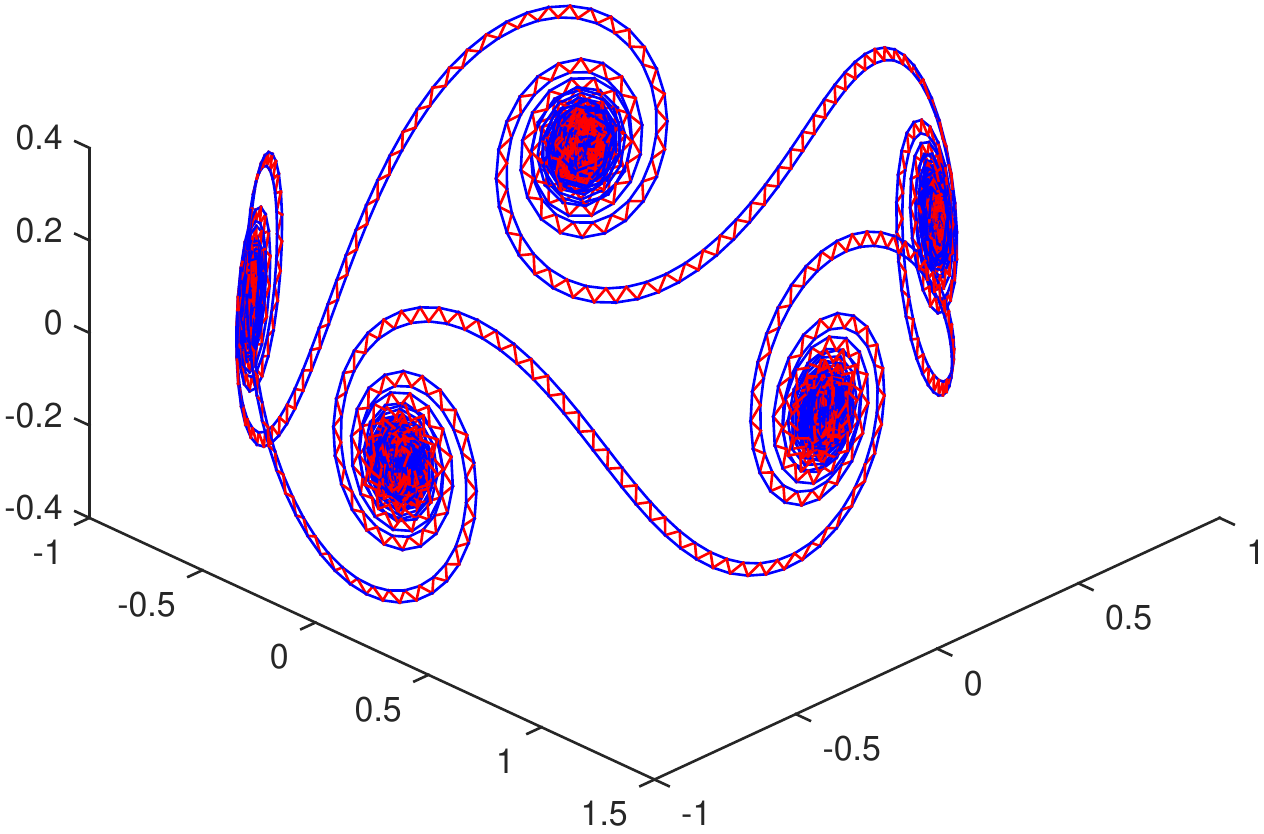}\includegraphics[width=0.49\textwidth, clip=true]{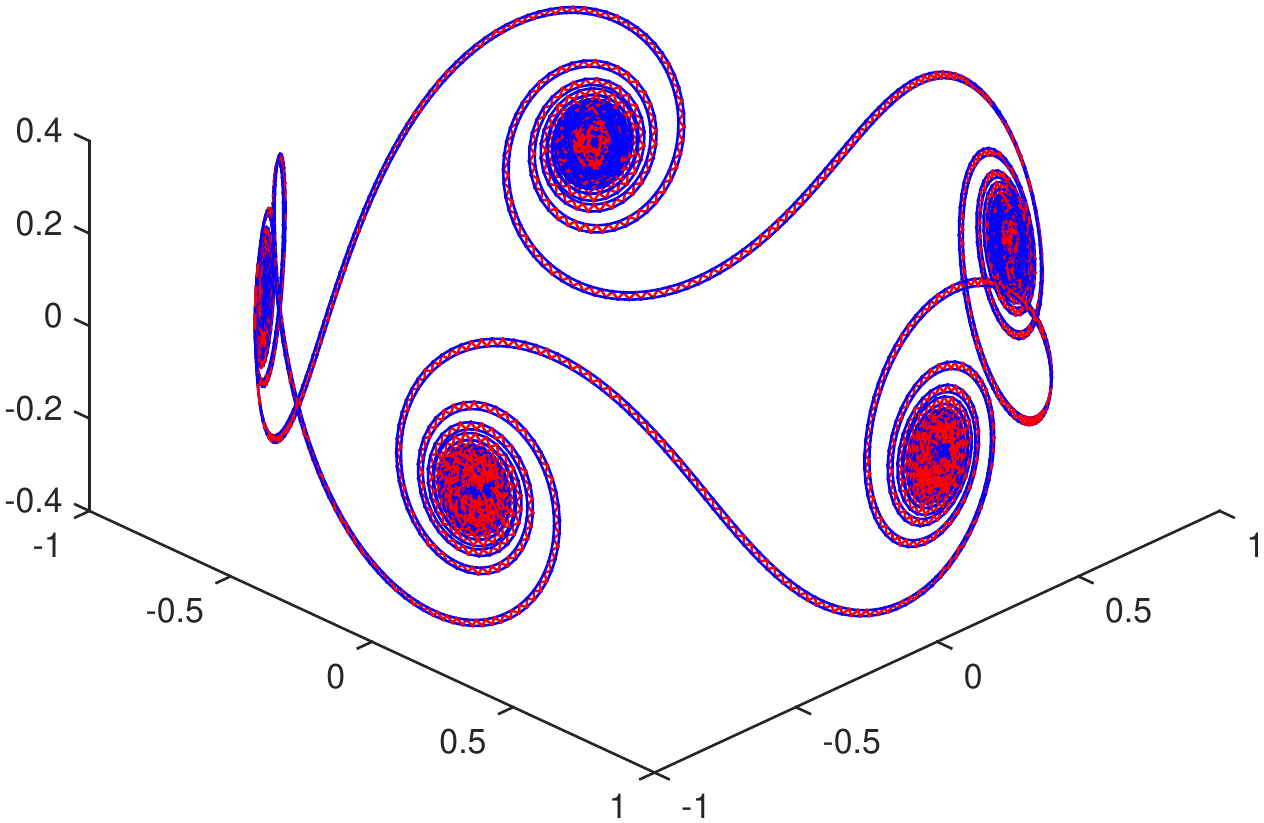}
\caption{Left: $\T_{alg}$, for $M = 5$, at $t = t_{1,1001}$. Right: $\T_{alg}$, for $M = 5$, at $t = t_{1,4001}$. We have sandwiched the teeth into two smooth curves. Moreover, the teeth corresponding to $t = t_{1,4001}$ are much less pronounced than those corresponding to $t = t_{1,1001}$.}\label{f:Tq1001q4001}
\end{figure}

\begin{figure}[!htb]
\centering
\includegraphics[width=0.5\textwidth, clip=true]{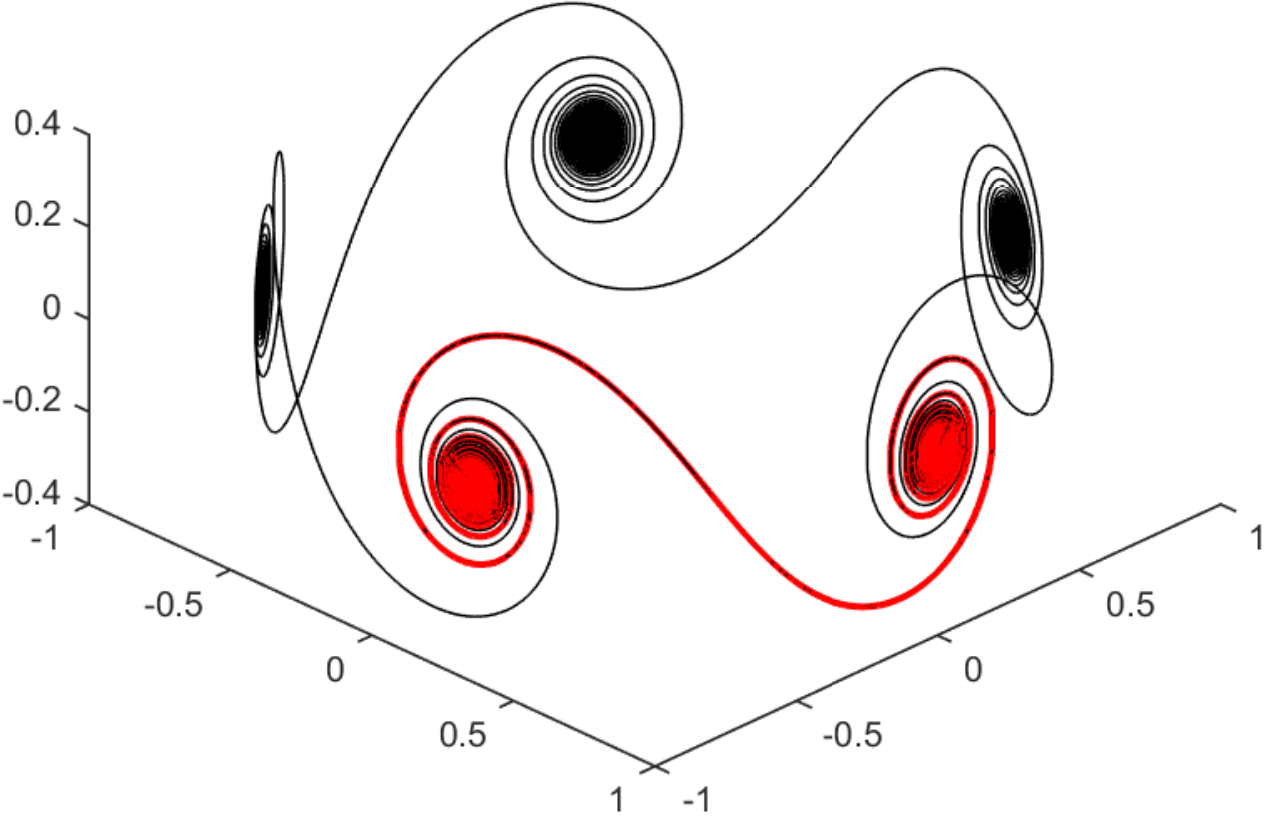}
\caption{$\T_{alg}$ (black), against $\T_{rot}$ (thick red), for $M = 5$, at $t = t_{1,128002}$. Except for the thicker stroke, the red curve is visually undistinguishable from the black one.}\label{f:comparison}
\end{figure}

Nonetheless, when $q$ is even, plotting $\{\T_{alg}(s_j, t_{1,q})\}$, together with their joining segments, yields an apparently very regular curve. For instance, in Figure \ref{f:comparison}, we have shown simultaneously $\T_{alg}$ and $\T_{rot}$, for $M = 5$, at $t = t_{1,128002}$; except for the thicker stroke, $\T_{rot}$ is visually undistinguishable from $\T_{alg}$. Furthermore, when $q$ is even, it is straightforward to give a good approximation of the curvature at the origin. Since $c_0 = \sqrt{t_{1,q}}|\T_s(0, t_{1, q})|$, we have to approximate $\T_s(0, t_{1,q})$, which is done by means of a finite difference. Let us consider, without loss of generality, $q\equiv2\bmod4$, because $\T_{alg}(s, t_{1,q})$ is continuous at $s = 0$ and at $s = \pm\Delta s$, where $\Delta s = 2\pi/(Mq/2) = 4\pi/(Mq)$. In Table \ref{t:c0origin}, we give $c_0 - \sqrt{t_{1,q}}|\T_{alg}(\Delta s, t_{1,q}) - \T_{alg}(-\Delta s, t_{1,q})| / (2\Delta s)$ for different values of $q$; from that table, the error in the approximation of $c_0$ clearly decreases as $\mathcal O(1/q) = \mathcal O(t_{1,q})$. Therefore, we are going to be able to recover analytically $c_0$ as
\begin{equation}
\label{e:c0limit}
c_0 = \lim_{\substack{q\to\infty \\ q\equiv2\bmod4}}\sqrt{t_{1,q}}\frac{|\T_{alg}(\tfrac{4\pi}{Mq}, t_{1,q}) - \T_{alg}(-\tfrac{4\pi}{Mq}, t_{1,q})|}{2\tfrac{4\pi}{Mq}}.
\end{equation}

\begin{table}[htb!]
\centering
\begin{tabular}{|c|c|c|}
\hline $\vphantom{\Bigg(}q\vphantom{\Bigg)}$ & $c_0 - \operatorname{approx}(c_0)$
\\
\hline $1002$ & $2.3300\cdot10^{-4}$
\\
\hline $2002$ & $1.1663\cdot10^{-4}$
\\
\hline $4002$ & $5.8352\cdot10^{-5}$
\\
\hline $8002$ & $2.9184\cdot10^{-5}$
\\
\hline $16002$ & $1.4594\cdot10^{-5}$
\\
\hline $32002$ & $7.2979\cdot10^{-6}$
\\
\hline $64002$ & $3.6490\cdot10^{-6}$
\\
\hline $128002$ & $1.8259\cdot10^{-6}$
\\
\hline
\end{tabular}
\caption{$c_0 - \sqrt{t_{1,q}}|\T_{alg}(\Delta s, t_{1,q}) - \T_{alg}(-\Delta s, t_{1,q})| / (2\Delta s)$. The error in the approximation of $c_0$ clearly decreases as $\mathcal O(1 / q) = \mathcal O(t_{1, q})$.} \label{t:c0origin}
\end{table}

\noindent Let us recall from \eqref{e:Mm} that
\begin{equation*}
\begin{split}
\left(
 \begin{array}{c}
\Talg(0, t_{1, q})^T
 \\
\hline
 \\[-1em]
\ppalg(0, t_{1, q})^T
 \\
\hline
 \\[-1em]
\qqalg(0, t_{1, q})^T
 \end{array}
\right)
 & = \M_{q-1}\cdot
\left(
 \begin{array}{c}
\Talg(-\frac{4\pi}{Mq}, t_{1, q})^T
 \\
\hline
 \\[-1em]
\ppalg(-\frac{4\pi}{Mq}, t_{1, q})^T
 \\
\hline
 \\[-1em]
\qqalg(-\frac{4\pi}{Mq}, t_{1, q})^T
 \end{array}
\right),
 \\
\left(
 \begin{array}{c}
\Talg(\frac{4\pi}{Mq}, t_{1, q})^T
 \\
\hline
 \\[-1em]
\ppalg(\frac{4\pi}{Mq}, t_{1, q})^T
 \\
\hline
 \\[-1em]
\qqalg(\frac{4\pi}{Mq}, t_{1, q})^T
 \end{array}
\right)
 & = \M_1\cdot
\left(
 \begin{array}{c}
\Talg(0, t_{1, q})^T
 \\
\hline
 \\[-1em]
\ppalg(0, t_{1, q})^T
 \\
\hline
 \\[-1em]
\qqalg(0, t_{1, q})^T
 \end{array}
\right).
\end{split}
\end{equation*}

\noindent Moreover, since we are interested only in the Euclidean norm of $\T_{alg}(4\pi/(Mq), t_{1,q}) - \T_{alg}(-4\pi/(Mq), t_{1,q})$, we can safely ignore the global rotation of $\T_{alg}$ and, thus, assume without loss of generality that $\Talg(0, t_{1, q})$, $\ppalg(0, t_{1, q})$ and $\qqalg(0, t_{1, q})$ form the identity matrix. Hence, up to a rotation,
\begin{equation*}
\Talg\left(-\frac{4\pi}{Mq}, t_{1, q}\right) =
\begin{pmatrix}
\cos(\rho) \cr -\sin(\rho)\cos(\theta_{q-1}) \cr -\sin(\rho)\sin(\theta_{q-1})
\end{pmatrix},
 \quad
\Talg(0, t_{1, q}) =
\begin{pmatrix}
1 \cr 0 \cr 0
\end{pmatrix},
 \quad
\Talg\left(\frac{4\pi}{Mq}\right) =
\begin{pmatrix}
\cos(\rho) \cr \sin(\rho)\cos(\theta_1) \cr \sin(\rho)\sin(\theta_1)
\end{pmatrix},
\end{equation*}

\noindent so
\begin{align}
\label{e:Talgnorm}
\left|\T_{alg}\left(\frac{4\pi}{Mq}, t_{1,q}\right) - \T_{alg}\left(-\frac{4\pi}{Mq}, t_{1,q}\right)\right| & =
\left|\begin{pmatrix}
0
 \cr
\sin(\rho)(\cos(\theta_1) + \cos(\theta_{q-1})).
 \cr
\sin(\rho)(\sin(\theta_1) + \sin(\theta_{q-1}))
\end{pmatrix}\right|
 \cr
& = 2\sin(\rho)\cos\left(\frac{\theta_{1} - \theta_{q-1}}{2}\right)
 \cr
& = 2\sin(\rho),
\end{align}

\noindent because, $G(-p, q-m, q) = G(-p, -m, q) = G(-p, m, q)$, and, in particular, $G(-p, q-1, q) = G(-p, 1, q)$, which implies $\theta_1 = \theta_{q-1}$ (see, for instance, \cite[Appendix A]{HozVega2014}). Substituting \eqref{e:Talgnorm} into \eqref{e:c0limit}, and bearing in mind \eqref{e:cosrho}, we recover \eqref{e:c_0}:
$$
c_0 = \lim_{q\to\infty}\sqrt{\frac{2\pi}{M^2q}}\frac{Mq}{8\pi}2\sin(\rho) = \lim_{q\to\infty}\sqrt{\frac{q}{8\pi}\left(1 - \left(2\cos^{4/q}\left(\frac{\pi}{M}\right) - 1\right)^2\right)} = \left[-\frac{2}{\pi}\ln\left(\cos\left(\frac{\pi}{M}\right)\right)\right]^{1/2}.
$$

\noindent From this last expression, bearing in mind \eqref{e:thetaA1} and \eqref{e:theta}, we recover \eqref{e:A1}:
$$
\exp(-(c_0^2/2)\pi) = \cos\left(\frac{\pi}{M}\right) = \cos\left(\frac{\pi-\theta}{2}\right) = \sqrt{\frac{1-\cos(\theta)}{2}} = A_1.
$$

\noindent We have also considered the evolution of $\X$ at $s = 0$. From \eqref{e:1cornerproblemsystem0}, \eqref{e:rotationM} and \eqref{e:rotatedXT}:
\begin{equation*}
\begin{pmatrix}
X_{rot,1}(0, t)
 \\
X_{rot,2}(0, t)
 \\
X_{rot,3}(0, t)
\end{pmatrix}
 =
\begin{pmatrix}
-\pi/M
 \\
-\frac{\pi/M}{\tan(\pi/M)}
 \\
0
\end{pmatrix}
+
2c_0\sqrt t
\begin{pmatrix}
\frac{A_3\sin(\pi/M)}{\sqrt{A_2^2+A_3^2}}
 \\[1em]
\frac{A_3\cos(\pi/M)}{\sqrt{A_2^2+A_3^2}}
 \\[1em]
\frac{A_2}{\sqrt{A_2^2+A_3^2}}
\end{pmatrix}.
\end{equation*}

\noindent As with $\X_{num}(0,t)$, $\X_{rot}(0, t)$ is a curve living in a plane containing the origin of $\mathbb R^3$, and parallel to the vectors $(\sin(\pi/M), \cos(\pi/M), 0)^T$ and $(0, 0, 1)$. On the left-hand side of Figure \ref{f:comparisonzh}, we plot both $\X_{num}(0,t)$ and $\X_{rot}(0, t)$ rotated clockwise $\pi/2 - \pi/M$ degrees around the $z$-axis; for small times, the movement of $\X_{num}(0, t)$ can be approximated quite well by means of a straight line with slope $A_2/A_3$. On the right-hand side of Figure \ref{f:comparisonzh}, we plot $X_{num,3}(0,t)$, $X_{rot,3}(0, t)$ as functions of $t$; for small times, $X_{num,3}(0,t)$ grows quite approximately like $(2c_0A_2/\sqrt{A_2^2+A_3^2})\sqrt t$.
\begin{figure}[!htb]
\centering
\includegraphics[width=0.49\textwidth, clip=true]{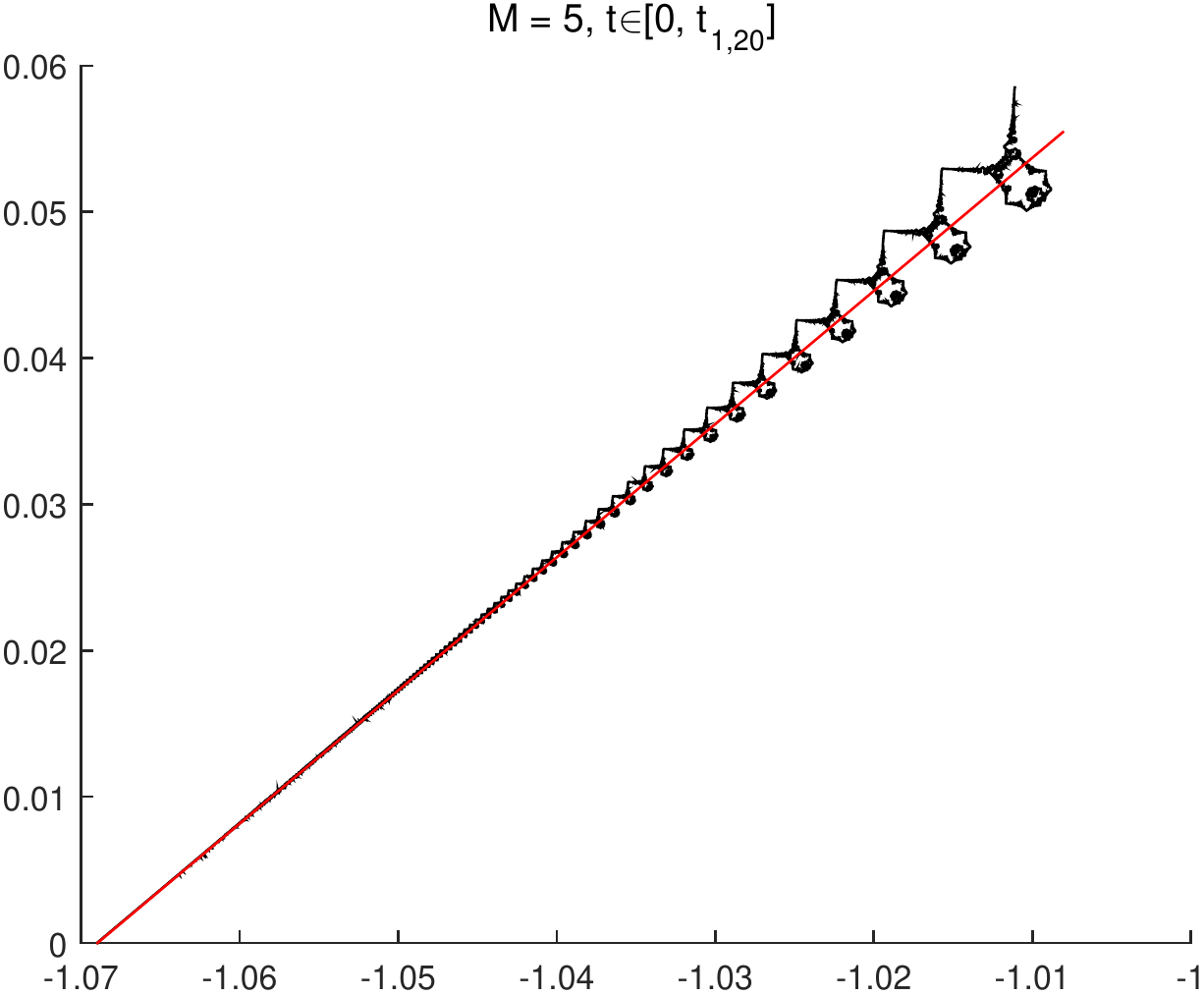}\includegraphics[width=0.49\textwidth, clip=true]{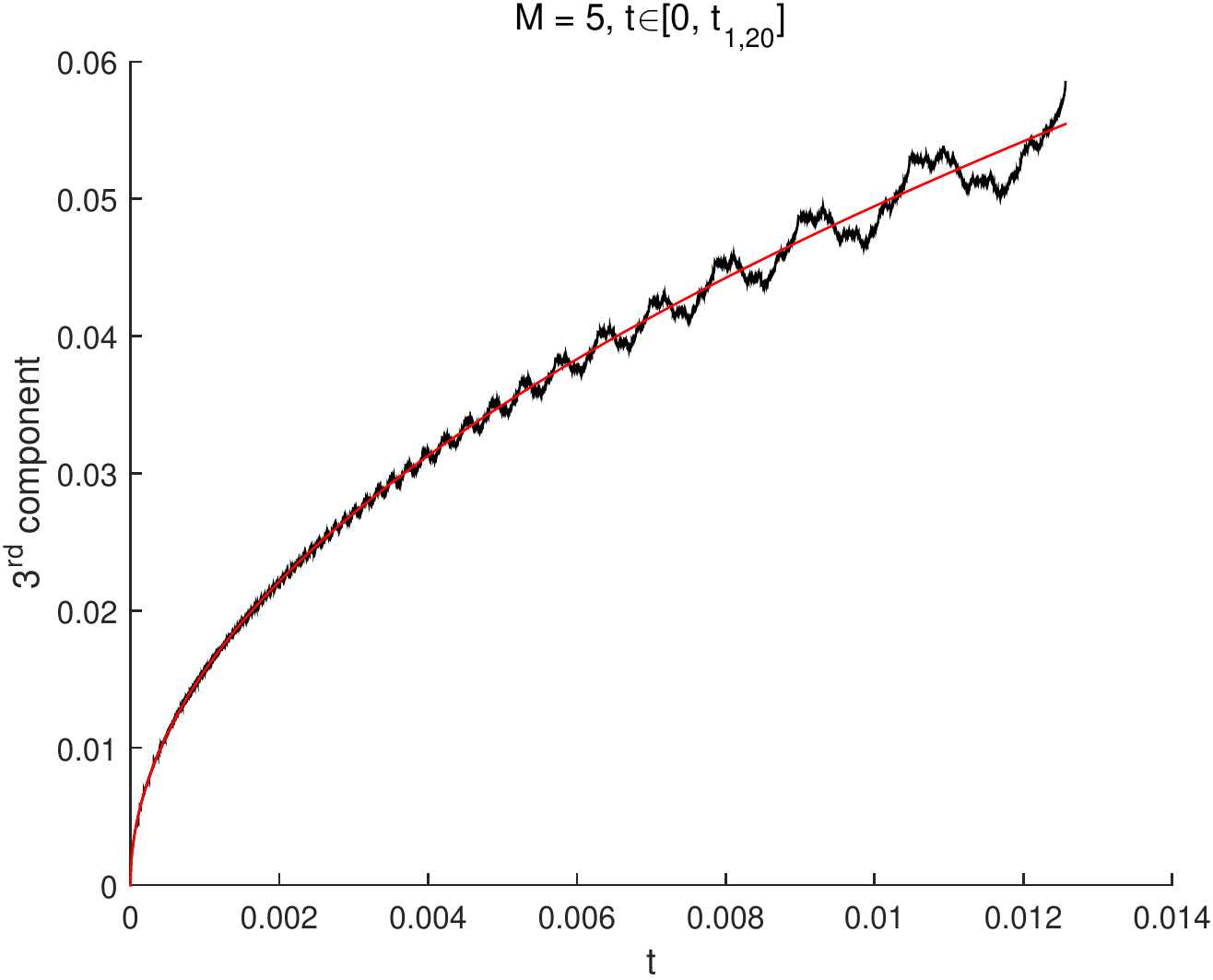}
\caption{Left: Evolution of $\X_{num}(0,t)$ (black) and $\X_{rot}(0, t)$ (red). Right: Evolution of $X_{num,3}(0,t)$ (black) and $X_{rot,3}(0, t)$ (red), as functions of $t$.}\label{f:comparisonzh}
\end{figure}

Let us finish this section by saying that, even if we have considered only the case $M = 5$, it is straightforward to check that everything holds for any $M\ge 3$. Therefore, in our opinion, there is concluding evidence that the $M$-corner problem can be indeed explained, at small times, as a superposition of $M$ one-corner problems. Moreover, assuming this fact, we will be able to compute $c_M$ in Section \ref{s:analytical}.

\section{Analytical computation of $c_M$ using the one-corner problem}

\label{s:analytical}

\subsection{Formulation of the problem}

In \cite{HozVega2014}, we gave concluding numerical evidence that the center of mass of $\XM$ moves upward with constant velocity $c_M$. Since $\XM$ is parameterized by arc-length, the center of mass is given by the mean of $\XM$ over a period:
\begin{equation}
\label{e:mean}
\mean(\XM)(t)\equiv\frac{1}{2\pi}\int_0^{2\pi}\XM(s,t)ds = (0, 0, c_Mt)^T;
\end{equation}

\noindent therefore, as in \cite{HozVega2014}, we write
\begin{equation}
\label{e:ht}
h(t) \equiv \mean(X_{M,3})(t) = c_M t.
\end{equation}

\noindent Moreover, due to the symmetries of the $M$-corner problem, it is enough to calculate the mean of $X_{M,3}$ over $s\in[-\pi/M, \pi/M]$:
\begin{equation*}
c_M t = \frac{M}{2\pi}\int_{-\pi/M}^{\pi/M} X_{M,3}(s,t)ds, \quad \forall t;
\end{equation*}

\noindent so, at any $t > 0$,
\begin{equation*}
c_M = \frac{M}{2\pi}\frac{1}{t}\int_{-\pi/M}^{\pi/M} X_{M,3}(s,t)ds,
\end{equation*}

\noindent and this formula holds in particular as $t\to0$. Now, bearing in mind the evidence given in Section \ref{s:evidence}, according to which the $M$-corner problem can at infinitesimal times be explained by the one-corner problem, we claim that
$$
c_M  = \lim_{t\to0}\frac{1}{t}\frac{M}{2\pi}\int_{-\pi/M}^{\pi/M} X_{rot,3}(s,t)ds.
$$

%\noindent where, from \eqref{e:rotatedXT},
%$$
%X_{rot,3} = \frac{-A_3}{\sqrt{A_2^2+A_3^2}}X_2 + \frac{A_2}{\sqrt{A_2^2+A_3^2}}X_3.
%$$

\noindent Hence, from \eqref{e:1selfsimilar},
$$
c_M = \lim_{t\to0}\frac{1}{\sqrt{t}}\frac{M}{2\pi}\int_{-\pi/M}^{\pi/M} X_{rot,3}(s/\sqrt t,1)ds = \lim_{t\to0}\frac{M}{2\pi}\int_{-\pi/(M\sqrt t)}^{\pi/(M\sqrt t)} X_{rot,3}(s,1)ds,
$$

\noindent which enables us to conjecture that	
\begin{equation}
\label{e:cMlimit2}
c_M = \frac{M}{2\pi}\int_{-\infty}^{+\infty} X_{rot,3}(s,1)ds.
\end{equation}

\noindent From now on, we will simply write $X_{rot}(s)$ to denote $X_{rot,3}(s,1)$. The aim of this section is to prove the following result:
\begin{teo}
\label{teo:teo}
\begin{equation}
\label{e:teo}
\int_{-\infty}^{\infty} X_{rot}(s)ds=\frac{2\pi c_0^2}{\sqrt{e^{\pi c_0^2} - 1}}.
\end{equation}
\end{teo}

\begin{remark}

\noindent Note that, from the choice of $\bb$ in the inicial conditions given by \eqref{e:1cornerproblemsystem0}, and the definition of $X_{rot}$ in \eqref{e:rotatedXT}, we have
\begin{equation}
\label{e:>0}
\int_{-\infty}^{\infty} X_{rot}(s)ds>0.
\end{equation}

\noindent The value of $c_M$ follows immediately from \eqref{e:teo}, but in a more general form, i.e., as a function of $c_0$. In order to recover \eqref{e:cM}, we just apply \eqref{e:c_0}:
\begin{align*}
c_M  = \frac{Mc_0^2}{\sqrt{e^{\pi c_0^2} - 1}} = \frac{-2\ln(\cos(\pi/M))}{(\pi/M)\tan(\pi/M)} = \frac{\ln(1 + \tan^2(\pi/M))}{(\pi/M)\tan(\pi/M)}.
\end{align*}

\end{remark}

\subsection{Proof of Theorem \ref{teo:teo}}

Along this section, we assume that $t = 1$. Therefore, in order to simplify the notation, we write $\X(s) \equiv \X_{c_0}(s, 1)$, $\T(s) \equiv \T_{c_0}(s, 1)$, $\nn(s) \equiv \nn_{c_0}(s, 1)$, and $\bb(s) \equiv \bb_{c_0}(s, 1)$; so, $\X_{c_0}(s, t) = \sqrt t\X(s/\sqrt t)$, $\T_{c_0}(s, t) = \T(s/\sqrt t)$, $\nn_{c_0}(s, t) = \nn(s/\sqrt t)$, and $\bb_{c_0}(s, t) = \bb(s/\sqrt t)$. Moreover, we use the nonbold letters $X(s)$, $T(s)$, $n(s)$ and $b(s)$, whenever we refer to any component of their bold counterparts at $t = 1$.

Let us obtain first the ODE satisfied by $X(s)$. As in \cite{GutierrezRivasVega2003}, we have to differenciate $X$ three times:
$$
X'''(s) = T''(s) = (c_0n(s))'=c_0n'(s) = -c_0^2T(s) + \dfrac s2 c_0b(s).
$$

\noindent On the other hand, differentiating both sides of $\X_{c_0}(s, t) = \sqrt t\X(s/\sqrt t)$ with respect to $t$:
$$
\left.
\begin{array}{r}
\X_{c_0,t}(s, t) = \dfrac{c_0}{\sqrt t}\bb_{c_0}(s, t) = \dfrac{c_0}{\sqrt t}\bb(s / \sqrt t)
 \\
\left[\sqrt t\X(s/\sqrt t)\right]_t = \dfrac{1}{2\sqrt t}\X(s / \sqrt t) - \dfrac{s}{2t}\Xs(s/\sqrt t)
\end{array}
\right\}
\Longrightarrow
c_0\bb(s / \sqrt t) = \dfrac{1}{2}\X(s / \sqrt t) - \dfrac{s}{2\sqrt t}\Xs(s/\sqrt t),
$$

\noindent i.e.,
\begin{equation}
\label{e:c0b}
c_0b(s) = \frac{1}{2}X(s) - \frac{s}{2}X'(s),
\end{equation}

\noindent so we conclude that
\begin{equation}
\label{e:ode}
X'''(s) + \left(c_0^2 + \frac{s^2}{4}\right)X'(s) - \frac{s}{4}X(s) = 0.
\end{equation}

\noindent Hence, if $\hat X(\xi)$ denotes the Fourier transform of $X$, i.e.,
\begin{equation}
\label{e:fouriertransform}
\hat X(\xi) = \frac{1}{\sqrt{2\pi}}\int_{-\infty}^{+\infty}X(s)e^{-is\xi}ds \quad \Longleftrightarrow \quad X(s) = \frac{1}{\sqrt{2\pi}}\int_{-\infty}^{+\infty}\hat X(\xi)e^{is\xi}d\xi,
\end{equation}

\noindent then, $\hat X(\xi)$ satisfies
\begin{equation}
\label{e:ODEhatX}
\xi\hat X''(\xi) + 3\hat X'(\xi) + 4\xi^3\hat X(\xi) - 4c_0^2\xi\hat X(\xi) = 0.
\end{equation}

\noindent Recall that any equation formulated in terms of $\hat X$ is valid for any component of $\hat \X$, and, in particular, for $\hat X_{rot}$, which is a linear combination of the components of $\hat \X$. From now on, we will work in the Fourier side, so, from \eqref{e:fouriertransform}, proving Theorem \eqref{teo:teo} is reduced to computing
\begin{equation}
\label{e:sqrt2pihatXrot}
\sqrt{2\pi}\hat X_{rot}(0) = \int_{-\infty}^{\infty} X_{rot}(s)ds.
\end{equation}

\noindent Let us define now $\hat Y(\xi^2) = \xi^2\hat X(\xi) \Longleftrightarrow \hat Y(\eta) = \eta\hat X(\sqrt\eta)$, where $\eta = \xi^2 > 0$; then, \eqref{e:ODEhatX} becomes
\begin{equation}
\label{e:hatY}
\hat Y''(\eta) + \left(1 - \frac{c_0^2}{\eta}\right)\hat Y(\eta) = 0, \quad \eta > 0.
\end{equation}

\noindent Observe that, if $\hat Y(0)=0$, then
\begin{equation}
\label{e:hat X(0)}
\hat X(0)=\lim_{\xi\to 0}\frac{\hat Y(\xi^2)}{\xi^2}=\hat Y'(0).
\end{equation}

\noindent Besides, if $\hat Y(0)=0$, the solution is analytic. Moreover, we have the following result.

\begin{lema}
\label{basiclemma}
\begin{itemize}
\item	[(i)]
If $\hat Y(0) = 0$ , then $|\hat Y(\eta)|$ and $|\hat Y'(\eta)|$ are bounded in $\eta\in[0,\infty)$.
\item	[(ii)]
If $\hat Y(0) \neq0$,  then $|\hat Y(\eta)|$ and $|\hat Y'(\eta)|$ are bounded in $\eta\in[\epsilon,\infty)$ for all $\epsilon>0$. Moreover, for $0<\eta$ small, we have
\begin{equation*}
\hat Y'(\eta)=c_0^2\ln(\eta) + \mathcal O(\eta |\ln(\eta)|).
\end{equation*}
\end{itemize}	
\end{lema}

\begin{proof}

\noindent Let us define the energy $E(\eta)$:
\begin{equation*}
E(\eta) = (\hat Y')^2(\eta) + \left(1 - \frac{c_0^2}{\eta}\right)\hat Y^2(\eta).
\end{equation*}

\noindent Differentiating it,
\begin{equation*}
E'(\eta) = 2\hat Y'(\eta)\left[\hat Y''(\eta) + \left(1 - \frac{c_0^2}{\eta}\right)\hat Y(\eta)\right] + \frac{c_0^2}{\eta^2}\hat Y^2(\eta) = \frac{c_0^2}{\eta^2}\hat Y^2(\eta).
\end{equation*}

\noindent On the other hand, if $\eta > 2c_0^2$,
\begin{equation*}
E(\eta) \ge \left(1 - \frac{c_0^2}{\eta}\right)\hat Y^2(\eta) \ge \frac12\hat Y^2(\eta),
\end{equation*}
\noindent i.e., 
$$
E'(\eta) - \frac{2c_0^2}{\eta^2}E(\eta) \le 0,
$$

\noindent and, by Gr\"onwall's lemma,
$$
0 \le E(\eta)\le C, \quad \forall \eta \ge c_0^2.
$$

\noindent Finally, when $\eta\to0^+$, we have trivially
$$
\lim_{\eta\to0^+}E(\eta) = (\hat Y')^2(0) - c_0^2\lim_{\eta\to0^+}\frac{\hat Y^2(\eta)}{\eta} = (\hat Y')^2(0) - c_0^2\lim_{\eta\to0^+}\frac{2\hat Y(\eta)\hat Y'(\eta)}{1} = (\hat Y')^2(0).
$$

\noindent Therefore, for some other positive constant $C$, $|E(\eta)|\le C$, for all $\eta\in[0,\infty)$, from which we conclude the boundedness of $|\hat Y(\eta)|$ and $|\hat Y'(\eta)|$, and
$$
E(\infty) = E_0 + \int_0^{+\infty} \frac{c_0^2}{\eta^2}\hat Y^2(\eta)\,d\eta,
$$

\noindent which gives (i). With respect to (ii), just write
$$
\hat Y'(\eta)= \hat Y'(1)-\int_\eta^1 \hat Y''(\tau)\,d\tau,
$$

\noindent and the result easily follows.

\end{proof}

\noindent The rest of the proof of Theorem \ref{teo:teo} is divided in several subsections.

\subsubsection{Asymptotics $(\eta\gg 1)$ for $\hat Y(\eta)$}

Since \eqref{e:hatY} has real coefficients, we can decompose $\hat Y(\eta)$ into its real and imaginary parts, i.e., $\hat Y(\eta) = \hat Y_{re}(\eta) + i\hat Y_{im}(\eta)$, with both $\hat Y_{re}(\eta)$ and $\hat Y_{im}(\eta)$ satisfying \eqref{e:hatY}. Let us define first $F(\eta) = \hat Y_{re}(\eta) + i\hat Y_{re}'(\eta)$; then,
$$
F'(\eta)=\hat Y_{re}'(\eta)+i\hat Y_{re}''(\eta)=\hat Y_{re}'(\eta)+i\left(\frac{c_0^2}{\eta}-1\right)\hat Y_{re}(\eta)=
-iF(\eta)+i\frac{c_0^2}{2\eta}\left(F(\eta)+\bar F(\eta)\right),
$$

\noindent i.e.,
\begin{equation}
\label{e:F'}
F'(\eta)+i\left(1 - \frac {c_0^2}{2\eta}\right)F(\eta)=i\frac{c_0^2}{2\eta}\bar F(\eta).
\end{equation}

\noindent After applying the integrating factor $e^{i\phi(\eta)}$, with $\phi(\eta)=\eta-(c_0^2/2)\ln|\eta|$, we get
$$
e^{-i\phi(\eta)}\left(e^{i\phi(\eta)}F(\eta)\right)'=i\frac{c_0^2}{2\eta}\bar F(\eta);
$$

\noindent hence, defining $G(\eta) = e^{i\phi(\eta)}F(\eta)$,
$$
G'(\eta)=i\frac{c_0^2}{2\eta}e^{2i\phi(\eta)}\bar G(\eta).
$$

\noindent We write $G(\eta)=F^\infty +g(\eta)$; then, $g(\eta)$ satisfies
$$
g'(\eta)=i\frac{c_0^2}{2\eta}\bar F^\infty e^{2i\phi(\eta)}+i\frac{c_0^2}{2\eta}e^{2i\phi(\eta)}\bar g(\eta).
$$

\noindent Integrating over $[\eta, \eta']$,
\begin{equation}
\label{e:grecursive}
g(\eta')-g(\eta) = \int_\eta^{\eta' }g'(\tau)d\tau = i\frac{c_0^2}{2}\bar F^\infty \int_{\eta}^{\eta'} e^{2i\phi(\tau)}\frac{d\tau}{\tau} + i\frac{c_0^2}{2}
\int_{\eta}^{\eta'} e^{2i\phi(\tau)}\bar g(\tau)\frac{d\tau}{\tau}.
\end{equation}

\noindent At this point, we observe that
$$
\int_{\eta}^{\infty}e^{i\alpha\phi(\tau)}\frac{d\tau}{\tau^m} = -\frac{e^{i\alpha\phi(\eta)}}{i\alpha\eta^m} + \frac{m + i\alpha c_0^2/2}{i\alpha}\int_{\eta}^{+\infty}e^{i\alpha\phi(\tau)}\frac{d\tau}{\tau^{m+1}} = \ldots = i\frac{e^{i\alpha\phi(\eta)}}{\alpha\eta^m} + \mathcal O\left(\frac{1}{\eta^{m+1}}\right);
$$

\noindent and
$$
i\frac{c_0^2}{2}\bar F^\infty \int_{\eta}^{\eta'}e^{2i\phi(\tau)}\frac{d\tau}{\tau} =\bar F^\infty \left(\frac{c_0^2} {4\eta'}e^{2i\phi(\eta')} -\frac{c_0^2} {4\eta}e^{2i\phi(\eta)}\right)+ \mathcal O\left(\frac{1}{\eta^2}\right).
$$

\noindent Inserting the last expression into \eqref{e:grecursive} and using that $g$ is bounded, we get
\begin{align*}
g(\eta')-g(\eta)  & = \frac{c_0^2}{4\eta'}\bar F^\infty e^{2i\phi(\eta')}-\frac{c_0^2}{4\eta}\bar F^\infty e^{2i\phi(\eta)} + \mathcal O\left(\frac{1}{\eta^2}\right) - i\frac{c_0^2}{2}
\int_{\eta}^{\eta' }e^{2i\phi(\tau)}\bar g(\tau)\frac{d\tau}{\tau}
	\cr
%& = \frac{c_0^2}{4\eta}\bar F^\infty e^{2i\phi(\eta)} + \mathcal O\left(\frac{e^{2i\phi(\eta)}}{\eta^2}\right) - i\frac{c_0^2}{2}
%\int_{\eta}^\eta'\left[\frac{c_0^2}{4\tau^2}F^\infty + \mathcal O\left(\frac{1}{\tau^3}\right) + i\frac{c_0^2}{2\tau}e^{2i\phi(\tau)}
%\int_{\tau}^{+\infty} e^{-2i\phi(\zeta)}g(\zeta)\frac{d\zeta}{\zeta}\right]d\tau
%	\cr
& = \left(\frac{c_0^2}{4\eta'}\bar F^\infty e^{2i\phi(\eta')}-i\frac{c_0^4}{8\eta'}F^\infty\right)-\left(\frac{c_0^2}{4\eta}\bar F^\infty e^{2i\phi(\eta)} -i\frac{c_0^4}{8\eta}F^\infty\right)
+ \mathcal O\left(\frac{1}{\eta^2}\right).
\end{align*}

\noindent Therefore, passing to the limit in $\eta'$, we get,
\begin{align}
\label{e:Feta}
F(\eta) & = e^{-i\phi(\eta)}G(\eta) = F^\infty e^{-i\phi(\eta)} + e^{-i\phi(\eta)}g(\eta)
	\cr
& = F^\infty e^{-i\phi(\eta)} + \frac{c_0^2}{4\eta}\bar F^\infty e^{i\phi(\eta)} - i\frac{c_0^4}{8\eta}F^\infty e^{-i\phi(\eta)} + \mathcal O\left(\frac{1}{\eta^2}\right).
\end{align}

\noindent Similarly, using the above expression and \eqref{e:F'}, we get that
\begin{align}
\label{e:F'eta}
(F(\eta)  - F^\infty e^{-i\phi(\eta)})'= i\frac{c_0^2}{4\eta}\bar F^\infty e^{i\phi(\eta)} -\frac{c_0^4}{8\eta}F^\infty e^{-i\phi(\eta)} + \mathcal O\left(\frac{1}{\eta^2}\right).
\end{align}

\noindent  On the other hand, if we define $F(\eta) = \hat Y_{im}(\eta) + i\hat Y_{im}'(\eta)$ and take $F^\infty = \lim_{\eta\to\infty}[e^{-i\phi(\eta)}(\hat Y_{im}(\eta) + i\hat Y_{im}'(\eta))]$, \eqref{e:Feta} is still valid. Consequently, it is also valid for $F(\eta) = \hat Y(\eta) + i\hat Y'(\eta)$, with $F^\infty = \lim_{\eta\to\infty}[e^{-i\phi(\eta)}(\hat Y(\eta) + i\hat Y'(\eta))]$.

\begin{lema}\label{lema1} Assume $F(\eta) = \hat Y(\eta) + i\hat Y'(\eta)$,  with $\hat Y(0)=0$. Define, for $\eta<0$, $F(\eta)=F(-\eta)$. Then, when $|s| \gg 1$,
\begin{align*}
\frac{1}{\sqrt{2\pi}}\int_{-\infty}^{+\infty}F(\xi^2)e^{is\xi}d\xi = \frac{1}{\sqrt{2\pi}}\int_{-\infty}^{+\infty}F^\infty e^{i(s\xi-\xi^2 + c_0^2\ln|\xi|)}d\xi+ \mathcal O\left(\frac{1}{|s|}\right).
\end{align*}
\end{lema}

\begin{proof}

Note that, if $\varphi$ is a smooth cut-off function outside the origin (i.e., $\varphi=0$, in $|\xi|<1$, and $\varphi=1$, in $|\xi|>2$), then
$$
\int_{-\infty}^{+\infty}\varphi(\xi)F(\xi^2)e^{is\xi}d\xi = \mathcal O\left(\frac{1}{|s|}\right).
$$

\noindent Hence, the lemma follows from \eqref{e:Feta} and \eqref{e:F'eta}, and repeated integration by parts (see also \eqref{e:OI} below).
\end{proof}

\subsubsection{Equation for $Y$}

Let us write
\begin{equation}
\label{e:hatYeta}
\hat Y(\eta) = \frac{1}{\sqrt{2\pi}}\int_{-\infty}^{+\infty}Y(t)e^{-it\eta}dt, \quad \eta \ge 0,
\end{equation}

\noindent for some $Y(t)$. From \eqref{e:hatY}, we take $Y(t)$ such that
$$
Y'(t) = \frac{2t + ic_0^2}{1 - t^2}Y(t);
$$

\noindent hence, if
\begin{equation}
\label{e:Yu}
Y(t) = \frac{Y(0)}{1 - t^2}\left|\frac{1 + t}{1 - t}\right|^{ic_0^2/2}, \quad t\not=\pm1,
\end{equation}

\noindent then, $\hat Y$ solves \eqref{e:hatY}. Eventually, we will assume $\hat Y(\eta)\in \mathbb R$; therefore, we will take $Y(0)\in \mathbb R$.

\begin{lema} \label{lema2} Let $Y(t)$ be as in \eqref{e:Yu}; then,
$$
\int_0^{+\infty}Y(t)dt = 0 \quad \mbox{and} \quad \int_{-\infty}^0Y(t)dt = 0,
$$
and, therefore, $\hat Y(0)=0$.
\end{lema}

\begin{proof}

Writing
$$
\int_0^{+\infty}Y(t)dt=\int_0^{1^-}Y(t)dt+\int_{1^+}^{+\infty}Y(t)dt,
$$

\noindent and applying the change of variable
$$
v=\frac{1+t}{1-t},\qquad
t=\frac{v-1}{v+1},\qquad
dt = \frac{2dv}{(v + 1)^2},
$$

\noindent we have
\begin{equation*}
\int_0^{+\infty}Y(t)dt = Y(0)\left(\int_{-\infty}^{-1} + \int_{1}^{\infty}\right)\frac{1}{2v}|v|^{ic_0^2/2}dv.
\end{equation*}

\noindent The proof is identical for $\int_{-\infty}^0Y(t)dt$.

\end{proof}

\subsubsection{$\hat Y'(0)$ in terms of $Y(0)$}

Recall \eqref{e:hat X(0)}:
$$
\hat Y'(0)=\lim_{\xi\to 0}\hat Y(\xi^2)=-Y(0)\frac{1}{\sqrt{2\pi}}\lim\limits_{\eta\to 0} J(\eta),
$$ 

\noindent with $\eta>0$, and
\begin{align*}
J(\eta) & = i\int_{-\infty}^{+\infty}\frac t{(t^2-1)}\left|\frac{t+1}{t-1}\right|^{i\frac{c_0^2}2}e^{-it\eta}\,dt
    \cr
& = \frac i2\int_{-\infty}^{+\infty}\frac 1{t+1}\left|\frac{t+1}{t-1}\right|^{i\frac{c_0^2}2}e^{-it\eta}\,dt + \frac i2\int_{-\infty}^{+\infty}\frac 1{t-1}\left|\frac{t+1}{t-1}\right|^{i\frac{c_0^2}2}e^{-it\eta}\,dt
    \cr
& = \frac i2 e^{i\eta}\int_{-\infty}^{+\infty}\frac 1t e^{-it\eta}\left|\frac{t}{t-2}\right|^{i\frac{c_0^2}2}\,dt + \frac i2 e^{-i\eta}\int_{-\infty}^{+\infty}\frac 1t e^{-it\eta}\left|\frac{t+2}{t}\right|^{i\frac{c_0^2}2}\,dt.
	\cr
& = J_1(\eta)+J_2(\eta).
\end{align*}

\noindent We need the following two lemmas.

\begin{lema} \label{lemma1} Take $\eta>0$, then
$$
\lim_{\eta\to 0}\int_{-\infty}^{+\infty}\frac{\sin(t\,\eta)}{t}\left|\frac{t}{t-2}\right|^{i\frac{c_0^2}2}dt = \lim_{\eta\to 0}\int_{-\infty}^{+\infty}\frac{\sin(t\,\eta)}{t}\left|\frac{t+2}{t}\right|^{i\frac{c_0^2}2}dt = 2\int_{-\infty}^{\infty}\frac{\sin(t)}{t} dt=2\pi.
$$

\end{lema}

\begin{proof}

Take $R\gg1$, then
\begin{align*}
\int_{-\infty}^{+\infty}\frac{\sin(t\,\eta)}{t}\left|\frac{t}{t-2}\right|^{i\frac{c_0^2}2} & = \int_{-\infty}^{+\infty}\frac{\sin(t)}{t}\left|\frac 1{1-2\eta/t}\right|^{i\frac{c_0^2}2}dt
	\cr
& = \int_{|t|\le R}\frac{\sin(t)}{t}\left|\frac 1{1-2\eta/t}\right|^{i\frac{c_0^2}2}dt + \int_{|t| > R}\frac{\sin(t)}{t}\left|\frac 1{1-2\eta/t}\right|^{i\frac{c_0^2}2}dt.
\end{align*}

\noindent In the second integral, we can integrate by parts to obtain a bound $C/M$, with $C$ uniform in $\eta$. In the first integral, we can pass to the limit in $\eta$, by the dominated convergence theorem. The result easily follows.

\end{proof}

\begin{lema} \label{lemma2}
Define
\begin{equation}
\label{e:J}
\tilde{J}=\int_{-\infty}^{\infty}\frac{t}{t^2-1}\left|\frac{t+1}{t-1}\right|^{i\frac{c_0^2}2}dt.
\end{equation}

\noindent Then
$$
\lim_{\eta\to 0}\int_{-\infty}^{\infty}\frac{\cos(t\,\eta)}{t}\left|\frac{t}{t-2}\right|^{i\frac{c_0^2}2}dt+
\lim_{\eta\to 0}\int_{-\infty}^{\infty}\frac{\cos(t\,\eta)}{t}\left|\frac{t+2}{t}\right|^{i\frac{c_0^2}2}dt
=\tilde{J},
$$
and
$$
\tilde{J}=
i\pi\tanh(c_0^2\pi/4).
$$
\end{lema}

\noindent As a consequence, we get
\begin{equation}
\label{e:J(0)}
\lim_{\eta\to 0}J(\eta)=\pi(1-\tanh(c_0^2\pi/4)),
\end{equation}
and
\begin{equation}
\label{e:hatY'(0)}
\hat Y'(0)=-Y(0)\frac{1}{\sqrt{2\pi}}\lim_{\eta\to 0}J(\eta)=-Y(0)\frac{\sqrt{2\pi}}{2}(1-\tanh(c_0^2\pi/4)).
\end{equation}

\begin{proof}

Take
$$
I(\eta)=\int_{-\infty}^{+\infty}\frac{\cos(t\,\eta)}{t}\left|\frac{t+2}{t}\right|^{i\frac{c_0^2}2}dt.
$$

\noindent Consider $\varphi$ an even cut-off function, such that $\text{supp}\,\varphi\subset\left\{|t|\le 1\right\}$, and $\varphi(t)\equiv 1$ ,if $|t|\le 1/2$. Define 
$$
I_1(\eta) = \int_{-\infty}^{+\infty}\frac{\cos(t\,\eta)}{t}\varphi(t)\left|\frac{t+2}{t}\right|^{i\frac{c_0^2}2}dt,\quad\text{and}\quad
I_2(\eta) = I(\eta)-I_1(\eta).
$$

\noindent On the one hand, 
\begin{align*}
I_1(\eta)& =\frac{2}{ic_0^2}\int_{-\infty}^{+\infty}\cos(t\,\eta)\varphi(t)|t+2|^{i\frac{c_0^2}2}\frac{d}{dt}e^{i\frac{c_0^2}2\ln |t|}dt
    \cr
&=-\frac{2}{ic_0^2}\int_{-\infty}^{+\infty}\frac{d}{dt}\left\{\cos (t\,\eta)\varphi(t)|t+2|^{i\frac{c_0^2}2}\right\}e^{i\frac{c_0^2}2\ln |t|}dt.
\end{align*}

\noindent The last integral is absolutely convergent and dominated by $C\chi_{|t|\le 1}$ for some constant $C$. Therefore,
\begin{align*}
\lim_{\eta\to 0}I_1(\eta)& =-\frac{2}{ic_0^2}\int_{-\infty}^{+\infty}\frac{d}{dt}\left\{\varphi(t)|t+2|^{i\frac{c_0^2}2}\right\}e^{i\frac{c_0^2}2\ln |t|}dt
    \cr
&=\frac{2}{ic_0^2}\int_{-\infty}^{+\infty}\varphi(t)|t+2|^{i\frac{c_0^2}2}\frac{d}{dt}e^{i\frac{c_0^2}2\ln |t|}.
\end{align*}

\noindent On the other hand,
$$
I_2(\eta)=\int_0^{\infty}\frac{\cos(\eta\,t)}{t}(1-\varphi(t)) f(t)dt,
$$

\noindent with 
$$
f(t)=|t|^{i\frac{c_0^2}2}\left(|t+2|^{i\frac{c_0^2}2}-|t-2|^{i\frac{c_0^2}2}\right),
$$ 

\noindent so that
$$
|f(t)|\le\frac{2c_0^2}{t}.
$$

\noindent Hence,
$$
\lim_{\eta\to 0}I_2(\eta)=\int_0^{\infty}(1-\varphi(t))f(t)\frac{dt}t,
$$

\noindent and
\begin{align*}
\lim_{\eta\to 0}I(\eta) =\frac{2}{ic_0^2}\int_{-\infty}^{+\infty} |t+2|^{i\frac{c_0^2}2} \frac{d}{dt} |t|^{i\frac{c_0^2}2} =\frac{2}{ic_0^2}\int_{-\infty}^{+\infty} |t+1|^{i\frac{c_0^2}2} \frac{d}{dt} |t-1|^{i\frac{c_0^2}2}.
\end{align*}

\noindent We can proceed in an analogous way and prove that
$$
\lim_{\eta\to 0}
\int_{-\infty}^{+\infty}\frac{\cos(t\,\eta)}{t}\left|\frac{t}{t-2}\right|^{i\frac{c_0^2}2}dt=
-\frac{2}{ic_0^2}\int_{-\infty}^{+\infty}|t-1|^{i\frac{c_0^2}2} \frac{d}{dt} |t+1|^{i\frac{c_0^2}2}.
$$

\noindent Therefore, the sum of the two integrals gives
$$
\int_{-\infty}^{\infty}\frac{t}{t^2-1}\left|\frac{t+1}{t-1}\right|^{i\frac{c_0^2}2}dt=\tilde{J}.
$$

\noindent Let us compute $\tilde{J}$:
\begin{align*}
\tilde{J} = 2i\Im\int_0^{+\infty}\frac{t}{t^2 - 1}\left|\frac{1 + t}{1 - t}\right|^{ic_0^2/2}dt = 4i\Im\int_1^{+\infty}\frac{1}{v^2 - 1}|v|^{ic_0^2/2}dv,
\end{align*}

\noindent where, in the last step, we have used the same change of variable as in Lemma \ref{lema2}. Finally, applying the change $v = e^w$,
\begin{equation}
\label{nueve}
\tilde{J} = 4i\int_0^{+\infty}\frac{\sin(c_0^2w/2)}{e^w - e^{-w}}dw = i\pi\tanh(c_0^2\pi/4).
\end{equation}

\end{proof}

\subsubsection{Asymptotics of $\hat Y(\eta)$ in terms of $Y(0)$}

Coming back to \eqref{e:Yu}, we observe that $Y(t)$ is regular if $t^2\not=1$, so we can decompose $\hat Y(\eta)$ in \eqref{e:hatYeta} as
$$
\hat Y(\eta) = \hat Y_0(\eta) + \hat Y_1(\eta),
$$

\noindent where $\hat Y_1(\eta)$ is given by the contribution coming from the region $|t^2-1|>1/4$ after introducing a smooth cut-off function $\varphi$, such that $\varphi(t) = 1$, in $|t^2 - 1| < 1 / 4$; and $\varphi(t) = 0$, in $|t^2 - 1| > 1 / 2$, i.e.,
\begin{align*}
\hat Y_0(\eta) & = \frac{1}{\sqrt{2\pi}}\int_{-\infty}^{+\infty}\varphi(t)Y(t)e^{-it\eta}dt,
	\cr
\hat Y_1(\eta) & = \frac{1}{\sqrt{2\pi}}\int_{-\infty}^{+\infty}[1 - \varphi(t)]Y(t)e^{-it\eta}dt.
\end{align*}

\noindent $\hat Y_1(\eta)$ poses no problems; indeed, it is immediate to check that $\hat Y_1(\eta) =  \mathcal O(1/\eta)$, when $\eta\gg1$. With respect to $\hat Y_0(\eta)$, we further decompose it into the sum of three integrals:
\begin{align*}
\hat Y_0(\eta) & = \frac{1}{\sqrt{2\pi}}\int_{|t-1|<1}\varphi(t)Y(t)e^{-it\eta}dt + \frac{1}{\sqrt{2\pi}}\int_{|t+1|<1}\varphi(t)Y(t)e^{-it\eta}dt + \frac{1}{\sqrt{2\pi}}\int_{|t|>2}\varphi(t)Y(t)e^{-it\eta}dt
	\cr
& = I_1(\eta) + I_2(\eta) + 0.	
\end{align*}

\noindent Let us obtain first the asymptotics of $I_1(\eta)$:
\begin{align*}
I_1(\eta) & = \frac{Y(0)}{\sqrt{2\pi}}\int_{|t-1|<1}\frac{\varphi(t)}{1 - t^2}\left|\frac{1 + t}{1 - t}\right|^{ic_0^2/2}e^{-it\eta}dt
	\cr
& = -\frac{Y(0)e^{-i\eta}}{\sqrt{2\pi}}\int_{|t|<1}\frac{\varphi(t + 1)}{t(t + 2)}\left|\frac{t + 2}{t}\right|^{ic_0^2/2}e^{-it\eta}dt
	\\
%& = -\frac{Y(0)e^{-i\eta}}{\sqrt{2\pi}}\int_{|t|<1}\frac{\varphi(t + 1)}{2t(t/2 + 1)}\left|\frac{t/2 + 1}{t/2}\right|^{ic_0^2/2}e^{-it\eta}dt\\
& = -\frac{Y(0)e^{-i\eta}}{\sqrt{2\pi}}\frac{2^{ic_0^2/2}}{2}\int_{|t|<1}\tilde\varphi(t)\frac{1}{t}|t|^{-ic_0^2/2}e^{-it\eta}dt,
\end{align*}

\noindent where $\tilde\varphi$ is a new smooth cut-off function, such that $\tilde\varphi(0) = 1$. Now, assuming $\eta \gg 1$ and writing 
$\eta\, t = \tilde t$:
\begin{align*}
I_1(\eta) & = -\frac{Y(0)e^{-i\eta}}{\sqrt{2\pi}}\frac{2^{ic_0^2/2}}{2}|\eta|^{ic_0^2/2}\int_{|\tilde t|<\eta}\tilde\varphi(\tilde t/\eta)\frac{1}{\tilde t}|\tilde t|^{-ic_0^2/2}e^{-i\tilde t}d\tilde t
	\cr
%& = -\frac{Y(0)e^{-i\eta}}{\sqrt{2\pi}}\frac{|2\eta|^{ic_0^2/2}}{2}\int_{|\tilde t|<\eta}\frac{1}{\tilde t}|\tilde t|^{-ic_0^2/2}e^{-i\tilde t}d\tilde t
%	\cr
& = -\frac{Y(0)e^{-i\eta}}{\sqrt{2\pi}}\frac{|2\eta|^{ic_0^2/2}}{2}\int_{-\infty}^{+\infty}\frac{1}{t}|t|^{-ic_0^2/2}e^{-it}dt + \mathcal O\left(\frac{1}{\eta}\right).
\end{align*}

\noindent The asymptotics of $I_2(\eta)$ are obtained exactly in the same way:
\begin{equation*}
I_2(\eta) = \frac{Y(0)e^{i\eta}}{\sqrt{2\pi}}\frac{|2\eta|^{-ic_0^2/2}}{2}\int_{|t|<1}\frac{1}{t}|t|^{ic_0^2/2}e^{-it}dt = \bar I_1(\eta),
\end{equation*}

\noindent where we are using that $Y(0)$ is real. Putting $I_1(\eta)$ and $I_2(\eta)$ together,
\begin{align*}
\hat Y(\eta) & = I_1(\eta) +\bar I_1(\eta)
	\cr
& = -\frac{Y(0)}{\sqrt{2\pi}}\Re\left[e^{-i\eta + i(c_0^2/2)\ln|2\eta|}\int_{-\infty}^{+\infty}\frac{1}{t}|t|^{-ic_0^2/2}e^{-it}dt\right] + \mathcal O\left(\frac{1}{\eta}\right).
\end{align*}

\noindent The integral can be immediately computed by means of Mathematica\textsuperscript{\textregistered}:
\begin{align*}
\gamma = \int_{-\infty}^{+\infty}\frac{1}{t}|t|^{-ic_0^2/2}e^{-it}dt & = - i\int_{-\infty}^{+\infty}\frac{1}{t}|t|^{-ic_0^2/2}\sin(t)dt = -2\sinh(\pi c_0^2/4)\Gamma(-ic_0^2/2).
\end{align*}

\noindent Then, using the identity $|\Gamma(iy)|^2 = \pi / (y\sinh (\pi y)$, $y\in\mathbb R$, it follows that $|\gamma| = (2/c_0)[\pi\tanh(\pi c_0^2/4)]^{1/2}$. If we represent $\gamma \equiv |\gamma|e^{i\arg(\gamma)}$,
\begin{align*}
\hat Y(\eta) = \frac{-Y(0)}{\sqrt{2\pi}}\frac{2}{c_0}\sqrt{\pi\tanh(\pi c_0^2/4)}\Re\left[e^{-i\eta + i(c_0^2/2)\ln|2\eta| + i\arg(\gamma)}\right] + \mathcal O\left(\frac{1}{\eta}\right),
\end{align*}

\noindent which enables us to conclude that
%\textbf{ESTO ES LO \'ULTIMO QUE HABLAMOS. APARENTEMENTE, APAREC\'IA LA PARTE IMAGINARIA. ME COMENTASTE QUE HAB\'IA QUE MULTIPLICAR POR $i$.}
\begin{equation}
\label{doce}
\varlimsup\limits_{\eta\to\infty}\hat Y(\eta) = |Y(0)|\frac{\sqrt2}{c_0}\sqrt{\tanh(\pi c_0^2/4)}.
\end{equation}

\subsubsection{Equation for $({-\mathbf n}+i{\mathbf b})^{\wedge} (\xi)$ and final computation}

\noindent It is straightforward to express $\hat {\mathbf n}(\xi)$ and $\hat {\mathbf b}(\xi)$ in terms of $\hat {\mathbf Y}(\eta)$,  $\eta>0$:
\begin{align*}
{\mathbf T}'(s) = c_0{\mathbf n}(s) & \Longrightarrow c_0\hat{\mathbf n}(\xi) = i\xi\hat{\mathbf T}(\xi) = -\xi^2\hat{\mathbf X}(\xi) = -\hat {\mathbf Y}(\xi^2) = -\hat{\mathbf Y}(\eta),
 \cr
c_0{\mathbf b}'(s) = -\frac{s}{2}{\mathbf X}''(s) & \Longrightarrow c_0i\xi \hat {\mathbf b}(\xi) = \frac{i}{2}\frac{d(\xi^2\hat {\mathbf X}(\xi))}{d\xi} = \frac{i}{2}\frac{d(\hat {\mathbf Y}(\xi^2))}{d\xi} = i\xi \hat {\mathbf Y}'(\xi^2)
 \cr
& \Longrightarrow c_0\hat {\mathbf b}(\xi) = \hat {\mathbf Y}'(\xi^2) = \hat {\mathbf Y}'(\eta) ,
\end{align*}

\noindent where we have differentiated \eqref{e:c0b} in the last expression. Then,
\begin{equation}
\label{energy3}
(-\Ts + i\X_t)^\wedge(\xi) = c_0(-{\mathbf n} + i{\mathbf b})^\wedge(\xi) = (\hat {\mathbf Y} + i\hat {\mathbf Y}')(\xi^2) = (\hat {\mathbf Y} + i\hat {\mathbf Y}')(\eta).
\end{equation}

\noindent Recall that $\mathbf A^+ = (A_1, A_2, A_3)$ and $\mathbf A^- = (A_1, -A_2, -A_3)$. Hence, bearing in mind that $\mathbf A^- \wedge \mathbf A^+=2A_1(0,-A_3,A_2)$, we define
\begin{equation*}
%\label{e:e0}
\e = \left(0, \frac{-A_3}{\sqrt{A_2^2+A_3^2}}, \frac{A_2}{\sqrt{A_2^2+A_3^2}}\right),
\end{equation*}

\noindent which is coherent with \eqref{e:rotatedXT}; for example, $X_{rot} \equiv \e\cdot\X$, etc. Then,
\begin{equation}\label{seis}
\hat Y_{rot}(\eta) = \hat Y_{rot}(\xi^2) = \xi^2 \hat X_{rot}(\xi^2) = -\e\cdot c_0\hat{\mathbf n}(\xi).
\end{equation}

\noindent On the one hand, $\hat Y_{rot}$ solves \eqref{e:hatY}; on the other hand,
$$
\hat Y_{rot}(0) = - \e\cdot c_0\hat{\mathbf n}(0) = -\frac{1}{\sqrt{2\pi}}\int_{-\infty}^{+\infty}\e\cdot
c_0\textbf n (s)ds = \left.\frac{1}{\sqrt{2\pi}}\e\cdot \textbf T(s)\right|_{-\infty}^{+\infty} = \frac{1}{\sqrt{2\pi}}\e\cdot(0,2A_2,2A_3)=0,
$$

\noindent so $\hat Y_{rot}'(0)<+\infty$. Moreover, $\e\cdot\mathbf T$ is odd. Therefore, $\e\cdot\mathbf n$ and $\e\cdot\mathbf b =
(1/2)\e\cdot\X - (s/2)\e\cdot\T$ are both even. Hence, Lemma \ref{lema1} applies to
$$
F_{rot}=\hat Y_{rot}+i\hat Y_{rot}'=c_0\e\cdot(-{\hat{\mathbf n}}+i{\hat{\mathbf b}}).
$$

\noindent Our goal is to compute
\begin{equation}
\label{siete}
\hat Y_{rot}'(0) = c_0\int_{-\infty}^{+\infty} \e\cdot {\mathbf b}(s)ds.
\end{equation}

\noindent Observe that
$$
\frac 1{\sqrt{2\pi}}\int_{-\infty}^{+\infty}e^{-i\xi^2+ic_0^2\ln\xi+is\xi}d\xi = \frac{e^{is^2/4}}{\sqrt{2\pi}}\int_{-\infty}^{+\infty}
e^{-i(\xi-\frac s2)^2}|\xi|^{+ic_0^2}d\xi.
$$

\noindent Assuming $s\gg1$, the above integral is then
\begin{align}
\label{e:OI} \frac{e^{is^2/4}}{\sqrt{2\pi}}\left|\frac s2\right|^{ic_0^2}
\int_{-\infty}^{+\infty} e^{i\xi^2}d\xi+\mathcal O\left(\frac{1}{s}\right) = \frac{e^{is^2/4}}{\sqrt{2\pi}}\left|\frac s2\right|^{ic_0^2}\sqrt{i\pi}+\mathcal O\left(\frac{1}{s}\right).
\end{align}

\noindent Hence,
$$
c_0\e\cdot(-{\mathbf n}+i{\mathbf b})=e^{i\frac{s^2}4+ic_0^2\ln(\frac s2)}\sqrt{\frac i2}F_{rot}^\infty +O(1).
$$

\noindent On the other hand, ${\mathbf n}\cdot {\mathbf T}={\mathbf n}\cdot {\mathbf b}={\mathbf b}\cdot {\mathbf T}=0$, and
$$
\e\cdot\T(+\infty)=\left(0, \frac{-A_3}{\sqrt{A_2^2+A_3^2}}, \frac{A_2}{\sqrt{A_2^2+A_3^2}}\right)\cdot (A_1,A_2,A_3) = 0,
$$

\noindent so
$$
\left|\textbf c_0\e\cdot(-{\mathbf n}+i{\mathbf b})(+\infty)\right|^2=c_0^2\left((\e\cdot n)^2+(\e\cdot b)^2\right)=c_0^2.
$$

\noindent Therefore
$$
\left| F_{rot}^\infty \right|=\sqrt 2 c_0,
$$

\noindent and 
\begin{equation}\label{catorce}
\varlimsup\limits_{\eta\to\infty}\hat Y_{rot}(\eta)=\sqrt 2c_0.
\end{equation}

\noindent Recall that, from \eqref{e:hatY'(0)}, we have that
$$
\hat Y_{rot}'(0)=-Y_{rot}(0)\frac{\sqrt\pi}{\sqrt{2}}(1-\tanh(c_0^2\pi/4))=-Y_{rot}(0)\frac{\sqrt{2\pi}}{1+e^{\pi c_0^2/2}}.
$$

\noindent Moreover, from \eqref{doce}, we know that
$$
\varlimsup\limits_{\eta\to\infty}\hat Y_{rot}(\eta) = |Y_{rot}(0)|\frac{\sqrt2}{c_0}\sqrt{\tanh(\pi c_0^2/4)},
$$

\noindent which, from \eqref{catorce}, gives
$$
|Y_{rot}(0)|=\frac{c_0^2}{\sqrt{\tanh(\pi c_0^2/4)}}.
$$

\noindent Hence, bearing in mind \eqref{e:>0},
$$
\hat Y_{rot}'(0) = \frac{c_0^2}{\sqrt{\tanh(\pi c_0^2/4)}}\frac{\sqrt{2\pi}}{1+e^{\pi c_0^2/2}} = \frac{\sqrt {2\pi}c_0^2}{\sqrt{e^{\pi c_0^2} - 1}}.
$$

\noindent Therefore, from \eqref{e:sqrt2pihatXrot}, we get \eqref{e:teo}, which concludes the proof of Theorem \ref{teo:teo}.

\qed

\subsection{Numerical integration of \eqref{e:teo}}

From \eqref{e:ode}, together with \eqref{e:1cornerproblemsystem0} at $t = 1$, we get the following initial value problems for $X_2$ and $X_3$ (see also \cite{GutierrezRivasVega2003}):
\begin{equation*}
\begin{cases}
X_2'''(s) + \left(c_0^2 + \dfrac{s^2}{4}\right)X_2'(s) - \dfrac{s}{4}X_2(s) = 0,
\cr
X_2(0) = 0,
\cr
X_2'(0) = 0,
\cr
X_2''(0) = c_0,
\end{cases}
\end{equation*}

\noindent and
\begin{equation*}
\begin{cases}
X_3'''(s) + \left(c_0^2 + \dfrac{s^2}{4}\right)X_3'(s) - \dfrac{s}{4}X_3(s) = 0,
\cr
X_3(0) = 2c_0,
\cr
X_3'(0) = 0,
\cr
X_3''(0) = 0.
\end{cases}
\end{equation*}

\noindent Therefore, from \eqref{e:rotatedXT}, $X_{rot}$ is a solution of
\begin{equation}
\label{e:Grot}
\begin{cases}
X_{rot}'''(s) + \left(c_0^2 + \dfrac{s^2}{4}\right)X_{rot}'(s) - \dfrac{s}{4}X_{rot}(s) = 0,
\cr
X_{rot}(0) = \dfrac{2A_2\,c_0}{\sqrt{A_2^2 + A_3^2}},
\cr
X_{rot}'(0) = 0,
\cr
X_{rot}''(0) = -\dfrac{A_3\,c_0}{\sqrt{A_2^2 + A_3^2}}.
\end{cases}
\end{equation}

\noindent It is immediate to check that, if $X_{rot}$ is a solution of \eqref{e:Grot}, so is $\tilde X_{rot}(s) = X_{rot}(-s)$. Moreover, since $\tilde X_{rot}(0) = X_{rot}(0)$, $\tilde X_{rot}'(0) = X_{rot}'(0) = 0$, and $\tilde X_{rot}''(0) = X_{rot}''(0)$, we conclude that $X_{rot}$ is an even function, so \eqref{e:teo} becomes
\begin{equation}
\label{e:c_MH_rot}
\int_{-\infty}^{+\infty}X_{rot}(s)ds = 2\int_{0}^{+\infty}X_{rot}(s)ds = 
2\lim_{s\to\infty}H_{rot}(s),
\end{equation}

\noindent where
$$
H_{rot}(s) = \int_0^sX_{rot}(s')ds'
$$

\noindent is the solution of
\begin{equation}
\label{e:Hrot}
\begin{cases}
H_{rot}^{(4)}(s) + \left(c_0^2 + \dfrac{s^2}{4}\right)H_{rot}''(s) - \dfrac{s}{4}H_{rot}'(s) = 0,
\cr
H_{rot}(0) = 0,
\cr
H_{rot}'(0) = \dfrac{2A_2\,c_0}{\sqrt{A_2^2 + A_3^2}},
\cr
H_{rot}''(0) = 0,
\cr
H_{rot}'''(0) = -\dfrac{A_3\,c_0}{\sqrt{A_2^2 + A_3^2}}.
\end{cases}
\end{equation}

\noindent The numerical integration of \eqref{e:Hrot} poses no difficulty. We have approximated numerically $H_{rot}(s)$ in $s\in[0,1000]$, for a large set of $M$, namely $M\in\{3, \ldots, 10000\}$, by means of a fourth-order Runge-Kutta, taking $\Delta s = 10^{-3}$, and compared the value of $H_{rot}(1000)$ thus obtained, with the value of $\lim_{s\to\infty}H_{rot}(s)$ given by Theorem \eqref{teo:teo}, together with \eqref{e:c_0} and \eqref{e:c_MH_rot}:
\begin{equation}
\label{e:limitH}
\lim_{s\to \infty} H_{rot}(s) = \frac12\int_{-\infty}^{+\infty}X_{rot}(s)ds = \frac{\pi c_0^2}{\sqrt{e^{\pi c_0^2} - 1}} = \frac{\ln(1 + \tan^2(\pi/M))}{\tan(\pi/M)}.
\end{equation}

\noindent Since the largest deviation between the numerical approximations of $H_{rot}$ and their theoretical value for all the $M$ considered is only of $3.3507\cdot10^{-10}$ (and happens when $M = 4$), the results can be regarded as a sort of ``numerical proof'' of Thereom \ref{teo:teo}.

On the left-hand side of Figure \ref{f:HrotsM3}, we have plotted, in semilogarithmic scale, $H(s)$ as a function of $s$; remark how $H_{rot}(s)$ converges to a constant value approximately equal to $0.8$, as $s$ grows up; this convergence appears more clearly on the right-hand side of Figure \ref{f:HrotsM3}, where we have plotted, in logarithmic scale, $|H_{rot}(s) - \ln(4) / \sqrt{3}|$, i.e., the discrepancy between $H_{rot}(s)$ and its theoretical value in the limit, which, from \eqref{e:limitH}, is precisely $\lim_{s\to\infty}H_{rot}(s) = \ln(4) / \sqrt{3} = 0.800377422568629$.

\begin{figure}[!htb]
	\centering
	\includegraphics[width=0.49\textwidth, clip=true]{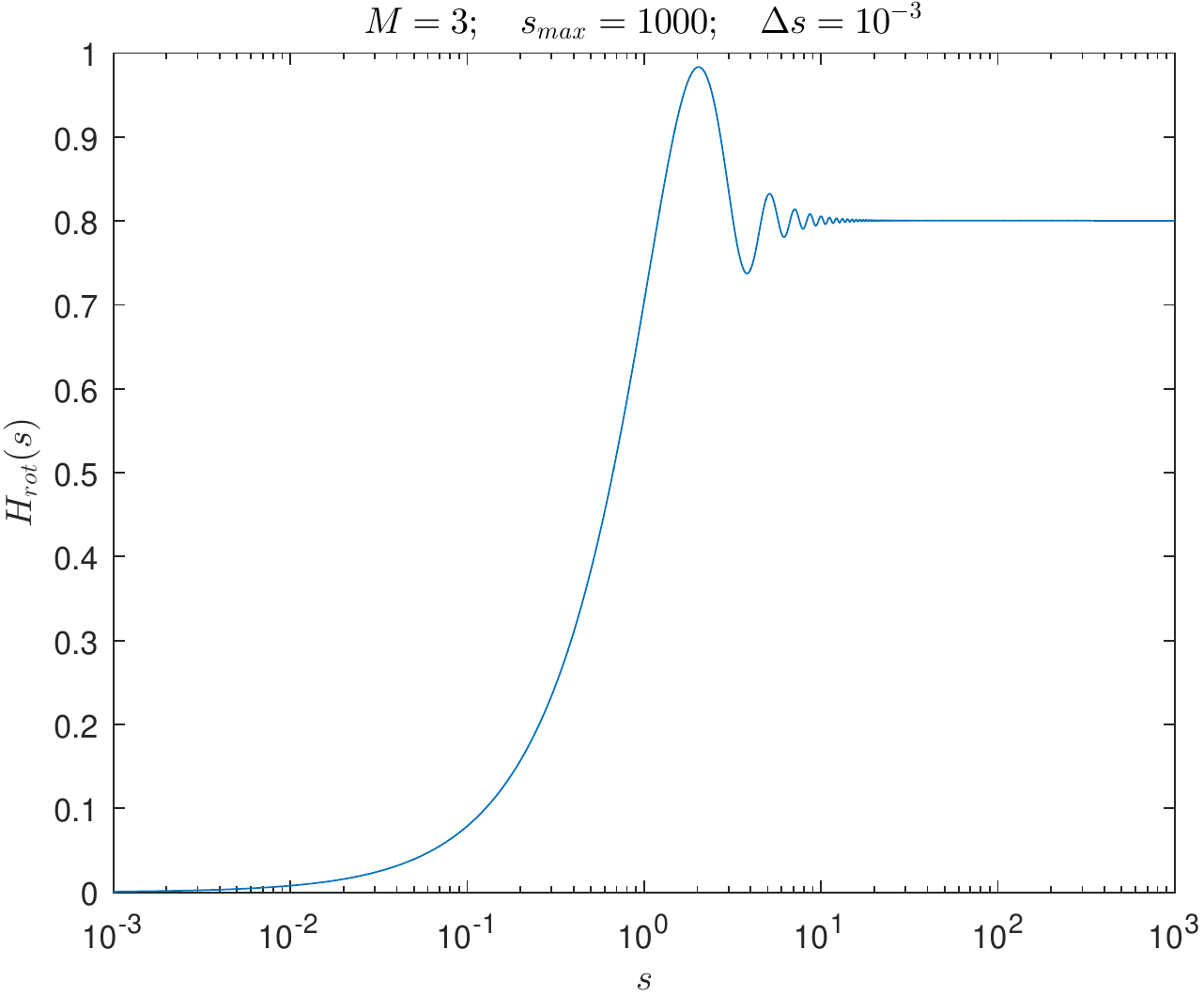}
	\includegraphics[width=0.49\textwidth, clip=true]{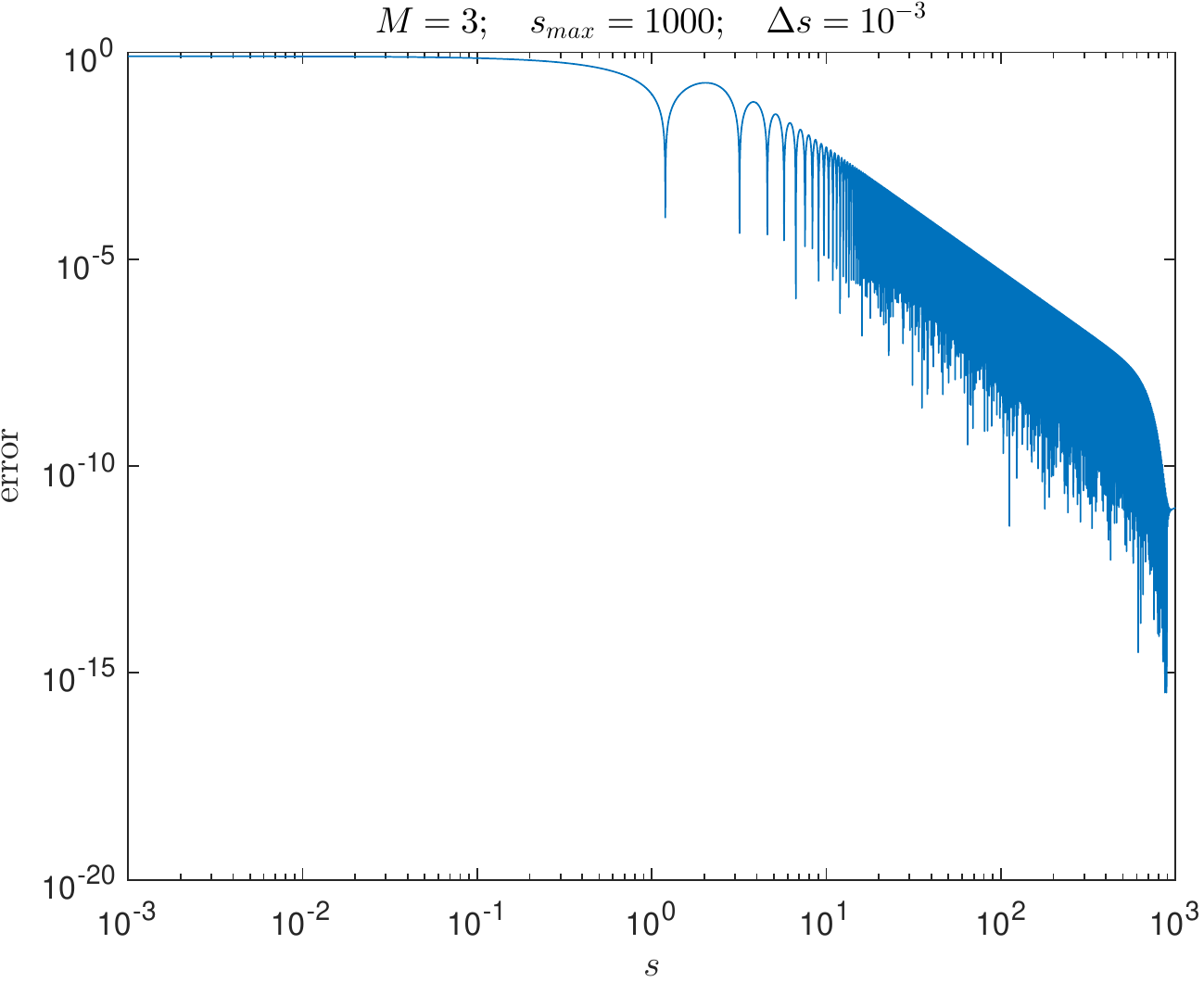}
	\caption{Left: Semilogarithmic plot of $H_{rot}(s)$, $s\in[0,1000]$, for $M = 3$. Right: Logarithmic plot of $|H_{rot}(s) - \ln(4) / \sqrt{3}|$, where $\lim_{s\to\infty}H_{rot}(s) = \ln(4) / \sqrt{3}$.}\label{f:HrotsM3}
\end{figure}

\subsection{Approximation of $c_M$ from the numerical simulation of \eqref{e:xt}-\eqref{e:schmap}}

In the beginning of this section, bearing in mind the relationship between the one-corner problem and the $M$-corner problem, we have conjectured that $c_M$ is given by \eqref{e:cMlimit2}. Then, assuming this conjecture to be true, the value of $c_M$ given by \eqref{e:cM} follows trivially from Theorem \ref{teo:teo}. In order to check numerically \eqref{e:cM}, we have compared it with the numerical approximations of $c_M$ given by $\mean(X_{num,3}(2\pi/M^2)) / (2\pi/M^2)$. For that purpose, we have simulated the evolution of \eqref{e:xt}-\eqref{e:schmap} as explained in \cite{HozVega2014}, for $M = 3, \ldots, 20$, and for different values of $N/M$ (cfr. \cite[Table 1]{HozVega2014}).

In Table \ref{t:errorcM}, we give $|c_M - \mean(X_{num,3}(2\pi/M^2)) / (2\pi/M^2)|$, where $c_M$ is given by \eqref{e:cM}. As in \cite{HozVega2014}, we have taken $N_t = 151200\cdot4^r$ time steps, where $N/M = 512\cdot2^r$; hence, $\Delta t = (2\pi/M^2)/N_t$. On the one hand, for a given $M$, when $N/M$ is doubled, the errors are divided by a factor a bit smaller than two, hence suggesting that the errors behave as $\mathcal O((N/M)^{-1})$, or, in other words, that there is a first-order convergence rate from $\mean(X_{num,3}(2\pi/M^2)) / (2\pi/M^2)$ to $c_M$ in \eqref{e:cM}. On the other hand, the accuracy also improves as $M$ gets bigger; this makes sense, too, because, $\Delta t$ gets smaller as well. When $N / M = 16384$, the numerical experiments are computationally expensive, so we have considered only $M = 3$, i.e., the case with the worst agreement, obtaining an improved error of $6.1416\cdot10^{-5}$ in the approximation of $c_M$.

\begin{table}[htb!]
	\centering
	\begin{tabular}{|c|c|c|c|c|c|}
		\hline $M$ & $N / M = 512$ & $N / M = 1024$ & $N / M = 2048$ & $N / M = 4096$ & $N / M = 8192$
		\\
		\hline $3$ & $7.4202\cdot10^{-4}$ & $4.7459\cdot10^{-4}$ & $2.9348\cdot10^{-4}$ & $1.7710\cdot10^{-4}$ & $1.0498\cdot10^{-4}$
		\\
		\hline $4$ & $5.5496\cdot10^{-4}$ & $3.3166\cdot10^{-4}$ & $1.9528\cdot10^{-4}$ & $1.1373\cdot10^{-4}$ & $6.5697\cdot10^{-5}$
		\\
		\hline $5$ & $3.8430\cdot10^{-4}$ & $2.2572\cdot10^{-4}$ & $1.3119\cdot10^{-4}$ & $7.5657\cdot10^{-5}$ & $4.3376\cdot10^{-5}$
		\\
		\hline $6$ & $2.7679\cdot10^{-4}$ & $1.6134\cdot10^{-4}$ & $9.3217\cdot10^{-5}$ & $5.3512\cdot10^{-5}$ & $3.0570\cdot10^{-5}$
		\\
		\hline $7$ & $2.0761\cdot10^{-4}$ & $1.2050\cdot10^{-4}$ & $6.9394\cdot10^{-5}$ & $3.9732\cdot10^{-5}$ & $2.2651\cdot10^{-5}$
		\\
		\hline $8$ & $1.6105\cdot10^{-4}$ & $9.3235\cdot10^{-5}$ & $5.3579\cdot10^{-5}$ & $3.0626\cdot10^{-5}$ & $1.7437\cdot10^{-5}$
		\\
		\hline $9$ & $1.2838\cdot10^{-4}$ & $7.4195\cdot10^{-5}$ & $4.2578\cdot10^{-5}$ & $2.4310\cdot10^{-5}$ & $1.3828\cdot10^{-5}$
		\\
		\hline $10$ & $1.0465\cdot10^{-4}$ & $6.0406\cdot10^{-5}$ & $3.4630\cdot10^{-5}$ & $1.9757\cdot10^{-5}$ & $1.1231\cdot10^{-5}$
		\\
		\hline $11$ & $8.6891\cdot10^{-5}$ & $5.0112\cdot10^{-5}$ & $2.8707\cdot10^{-5}$ & $1.6368\cdot10^{-5}$ & $9.3001\cdot10^{-6}$
		\\
		\hline $12$ & $7.3271\cdot10^{-5}$ & $4.2229\cdot10^{-5}$ & $2.4178\cdot10^{-5}$ & $1.3780\cdot10^{-5}$ & $7.8265\cdot10^{-6}$
		\\
		\hline $13$ & $6.2605\cdot10^{-5}$ & $3.6063\cdot10^{-5}$ & $2.0639\cdot10^{-5}$ & $1.1758\cdot10^{-5}$ & $6.6766\cdot10^{-6}$
		\\
		\hline $14$ & $5.4098\cdot10^{-5}$ & $3.1151\cdot10^{-5}$ & $1.7822\cdot10^{-5}$ & $1.0150\cdot10^{-5}$ & $5.7623\cdot10^{-6}$
		\\
		\hline $15$ & $4.7209\cdot10^{-5}$ & $2.7175\cdot10^{-5}$ & $1.5543\cdot10^{-5}$ & $8.8505\cdot10^{-6}$ & $5.0234\cdot10^{-6}$
		\\
		\hline $16$ & $4.1551\cdot10^{-5}$ & $2.3912\cdot10^{-5}$ & $1.3674\cdot10^{-5}$ & $7.7848\cdot10^{-6}$ & $4.4179\cdot10^{-6}$
		\\
		\hline $17$ & $3.6850\cdot10^{-5}$ & $2.1203\cdot10^{-5}$ & $1.2122\cdot10^{-5}$ & $6.9003\cdot10^{-6}$ & $3.9154\cdot10^{-6}$
		\\
		\hline $18$ & $3.2902\cdot10^{-5}$ & $1.8928\cdot10^{-5}$ & $1.0820\cdot10^{-5}$ & $6.1582\cdot10^{-6}$ & $3.4940\cdot10^{-6}$
		\\
		\hline $19$ & $2.9555\cdot10^{-5}$ & $1.7000\cdot10^{-5}$ & $9.7164\cdot10^{-6}$ & $5.5296\cdot10^{-6}$ & $3.1370\cdot10^{-6}$
		\\
		\hline $20$ & $2.6692\cdot10^{-5}$ & $1.5351\cdot10^{-5}$ & $8.7733\cdot10^{-6}$ & $4.9924\cdot10^{-6}$ & $2.8320\cdot10^{-6}$
		\\
		\hline
	\end{tabular}
	\caption{$|c_M - \mean(X_{num,3}(2\pi/M^2)) / (2\pi/M^2)|$, where $c_M$ is given by \eqref{e:cM}. As in \cite{HozVega2014}, we have taken $N_t = 151200\cdot4^r$ time steps, where $N/M = 512\cdot2^r$; hence, $\Delta t = (2\pi/M^2)/N_t$. For a given $M$, the results suggest a first-order convergence rate from $\mean(X_{num,3}(2\pi/M^2)) / (2\pi/M^2)$ to $c_M$ in \eqref{e:cM}. In general, the results improve also as $M$ grows. In the worst case, i.e., when $M = 3$, i.e., we have considered also $N/M = 16384$, obtaining an improved error of $6.1416\cdot10^{-5}$ in the approximation of $c_M$.}\label{t:errorcM}
\end{table}

In our opinion, the results in Table \ref{t:errorcM} give additional support to the validity of \eqref{e:cM} and, hence, to the validity of the hypotheses upon which Theorem \ref{teo:teo} was based. In order to give concluding evidence that \eqref{e:cM} is correct, we will reobtain it in the following section, in a completely unrelated way.

\section{Algebraic computation of $c_M$}

\label{s:algebraic}

In this section, we deduce \eqref{e:cM} using algebraic means. Unlike in Section \ref{s:analytical}, where we worked mainly with the one-problem corner, we work here exclusively with the $M$-corner problem. Therefore, in order not to burden the notation, we omit the subscript $M$ in $\X_M$, because there is no risk of confusion.

Our starting point is given again by \eqref{e:mean} and \eqref{e:ht}. However, unlike in Section \ref{s:analytical}, where, in order to relate the calculation of $c_M$ with the one-corner problem, we integrated the third component of \eqref{e:mean} over one $M$-th of the period, we consider the whole period here. The numerical simulations clearly shows that
\begin{equation*}
\mean(X_{3,t})(t) \not= [\mean(X_3)]'(t),
\end{equation*}

\noindent i.e.,
\begin{equation}
\label{e:intX_3cM}
\frac{1}{2\pi}\int_0^{2\pi}X_{3,t}(s,t)ds\not=\frac{d}{dt}\left(\frac{1}{2\pi}\int_0^{2\pi}X_3(s,t)ds\right) = h'(t) = c_M;
\end{equation}

\noindent in fact, the left-hand side is not only nonconstant, but also very singular (later on, in \eqref{e:intX3tpq}, we will give an explicit expression of it, at times of the form $t = t_{pq}$). However, we do compute numerically $h(t)$ precisely as
\begin{equation}
\label{e:h(t)}
h(t) = \int_0^t\mean(X_{3,t})(t')dt' = \int_0^t\left[\frac{1}{2\pi}\int_0^{2\pi}X_{3,t}(s,t')ds\right]dt'.
\end{equation}

\noindent Therefore, denoting $\ez = (0, 0, 1)^T$, we have to give sense to
\begin{equation}
\label{e:x3t}
\int_0^{2\pi}X_{3,t}(s,t)ds = \int_0^{2\pi}(\Xt(s, t)\cdot\ez)ds = \left[\int_0^{2\pi}\T(s, t)\wedge\Ts(s, t)ds\right]\cdot\ez,
\end{equation}

\noindent at times of the form $t = t_{pq}$. Since the cross product $\wedge$ is rotation invariant, let us suppose, without loss of generality, that
\begin{equation*}
\T(s) =
\begin{cases}
(1, 0, 0)^T, & s < 0,
\cr
(\cos(\rho), \sin(\rho), 0)^T, & s > 0,
\end{cases}
\end{equation*}

\noindent i.e., there is a rotation of angle $\rho$ at $s = 0$, where $\rho$ is given by \eqref{e:cosrho}. Then,
$$
\int_{-\infty}^{+\infty}\T(s)\wedge\Ts(s)ds = \int_{0^-}^{0^+}\T(s)\wedge\Ts(s)ds = \rho(0,0,1)^T = \frac{\rho}{\sin(\rho)}\T(0^-)\wedge\T(0^+).
$$

\noindent Therefore, at a time $t_{pq}$, we have
\begin{equation*}
\int_0^{2\pi}X_{3,t}(s,t_{pq})ds = \left[\frac{\rho}{\sin(\rho)} \sum_j\T_j\wedge\T_{j + 1}\right]\cdot\ez,
\end{equation*}

\noindent where $\T_j$ are the tangent vectors of the $Mq$ or $Mq/2$ sides of the corresponding skew polygon, with $j\in\{0,Mq-1\}$, for $q$ odd, or $j\in\{0,Mq/2-1\}$, for $q$ even. In fact, since $\int_0^{2\pi}\Xt(s,t_{pq})ds$ points already in the $\ez$ direction,
\begin{equation*}
\int_0^{2\pi}\Xt(s,t_{pq})ds = \frac{\rho}{\sin(\rho)} \sum_j\T_j\wedge\T_{j + 1}.
\end{equation*}

\noindent Observe that the last equation is equivalent to writing
\begin{equation*}
\int_0^{2\pi}\Xt(s,t_{pq})ds = \frac{\rho}{\sin(\rho)} \Delta s\sum_j\T_j\wedge\frac{\T_{j + 1} - \T_j}{\Delta s};
\end{equation*}

\noindent hence, formally,
\begin{equation*}
\lim_{q\to\infty}\left[\frac{\rho}{\sin(\rho)} \Delta s\sum_j\T_j\wedge\frac{\T_{j + 1} - \T_j}{\Delta s}\right] = \int_0^{2\pi}\T(s)\wedge\Ts(s)ds.
\end{equation*}

\noindent Coming back to \eqref{e:x3t}, we have used symbolic manipulation exactly as in \cite{HozVega2014}, in order to compute $\sum_j\T_j\wedge\T_{j + 1} \equiv \sum_j\T_{alg,j}\wedge\T_{alg,j + 1}$. After considering a few small values of $q$, the logic becomes apparent and we can conjecture that, for all $p$ coprime with $q$,
\begin{equation*}
\left[\sum_j\T_{alg,j}\wedge\T_{alg,j + 1}\right]\cdot\ez =
\begin{cases}
\dfrac{(1 - \cos(\rho)) M q}{\tan(\pi / M)}, & \mbox{if $q\equiv1\bmod2$},
\cr
\dfrac{(1 - \cos(\rho)) M (q/2)}{\tan(\pi / M)}, & \mbox{if $q\equiv0\bmod2$},
\end{cases}
\end{equation*}

\noindent from which the value of the left-hand side of \eqref{e:intX_3cM} follows:
\begin{equation}
\label{e:intX3tpq}
\frac{1}{2\pi}\int_0^{2\pi}X_{3,t}(s,t_{pq})ds =
\begin{cases}
\dfrac{1}{2\pi}\dfrac{\rho}{\sin(\rho)}\dfrac{(1 - \cos(\rho)) M q}{\tan(\pi / M)}, & \mbox{if $q\equiv1\bmod2$},
\cr
\dfrac{1}{2\pi}\dfrac{\rho}{\sin(\rho)}\dfrac{(1 - \cos(\rho)) M (q/2)}{\tan(\pi / M)}, & \mbox{if $q\equiv0\bmod2$}.
\end{cases}
\end{equation}

\noindent On the other hand, when $p\equiv 0\bmod q$ (and also at $t = t_{12}$), we have a planar polygon, so
\begin{equation}
\label{e:intX3t0q}
\frac{1}{2\pi}\int_0^{2\pi}X_{3,t}(s,t_{pq})ds = 1.
\end{equation}

\noindent Once understood $\int_0^{2\pi}X_{3,t}(s,t_{pq})ds$, we are going to be able to calculate $c_M$. Let us assume without loss of generality that $q$ is an odd prime. Then, from \eqref{e:h(t)},
\begin{align*}
h(2\pi/M^2) & = \int_0^{2\pi/M^2}\left[\frac{1}{2\pi}\int_0^{2\pi}X_{3,t}(s,t')ds\right]dt'
\cr
& = \lim_{q\to\infty}\left[\frac{2\pi}{M^2q}\frac{1}{2\pi}\left(\frac{1}{2}\int_0^{2\pi}X_{3,t}(s,t_{0q})ds + \sum_{p=1}^{q-1}\int_0^{2\pi}X_{3,t}(s,t_{pq})ds + \frac{1}{2}\int_0^{2\pi}X_{3,t}(s,t_{qq})ds\right)\right]
\cr
& = \lim_{q\to\infty}\left[\frac{2\pi}{M^2q}\left(\frac{1}{2} + \sum_{p=1}^{q-1}\left(
\dfrac{1}{2\pi}\dfrac{\rho}{\sin(\rho)}\dfrac{(1 - \cos(\rho)) M q}{\tan(\pi / M)}\right) + \frac{1}{2}\right)\right],
\cr
& = \lim_{q\to\infty}\left[\frac{1}{M^2}\dfrac{\rho}{\sin(\rho)}\dfrac{(1 - \cos(\rho)) M q}{\tan(\pi / M)}\right],
\end{align*}

\noindent where we have approximated the integral over $t$ by means of the trapezoidal rule, and taken the limit $q\to\infty$. On the other hand, $\lim_{q\to\infty}\rho = 0$, so $\lim_{q\to\infty}\rho / \sin(\rho) = 1$, and we conclude that
\begin{align}
\label{e:htqq}
h(2\pi/M^2) & = \frac{1}{M\tan(\pi/M)}\lim_{q\to\infty}[(1 - \cos(\rho)) q]
\cr
& = \frac{1}{M\tan(\pi/M)}\lim_{q\to\infty}[(1 - (2\cos^{2/q}(\pi/M) - 1)) q]
\cr
& = \frac{-4\ln(\cos(\pi/M))}{M\tan(\pi/M)}.
\end{align}

\noindent Finally, bearing in mind that $h(2\pi/M^2) = c_M (2\pi/M^2)$, we recover \eqref{e:cM}:
\begin{equation}
\label{e:c_Malg}
c_M = \frac{h(2\pi/M^2)}{2\pi/M^2} = \frac{-2\ln(\cos(\pi/M))}{(\pi/M)\tan(\pi/M)} = \frac{\ln(1 + \tan^2(\pi/M))}{(\pi/M)\tan(\pi/M)}.
\end{equation}

\noindent A similar reasoning would show that $h(t_{pq})= c_M t_{pq}$, for any time $t_{pq}$.

\section{Transfer of energy}

\label{s:energy}

Recall that, from \eqref{energy1} and \eqref{energy2}, $|\X_t(s,t)|^2ds $ and $|\T_s(s,t)|^2ds $ are the natural energy densities associated to the solutions of \eqref{e:xt} and \eqref{e:schmap}, respectively. We know from Lemma \ref{basiclemma} and \eqref{energy3} that, for the selfsimilar solution at time $t=1$, while $|\T_s|$ remains bounded, there are some components of $\X_t$ that have a logarithmic growth. On the other hand, in \cite{BanicaVega2016}, the behavior of
$$
\|\widehat{\T_s}(t)\|_{\infty}=\sup_\xi|\xi\,\widehat{\T}(\xi, t)|
$$

\noindent is studied for solutions of the one-corner problem and small regular perturbations of it. In particular, in Theorem 1.1 of that paper, a lack of continuity of this quantity is proved at the time where the corner appears. Let us assume that this fact happens at the initial time $t=0$. Then, if $0<t$, the following result is proved:
\begin{equation}
\label{e:continuity}\|\widehat{\T_s}(0)\|_{\infty}<\|\widehat{\T_s}(t)\|_{\infty}.
\end{equation}

\noindent Moreover, in the proof of the above result, it is observed that, while the maximum of the left-hand side of \eqref{e:continuity} is taken at small frequencies, the maximum at the right-hand side is achieved at large frequencies. Motivated again by the results obtained in the previous sections, it is also very natural to ask which would be the behavior of
$$
\|\widehat{\T_{M,s}}(\tpq)\|_{\infty} = \max_k|k\, \widehat{\T_{M}}(k, \tpq)|,
$$
\noindent in the case of a regular polygon with $M$ sides. Note that, as usually, $\X_M$, $\T_M$ and related quantities will be computed algebraically.

In order to compute accurately $\|\widehat{\T_{M,s}}(\tpq)\|_{\infty}$, we will need a couple of lemmas. During the rest of this section, all the appearances of $\X$ and $\T$ will refer exclusively to the $M$-corner problem, so we will omit the subscript $_M$, in order not to burden the notation.

\begin{lema}
\label{lema:periodick}
Let be $P(s)$ a real $2\pi$-periodic function, such that, for some $N\in\mathbb N$, some $\delta\in\mathbb R$, and for all $j\in\mathbb Z$, it is piecewise constant, with value $P(s)\equiv P_j$ at $s\in(s_j, s_{j+1})$, where $s_j = 2\pi j / N - \delta$. Then, $|\widehat{P_s}(k)| = |k\,\widehat P(k)|$ is $N$-periodic in $k$.

\end{lema}

\begin{proof}
\begin{align*}
\widehat{P_s}(k) & = ik\,\widehat P(k) = \frac{ik}{2\pi}\int_{-\delta}^{2\pi - \delta}P(s)e^{-iks}ds = \frac{ik}{2\pi}\sum_{j=0}^{N-1}\int_{s_j}^{s_{j+1}}P_je^{-iks}ds 
	\cr
& = -\frac{1}{2\pi}\sum_{j=0}^{N-1}P_j(e^{-iks_{j+1}} - e^{-iks_j}) = e^{ik\delta}\frac{1 - e^{-2\pi ik/N}}{2\pi}\sum_{j=0}^{N-1}P_je^{-2\pi ijk/N}.
\end{align*}

\noindent Hence,
\begin{align*}
|\widehat{P_s}(k)| = \left|\frac{\sin(\pi k / N)}{\pi}\sum_{j=0}^{N-1}P_je^{-2\pi ijk/N}\right|,
\end{align*}

\noindent so we conclude that
$$
|\widehat{P_s}(k+N)| = \left|\frac{\sin(\pi (k+N) / N)}{\pi}\sum_{j=0}^{N-1}P_je^{-2\pi ij(k+N)/N}\right| = |\widehat{P_s}(k)|.
$$

\end{proof}	

\begin{coro}
\label{coro:khatT}
Let $\T$ be the solution of \eqref{e:schmap} for the $M$-corner problem, and let $T(s)$ denote any component of $\T(s, t_{pq})$. Then, $|\widehat{T_s}(k)| = |k\,\widehat T(k)|$ is $Mq$ periodic in $k$, for $q$ odd, and $Mq/2$ periodic in $k$, for $q$ even.
\end{coro}

In fact, looking at the proof of Lemma \ref{lema:periodick}, when $q$ is odd, $\delta = 0$, so $k\,\widehat T(k)$ is $Mq$ periodic in $k$; and when $q\equiv 0\bmod4$, $\delta = 0$ again, so $k\,\widehat T(k)$ is $Mq/2$ periodic in $k$. However, when $q\equiv2\bmod4$, $\delta = \pi / N$, so $k\,\widehat T(k)$ is only $Mq$ periodic in $k$. In general, since we are interested in $|k\,\widehat T(k)|$ and not in $k\,\widehat T(k)$ itself, we can discretize $s\in[0,2\pi)$ systematically at $s_j = 2\pi j / (Mq)$, for $q$ odd; and at $s_j = 4\pi j / (Mq)$, for $q$ even. Then, we take $T_j\equiv T(s_j^+)$ in all cases, and the following equality always holds:
\begin{equation}
\label{e:hatTsk}
|\widehat{T_s}(k)| = \left|\frac{\sin(\pi k / N)}{\pi}\sum_{j=0}^{N-1}T(s_j^+)e^{-2\pi ijk/N}\right|.
\end{equation}

\begin{lema}
\label{lemma:T1sT2s}
Let $(T_1, T_2, T_3)$ be the components of $\T(s, \tpq)$. Let us define $Z \equiv T_1 + iT_2$, and denote $T_{1,j} \equiv T_1(s_j^+)$, $T_{2,j} \equiv T_2(s_j^+)$, $T_{3,j} \equiv T_3(s_j^+)$, $Z_j \equiv T_{1,j} + iT_{2,j}\equiv Z(s_j^+)$, $s_j = 2\pi j/N$, $j\in\{0,\ldots,N-1\}$, where $N = Mq$, for $q$ odd; and $N = Mq/2$, for $q$ even. Then
\begin{equation}
\label{e:normhatTsk}
|\widehat{\Ts}(k)| = \left(\sum_{l=1}^3|\widehat{T_{l,s}}(k)|^2\right)^{1/2} =
\begin{cases}
\left|\dfrac{M\sin(\pi k / N)}{\sqrt2\pi}\displaystyle{\sum_{j = 0}^{N/M - 1}}Z_je^{-2\pi i jk/N}\right|, & \mbox{if $k \equiv \pm1 \bmod M$},
\\[2em]
\left|\dfrac{M\sin(\pi k / N)}{\pi}\displaystyle{\sum_{j = 0}^{N/M - 1}}T_{3,j}e^{-2\pi i jk/N}\right|, & \mbox{if $k \equiv 0 \bmod M$},
\\[1em]
0, & \mbox{otherwise}.
\end{cases}
\end{equation}

\end{lema}

\begin{proof}

Because of the symmetries, $Z_{j+N/M} = Z_je^{2\pi i/M}$, so
$$
\sum_{j = 0}^{N - 1}Z_je^{-2\pi i jk/N} = \left(\sum_{l=0}^{M-1}e^{2\pi i(1-k)l/M}\right)\sum_{j = 0}^{N/M - 1}Z_je^{-2\pi i jk/N}
=
\begin{cases}
M\displaystyle{\sum_{j = 0}^{N/M - 1}}Z_je^{-2\pi i jk/N}, & \mbox{if $k \equiv 1 \bmod M$},
\\
0, & \mbox{if $k \not\equiv 1 \bmod M$}.
\end{cases}
$$

\noindent On the other hand,
$$
\sum_{j = 0}^{N - 1}\bar Z_je^{-2\pi i jk/N} = \overline{\displaystyle{\sum_{j = 0}^{N - 1}} Z_je^{-2\pi i j(-k)/N}}
=
\begin{cases}
M\displaystyle{\overline{\sum_{j = 0}^{N/M - 1}Z(s_j)e^{-2\pi i j(-k)/N}}}, & \mbox{if $k \equiv -1 \bmod M$},
\\
0, & \mbox{if $k \not\equiv -1 \bmod M$}.
\end{cases}
$$

\noindent Therefore,
$$
\sum_{j = 0}^{N - 1}T_{1,k}e^{-2\pi i jk/N} = \displaystyle{\sum_{j = 0}^{N - 1}}\frac{Z_j + \bar Z_j}{2}e^{-2\pi i jk/N} =
\begin{cases}
\dfrac{M}2\displaystyle{\sum_{j = 0}^{N/M - 1}}Z_je^{-2\pi i jk/N}, & \mbox{if $k \equiv 1 \bmod M$},
\\[1.5em]
\dfrac{M}2\displaystyle{\overline{\sum_{j = 0}^{N/M - 1}Z_je^{-2\pi i j(-k)/N}}}, & \mbox{if $k \equiv -1 \bmod M$},	
\\[1em]
0, & \mbox{otherwise};
\end{cases}
$$

\noindent and
$$
\sum_{j = 0}^{N - 1}T_{2,k}e^{-2\pi i jk/N} = \displaystyle{\sum_{j = 0}^{N - 1}}\frac{Z_j - \bar Z_j}{2i}e^{-2\pi i jk/N} =
\begin{cases}
\dfrac{M}{2i}\displaystyle{\sum_{j = 0}^{N/M - 1}}Z_je^{-2\pi i jk/N}, & \mbox{if $k \equiv 1 \bmod M$},
\\[1.5em]
-\dfrac{M}{2i}\displaystyle{\overline{\sum_{j = 0}^{N/M - 1}Z_je^{-2\pi i j(-k)/N}}}, & \mbox{if $k \equiv -1 \bmod M$},	
\\[1em]
0, & \mbox{otherwise}.
\end{cases}
$$

\noindent Finally, from \eqref{e:hatTsk},
\begin{equation}
\label{e:hatT1skhatT2sk}
|\widehat{T_{1,s}}(k)| = |\widehat{T_{2,s}}(k)| =
\begin{cases}
\left|\dfrac{M\sin(\pi k / N)}{2\pi}\displaystyle{\sum_{j = 0}^{N/M - 1}}Z_je^{-2\pi i jk/N}\right|, & \mbox{if $k \equiv \pm1 \bmod M$},
\\[1em]
0, & \mbox{otherwise}.
\end{cases}
\end{equation}

\noindent With respect to the third component, the symmetries imply $T_{3,j+N/M} = T_{3,j}$, so reasoning similarly as before,
\begin{equation}
\label{e:hatT3sk}
|\widehat{T_{3,s}}(k)| = 
\begin{cases}
\left|\dfrac{M\sin(\pi k / N)}{\pi}\displaystyle{\sum_{j = 0}^{N/M - 1}}T_{3,j}e^{-2\pi i jk/N}\right|, & \mbox{if $k \equiv 0 \bmod M$},
\\[1em]
0, & \mbox{otherwise}.
\end{cases}
\end{equation}

\noindent Therefore, introducing \eqref{e:hatT1skhatT2sk} and \eqref{e:hatT3sk} into $|\widehat{\Ts}(k)| = (\sum_{l=1}^3|\widehat{T_{l,s}}(k)|^2)^{1/2}$, we obtain \eqref{e:normhatTsk}, which completes the proof.

\end{proof}

From Corollary \eqref{coro:khatT}, $|\widehat{\Ts}(k)|$ is $k$-periodic with period $Mq$ (for $q$ odd) or $Mq/2$ (for $q$ even). Furthermore, using symmetry arguments as in \cite{HozVega2014}, the computation of $|\widehat{\Ts}(k)|$ is basically reduced to two discrete Fourier transforms of $q$ elements. For instance, when $q$ is odd:
\begin{equation}
\label{e:TsMk1}
|\widehat{\Ts}(Mk + 1)|2 = |\widehat{\Ts}(q - (Mk + 1))| = \left|\dfrac{M\sin(\pi(k/q+1/N))}{\sqrt2\pi}\displaystyle{\sum_{j = 0}^{q - 1}}[e^{-2\pi ij/N}Z_j]e^{-2\pi i jk/q}\right|,
\end{equation}

\noindent and
\begin{equation}
\label{e:TsMk}
|\widehat{\Ts}(Mk)| = \left|\dfrac{M\sin(\pi k / q)}{\pi}\displaystyle{\sum_{j = 0}^{q - 1}}T_{3,j}e^{-2\pi i jk/q}\right|,
\end{equation}

\noindent for $k\in\{0, \ldots, q-1\}$. When $q$ is even, it is enough to substitute $q$ by $q/2$ in \eqref{e:TsMk1} and \eqref{e:TsMk}.

From now on, because of \eqref{e:hatT1skhatT2sk}, we will refer only to the first and third components of $\widehat{\Ts}(k)$, but everything said about the first component will be automatically valid for the second one. All the numerical experiments in this section correspond to $M = 3$. In order to study the growth of the maximum of \eqref{e:normhatTsk}, we have taken systematically $q$ equal to a prime number multiplied by two. Indeed, the conclusions that we draw in this section seem to be independent from the choice of $q$. Therefore, by chosing the double of a prime number, we ensure consistency in the comparisons for the different values of $p$, because, for any $\tpq$ with $p\not\equiv0\bmod q$, $\X(s, \tpq)$ has the same number of sides, i.e., $Mq/2$.

In Figure \ref{f:maxhatTsq200006}, we have plotted, as functions of $\tpq$, the maximum of \eqref{e:TsMk1}, i.e., $\max_k|\widehat{\Ts}(Mk+1, \tpq)| = \sqrt2\max_k|\widehat{T_{1,s}}(k, \tpq)|$  (left-hand side); the maximum of \eqref{e:TsMk}, i.e., $\max_k|\widehat{\Ts}(Mk, \tpq)| = \max_k|\widehat{T_{3,s}}(k, \tpq)|$ (center); and the maximum of \eqref{e:TsMk1} and \eqref{e:TsMk}, which is precisely $\|\widehat{\Ts}(t_{pq})\|_\infty$ (right-hand side). In all cases, $\tpq$ is such that $q = 2\times 100003 = 200006$, and $p\in\{1,\ldots,\lfloor q/4\rfloor\}$, so $p/q\in(0,1/4)$. Indeed, because of the symmetries of the problem, $\|\widehat{\Ts}(\pm t)\|_\infty = \|\widehat{\Ts}(t_{1,4}\pm t)\|_\infty$. From now on, we will use systematically $\|\cdot\|_\infty$, in order to denote the maximum of $|\cdot|$ over all $k$.

\begin{figure}[!htb]
	\centering
	\includegraphics[width=0.325\textwidth, clip=true]{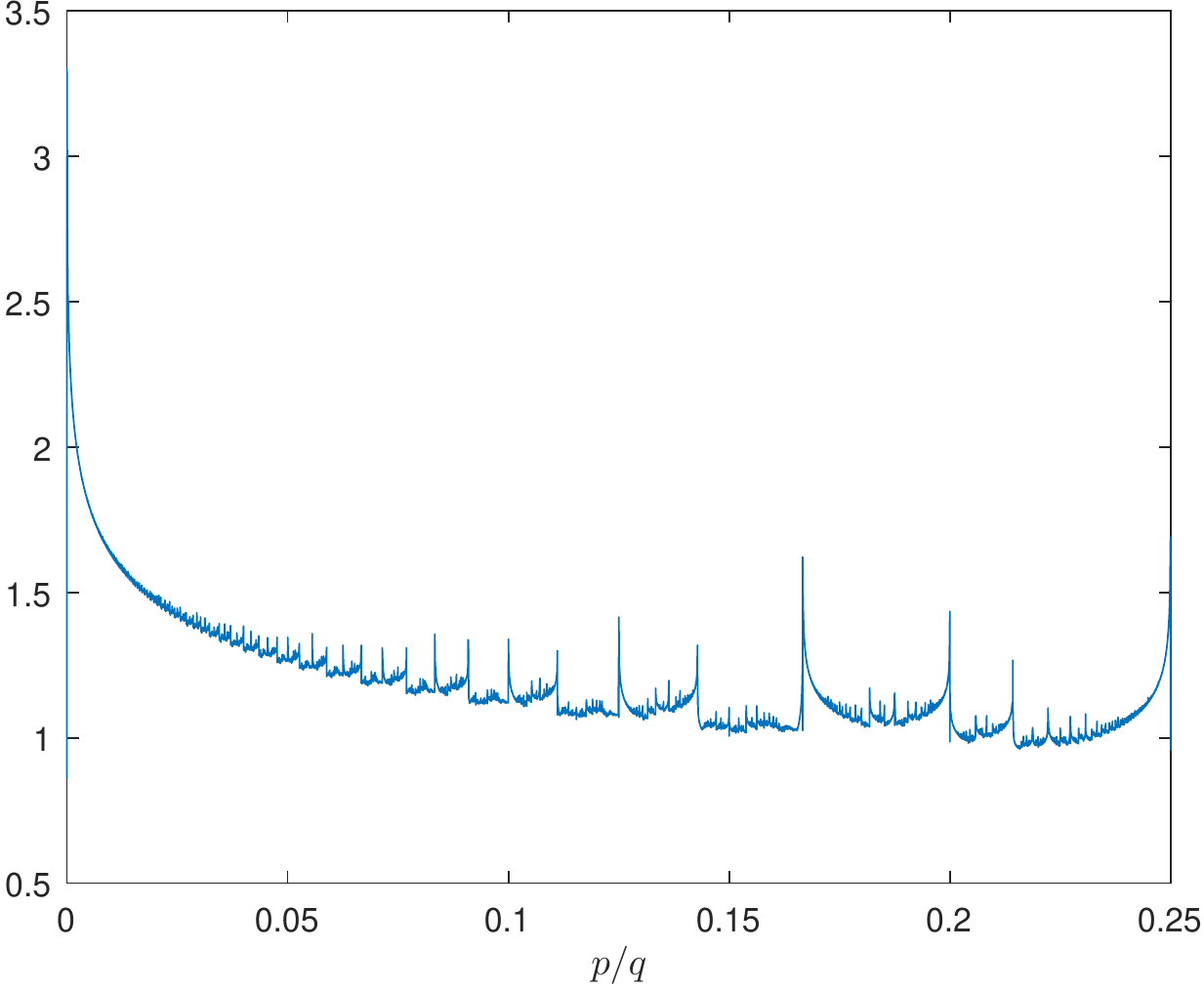}
	\includegraphics[width=0.325\textwidth, clip=true]{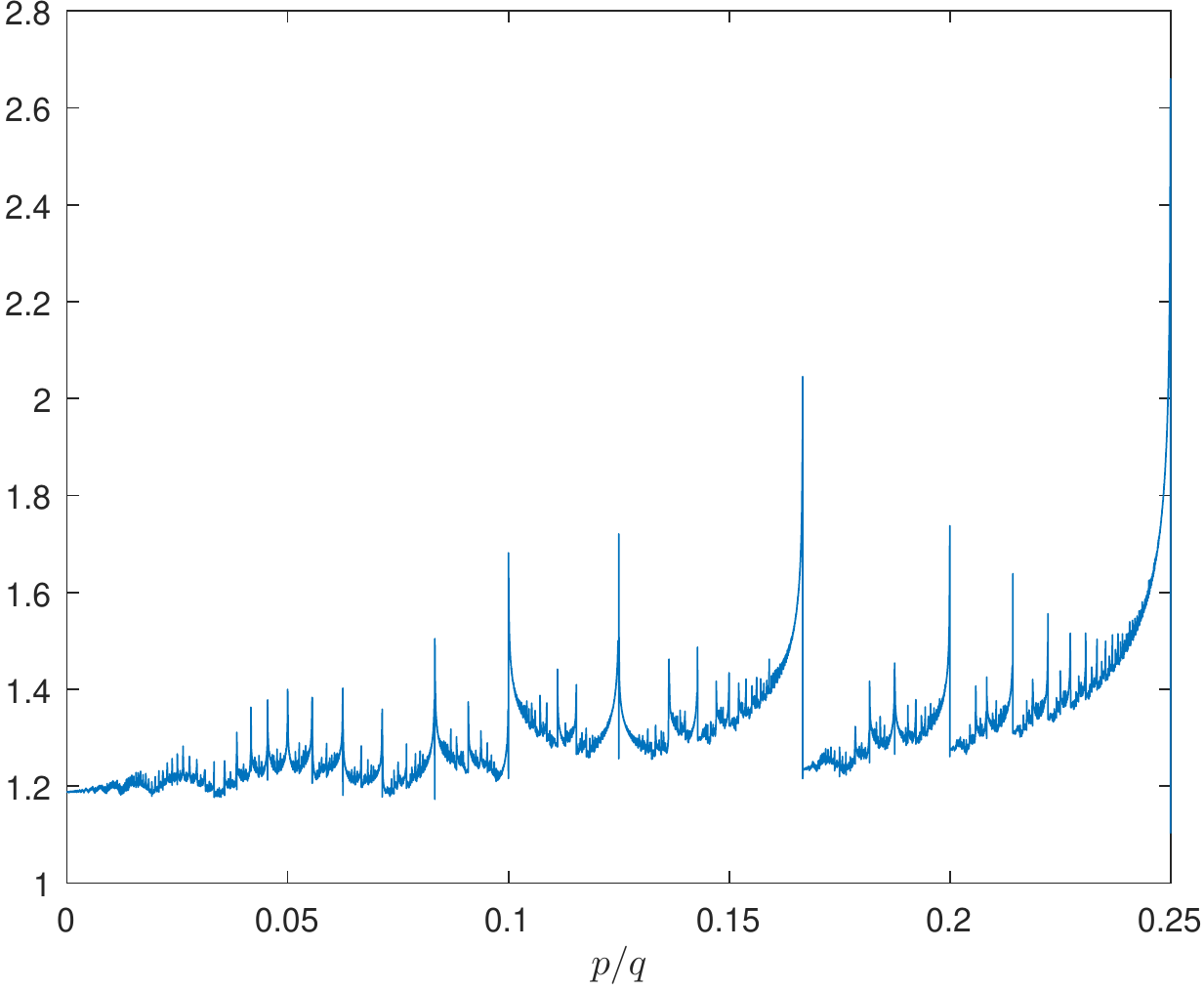}
	\includegraphics[width=0.325\textwidth, clip=true]{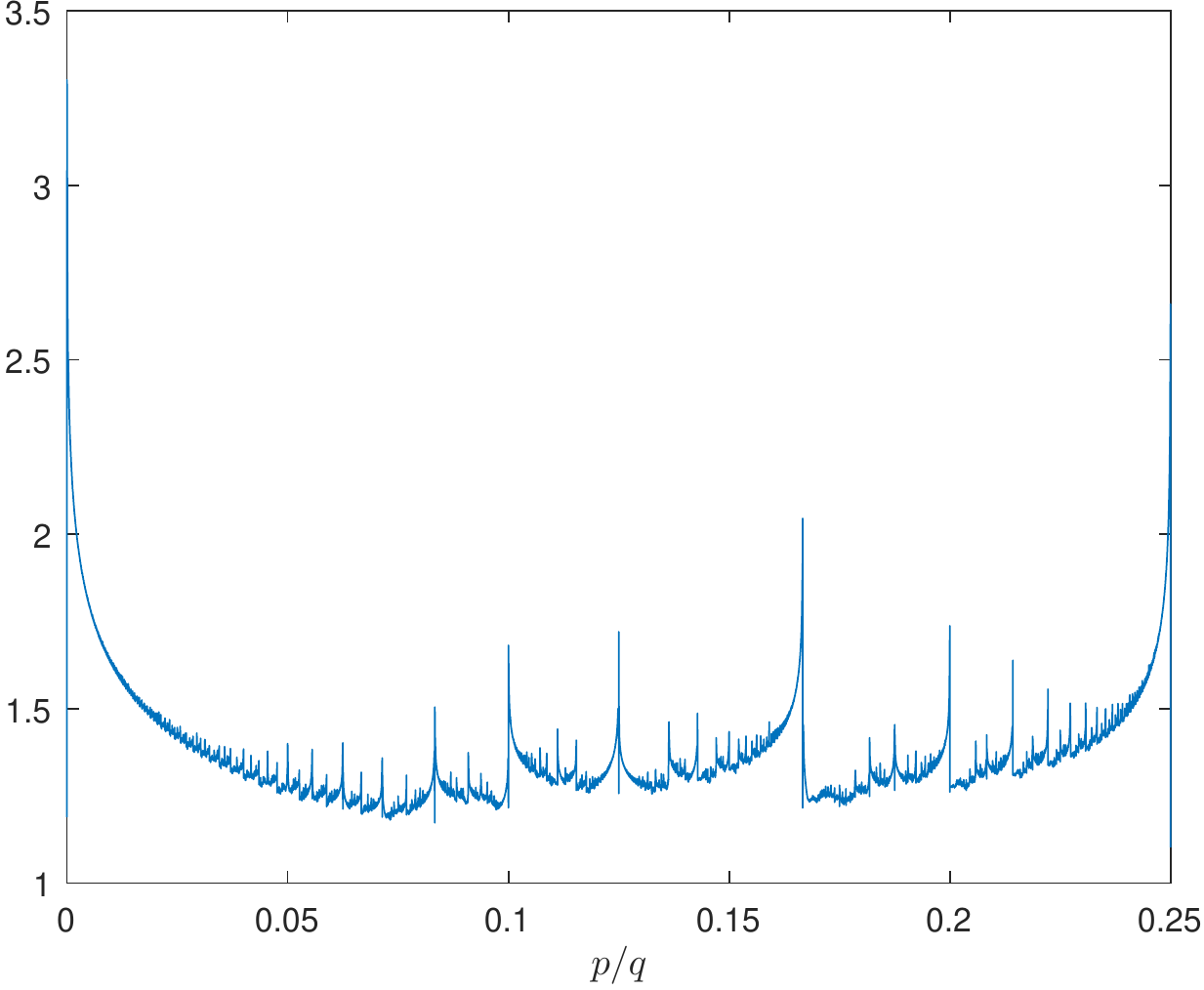}
	\caption{Left: maximum of \eqref{e:TsMk1}. Center: maximum of \eqref{e:TsMk}. Right: $\|\widehat{\Ts}(t_{pq})\|_\infty$, which is given precisely by taking the maximum of the other two curves at every $\tpq$. In all cases, $q = 2\times 100003 = 200006$, $p\in\{1,\ldots,50001\}$, so $p/q\in(0,1/4)$. Although the three curves appear to be discontinous everywhere, we have joined the consecutive points with line segments, so that the jumps are better appreciated.}\label{f:maxhatTsq200006}
\end{figure}

The curves in Figure \ref{f:maxhatTsq200006} appear to be discontinuous everywhere, and the jumps seem to be higher when $p/q$ is close to a rational $a/b$, $\gcd(a,b)=1$, with a smaller denominator $b$. Thus, we can identify immediately the rationals $1/4$, $1/5$, $1/6$, etc. On the other hand, the behavior of $\sqrt2\|\widehat{T_{1,s}}(\tpq)\|_\infty$ and $\|\widehat{T_{3,s}}(\tpq)\|_\infty$ is quite opposite; in the former case, the global maximum is reached as $\tpq$ approaches $0^+$; whereas, in the latter case, it is reached as $\tpq$ approaches $t_{1,4}^-$. Therefore, since both cases are of similar order of magnitude, and they are uncoupled from \eqref{e:normhatTsk}, we can study them separately.

We have analyzed $9591$ different values of $q$; more precisely, $q\in\{10, 14, \ldots, 199982, 200006\}$, where $q/2$ is a prime number. For each $q$, we have computed $\sqrt2\max_{\tpq}\|\widehat{T_{1,s}}(\tpq)\|_\infty$ and $\max_{\tpq}\|\widehat{T_{3,s}}(\tpq)\|_\infty$, where $\tpq$ is such that $p\in\{1, \ldots, \lfloor q/4\rfloor\}$. As in Figure \ref{f:maxhatTsq200006}, the global maximum of $\sqrt2\|\widehat{T_{1,s}}(\tpq)\|_\infty$ is always reached for $\tpq$ close to $0^+$; whereas the global maximum of $\|\widehat{T_{3,s}}(\tpq)\|_\infty$ is reached at some $\tpq$ close to $t_{1,4}^-$. Therefore, for each $q$, we can restrict the search of the global maxima to a small subset of all the possible values of $p$, which has a huge impact from a computational point of view. In fact, in the case of the first component, except for a few small values of $q$ (the largest one being $q = 158$), the maximum is always reached when $p = 2$, i.e., at $\tpq = t_{2,q}$. On the other hand, in the case of the third component, the maximum is reached when $p = \lfloor q/4\rfloor - 1$, $p = \lfloor q/4\rfloor - 2$, or $p = \lfloor q/4\rfloor - 3$, in practically all the cases.

In Figure \ref{f:globalmax}, we plot on the left-hand side $\sqrt2\max_{\tpq}\|\widehat{T_{1,s}}(\tpq)\|_\infty$ (in blue) and $\max_{\tpq}\|\widehat{T_{3,s}}(\tpq)\|_\infty$ (in red), as functions of $q$. The most surprising fact is that the values $\sqrt2\max_{\tpq}\|\widehat{T_{1,s}}(\tpq)\|_\infty$ form a single curve; whereas the values of $\max_{\tpq}\|\widehat{T_{3,s}}(\tpq)\|_\infty$ are distributed in four different curves. This much more complex character of the third component is due to the fact that $\X(s, 0)$ is planar and has only three corners, while $\X(s, t_{1,4})$ is skew and has six corners. On the other hand, all the curves have a marked logarithmic character. Indeed, in the center of Figure \ref{f:globalmax}, we have replotted the five curves in semilogarithmic scale. In the case of $\sqrt2\max_{\tpq}\|\widehat{T_{1,s}}(\tpq)\|_\infty$, except for the smallest values of $q$, the points are aligned forming what resembles a sharp straight line; whereas, in the case of  $\max_{\tpq}\|\widehat{T_{3,s}}(\tpq)\|_\infty$, we have what resembles four sharp straight lines, at least from $q\approx 10^4$ (in the worst case). 

\begin{figure}[!htb]
	\centering
	\includegraphics[width=0.325\textwidth, clip=true]{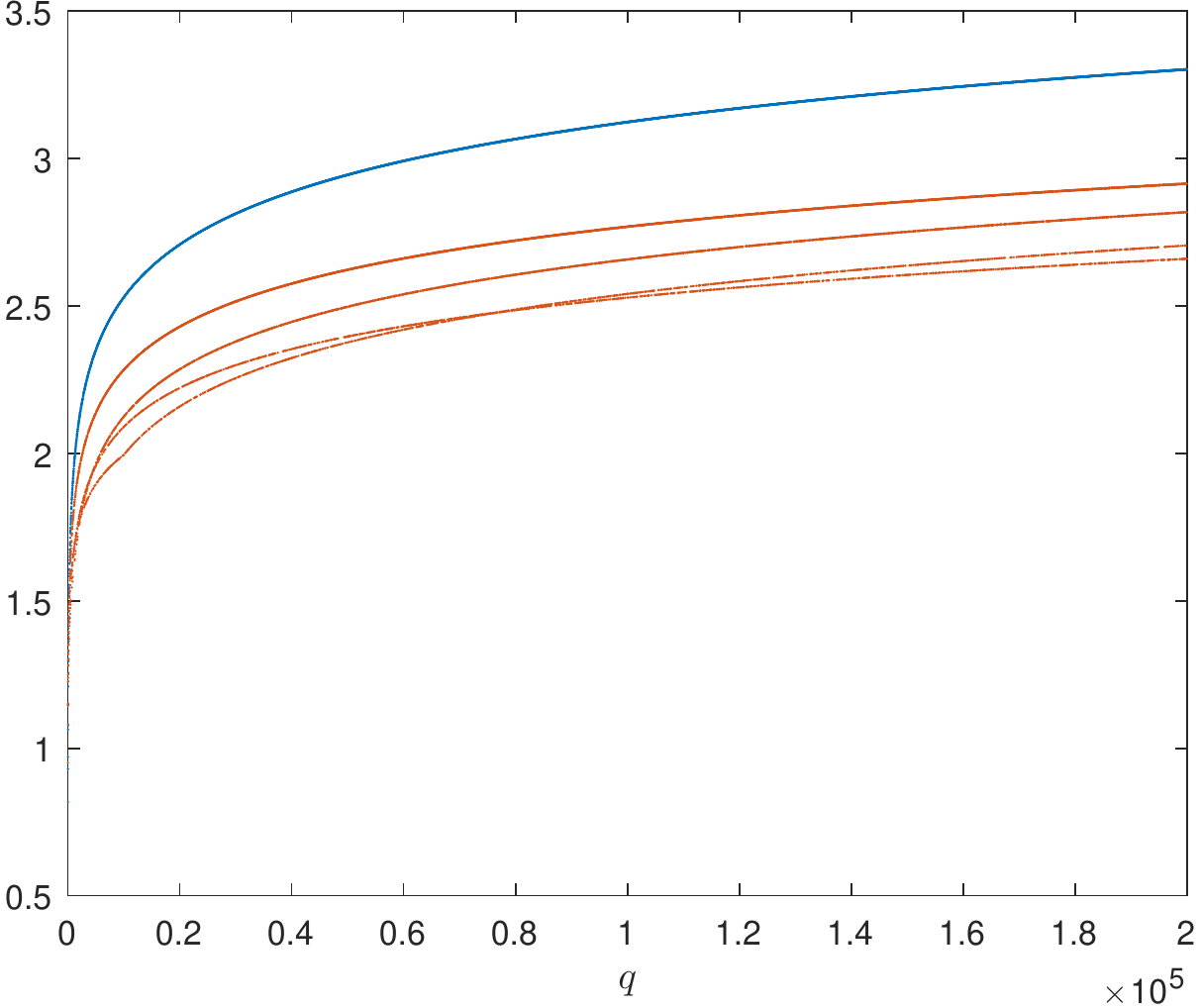}
	\includegraphics[width=0.325\textwidth, clip=true]{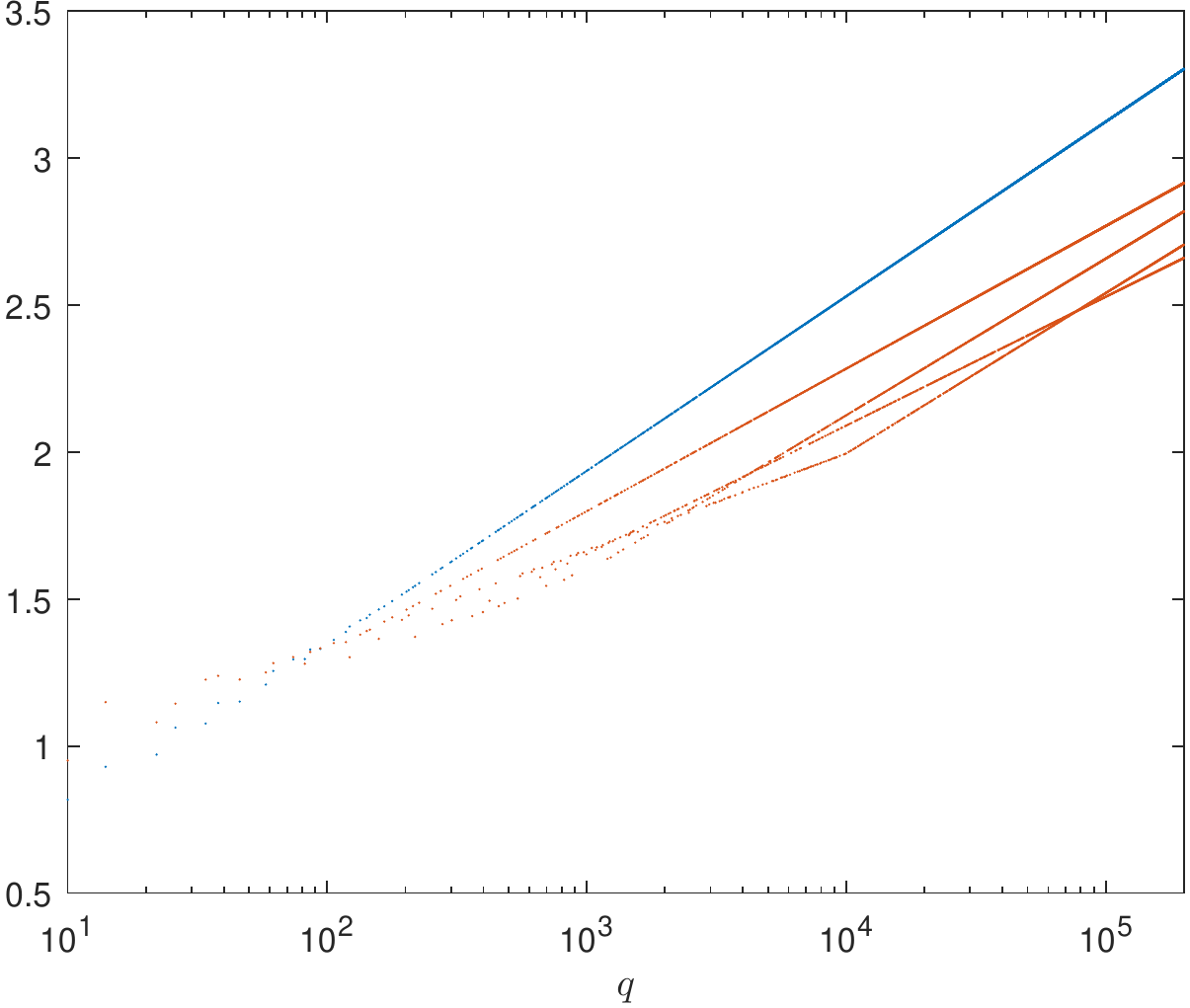}
	\includegraphics[width=0.325\textwidth, clip=true]{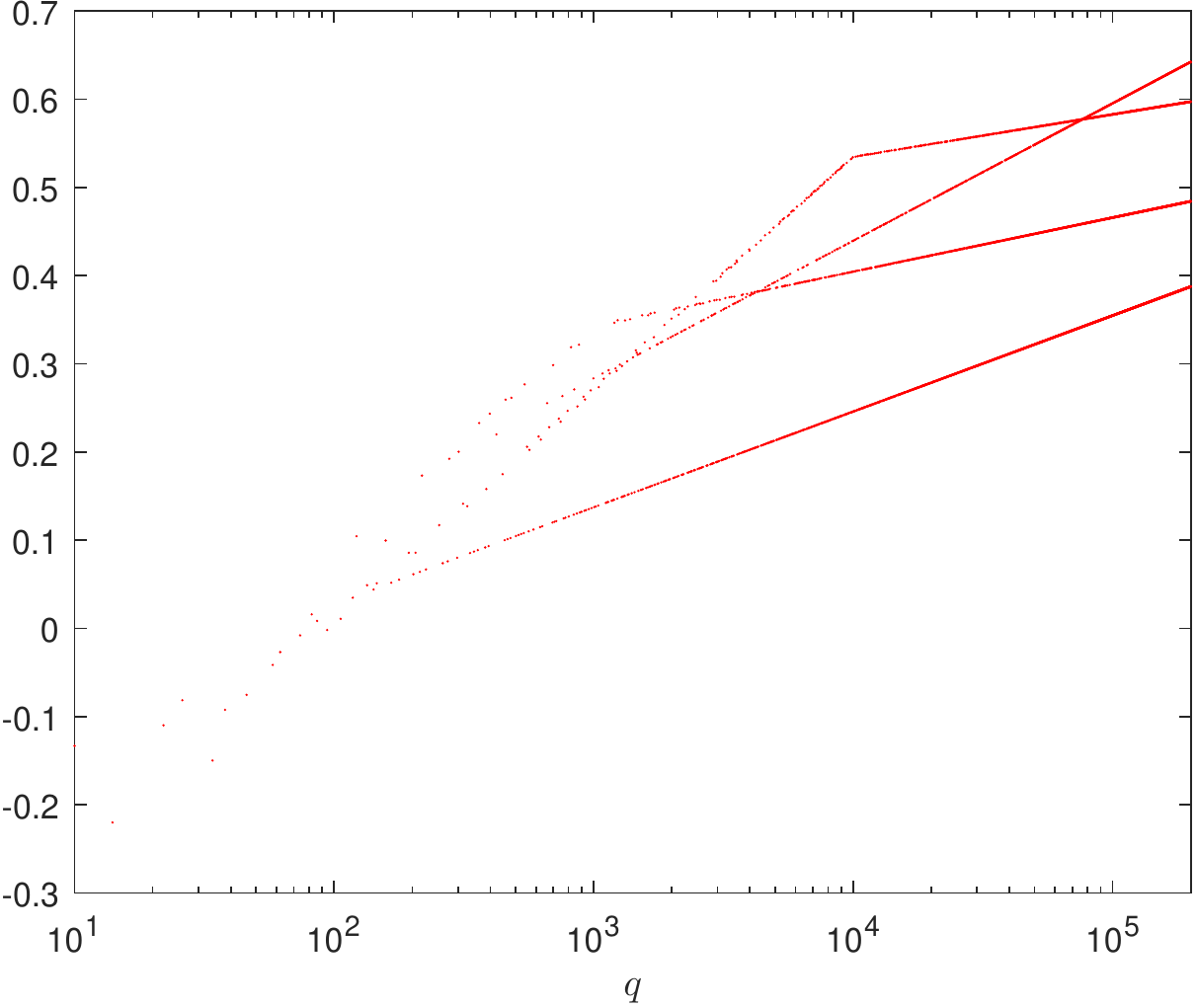}
	\caption{Left: $\sqrt2\max_{\tpq}\|\widehat{T_{1,s}}(\tpq)\|_\infty$ (in blue) and $\max_{\tpq}\|\widehat{T_{3,s}}(\tpq)\|_\infty$ (in red), as functions of $q$. Center: the same curves from the left-hand side plotted in semilogarithmic scale. Right: $\sqrt2\max_{\tpq}\|\widehat{T_{1,s}}(\tpq)\|_\infty - \max_{\tpq}\|\widehat{T_{3,s}}(\tpq)\|_\infty$, plotted in semilogarithmic scale.}\label{f:globalmax}
\end{figure}

The left-hand side and center of Figure \ref{f:globalmax} also show that $\sqrt2\max_{\tpq}\|\widehat{T_{1,s}}(\tpq)\|_\infty > \max_{\tpq}\|\widehat{T_{3,s}}(\tpq)\|_\infty$ for all $\tpq$; moreover, $\sqrt2\max_{\tpq}\|\widehat{T_{1,s}}(\tpq)\|_\infty$ grows more quickly than $\max_{\tpq}\|\widehat{T_{3,s}}(\tpq)\|_\infty$, as can be seen on the right-hand side of Figure \ref{f:globalmax}, where we have plotted $\sqrt2\max_{\tpq}\|\widehat{T_{1,s}}(\tpq)\|_\infty - \max_{\tpq}\|\widehat{T_{3,s}}(\tpq)\|_\infty$ in semilogarithmic scale. Therefore, the conjecture that, when $M = 3$,
$$
\max_{\tpq}\|\widehat{\Ts}(\tpq)\|_\infty \equiv \sqrt2\max_{\tpq}\|\widehat{T_{1,s}}(\tpq)\|_\infty,
$$

\noindent for all $q$, can be regarded as solidly founded; so we will consider only the first component during the rest of this section.

Besides giving strong graphical evidence that the global maximum of the first component grows logarithmically with respect to $q$, it is interesting to quantify numerically how good this fitting is. We claim that $\sqrt2\max_{\tpq}\|\widehat{T_{1,s}}(\tpq)\|_\infty = a\ln(q) + b$, so, by defining $\tilde q = \ln(q)$, the computation of $a$ and $b$ is reduced to a regression line problem. However, instead of following a standard minimum square approach, we have taken a much more challenging path: to only use the information of the two largest values of $q$ used in our numerical experiments, i.e., $q = 199982$, with $\sqrt2\max_{\tpq}\|\widehat{T_{1,s}}(\tpq)\|_\infty = 3.302621251065180$, and $q = 200006$, with $\sqrt2\max_{\tpq}\|\widehat{T_{1,s}}(\tpq)\|_\infty = 3.302652216764496$; remark that the diference between both global maxima is of only $3\cdot10^{-5}$. After obtaining the values $a = 0.258039752572419$ and $b = 0.152992510344641$, we have plotted in Figure \ref{f:globalerror}, in semilogarithmic scale, $|\sqrt2\max_{\tpq}\|\widehat{T_{1,s}}(\tpq)\|_\infty - a\ln(q) - b|$ as a function of $q$. Except for the smallest values of $q$, the errors are always under $10^{-3}$; moreover, the errors steadily decrease as $q$ increases. For instance, when $q \ge 554$, the error is $9.8359\cdot10^{-4}$; when $q \ge 5578$, it is $9.9859\cdot10^{-5}$; when $q \ge 36154$, it is $9.9926\cdot10^{-6}$; when $q \ge 106598$, it is $9.967\cdot10^{-7}$; and when $q \ge 166798$, it is $9.9989\cdot10^{-8}$. The closer we are to $q = 200006$, the faster the errors decay. Therefore, in our opinion, there is concluding evidence that
\begin{equation}
\label{logg}
\max_{\tpq}\|\widehat{\Ts}(\tpq)\|_\infty = \sqrt2\max_{\tpq}\|\widehat{T_{1,s}}(\tpq)\|_\infty = a\ln(q) + b + \mathcal O(\mbox{lower-order terms}).
\end{equation}

\begin{figure}[!htb]
	\centering
	\includegraphics[width=0.5\textwidth, clip=true]{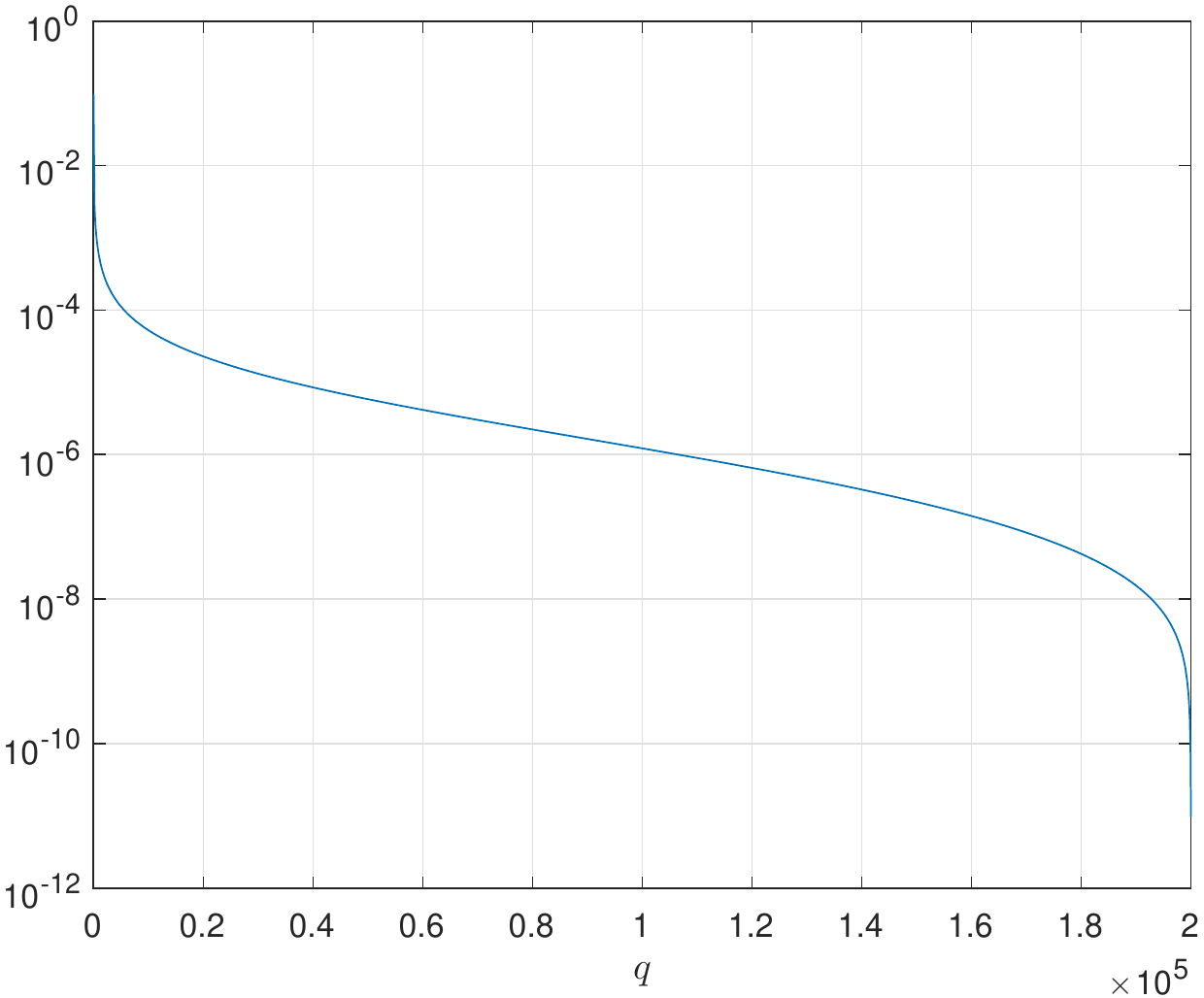}
	\caption{$|\sqrt2\max_{\tpq}\|\widehat{T_{1,s}}(\tpq)\|_\infty - a\ln(q) - b|$, for $a = 0.258039752572419$, $b = 0.152992510344641$.}\label{f:globalerror}
\end{figure}

\noindent Hence, bearing in mind that, except for some small values of $q$, $\tpq = t_{2, q} = 4\pi/(M^2q)$, we have $a\ln(q) + b = a\ln(4\pi/(M^2\tpq)) + b = -0.258039752572419\ln(\tpq) + 0.239126094505514$. Therefore, since $\|\widehat{\Ts}(-t)\|_\infty = \|\widehat{\Ts}(t)\|_\infty$, we can conjecture that
\begin{equation}
\label{e:log}
\varlimsup\limits_{t\to0}\frac{\|\widehat{\Ts}(t)\|_\infty}{-\alpha\ln|t| + \beta} = 1,
\end{equation}

\noindent where $\alpha \approx 0.258039752572419$, $\beta \approx 0.239126094505514$. In order to approximate $\alpha$ and $\beta$ with more accuracy, it is enough to repeat the previous analysis for larger values of $q$.

We have also performed some preliminary numerical tests on other rational numbers, and the jumps in Figure \ref{f:maxhatTsq200006} seem to grow again logarithmically with respect to $q$. More research is needed here. Furthermore, we have considered briefly some other values of $M$. In general, when $M > 3$, $\sqrt2\max_{\tpq}\|\widehat{T_{1,s}}(\tpq)\|_\infty$ seems to be reached always at a value of $\tpq$ close to $0^+$, but $\max_{\tpq}\|\widehat{T_{3,s}}(\tpq)\|_\infty$ can be reached at a value of $\tpq$ close to a time different from $t_{1,4}$. Besides, the graph of $\sqrt2\max_{\tpq}\|\widehat{T_{1,s}}(\tpq)\|_\infty$ with respect to $q$ is no longer a single curve, as in the left-hand side of Figure \ref{f:maxhatTsq200006}, but exhibits a complexity similar to that of the third component. Finally, $\sqrt2\max_{\tpq}\|\widehat{T_{1,s}}(\tpq)\|_\infty > \max_{\tpq}\|\widehat{T_{3,s}}(\tpq)\|_\infty$ seeems to be true only for $q > q_0(M)$, with $q_0(M)$ large enough. However, in spite of those important differences with respect to $M = 3$, the maxima of both the first and the third component seems to grow always logarithmically with respect to $q$, so, in our opinion, it does not seem unfounded at all to claim that \eqref{e:log} is also valid for all $M$. A detailed study of these facts lies beyond the scope of this paper.

Although we have only analyzed $\widehat{\T_{M,s}}$, the study of the behavior $\widehat{\X_{M,t}}$ appears to be equally interesting. However, unlike $\T_{M,s}$, $\X_{M,t}$ has to be understood in a distributional sense, so giving sense to it from a numerical point of view is pretty delicate. On the other hand, for a given $q$, it appears that $\max_k|\widehat{\X_{M,t}}(k, \tpq)|$ can be computed numerically in a consistent way, and preliminary numerical tests seem to suggest that its behavior (both quantitatively and qualitatively) is rather similar to that of $\max_k|k\, \widehat{\T_{M}}(k, \tpq)|$. Again, this topic deserves further research.

\section{Transfer of linear momentum}
\label{s:momentum}

The motivation of this section comes from some recent result proved in \cite{BanicaVega2016} about the conservation law associated to the linear momentum of solutions of \eqref{e:xt}. A simple computation proves that regular solutions of \eqref{e:xt} preserve the linear momentum
$$
\mathcal M(t) = \int_{-\infty}^{+\infty}\X(s,t)\wedge\T(s,t)ds.
$$

\noindent In \cite{Ricca}, the connection of this invariant with the so-called fluid impulse is proved. Remember that, (see \cite[p. 24]{Bertozzi-Majda}) for a 3D fluid governed by the Euler equations, with a regular vorticity $\omega$ having an appropriate decay at infinity, the fluid impulse given by
$$
\int x \wedge \omega (x, t)dx 
$$

\noindent is conserved in time. It turns out (see the appendix in \cite{BanicaVega2016}) that, in the case of the selfsimilar solution \eqref{e:selfsimilar}, the linear momentum is given by
$$
\mathcal M(t)=2c_0|t| (0,A_2,A_3),
$$ 

\noindent and, therefore, it is not preserved. Obviously, this is due to the boundary conditions that are satisfied at infinity by $\X_{c_0}$. Therefore, it seems a very natural question to try to understand what the behavior of the linear momentum density is, in the case of a regular polygon:
\begin{equation}
\label{e:lmdensity}
\mathbf\rho_M(s, t) = \X_M(s,t)\wedge\T_M(s,t)ds,
\end{equation}

\noindent for which, in that case, we have
\begin{equation}
\label{e:lm}
\mathcal {M}_M(t)=\int_0^{2\pi}\mathbf\rho_M(s,t)ds.
\end{equation}

\noindent In order to compute \eqref{e:lm} accurately, we first observe that the components of $\X_M(s,t)\wedge\T_M(s,t)$ are piecewise constant at times of the form $t_{pq}$, which is proved in the following, more general lemma.

\begin{lema} Let $\X(\alpha)$ be a skew polygon, with vertices located at $\{\alpha_j\,|\,j\in\mathbb Z\}$, and let $\T(\alpha)$ be its tangent vector. Then, $\X(\alpha)\wedge\T(\alpha)$ is constant at $\alpha\in(\alpha_j,\alpha_{j+1})$.
	
\begin{proof} In the interval $\alpha\in(\alpha_j,\alpha_{j+1})$, $\X(\alpha)$ is a line segment, whereas $\T(\alpha)$ is constant. Therefore, in the worst case, $\X(\alpha)\wedge\T(\alpha)$ will be a line segment as well, so the proof is reduced to showing that $\X(\alpha_j^+)\wedge\T(\alpha_j^+) = \X(\alpha_{j+1}^-)\wedge\T(\alpha_{j+1}^-)$. However, since $\X(\alpha)$ is continuous for all values of $\alpha$, and $\T(\alpha_j^+) \equiv \T(\alpha_{j+1}^-)$, it follows that $\X(\alpha_{j+1}^-)\wedge\T(\alpha_{j+1}^-) - \X(\alpha_j^+)\wedge\T(\alpha_j^+) = [\X(\alpha_{j+1}) - \X(\alpha_j)]\wedge\T(\alpha_j^+) = \mathbf 0$, because
\begin{equation}
\T(\alpha_j^+) \equiv \frac{\X(\alpha_{j+1}) - \X(\alpha_j)}{\|\X(\alpha_{j+1}) - \X(\alpha_j)\|}.
\end{equation}
		
\end{proof}
	
\end{lema}

\noindent In particular, since $\X_M(s,t_{pq})$ is a skew polygon, it follows that $\X_M(s,t_{tpq})\wedge\T_M(s,t_{tpq})$ is constant at $s\in(s_j,s_{j+1})$, $s_j = 2\pi j/(Mq)$, $j\in\{0,\ldots, Mq-1\}$. Therefore, we can compute \eqref{e:lm} \emph{exactly} as follows:
\begin{align*}
\mathcal M(t_{pq}) & = \int_0^{2\pi}\X_M(s,t_{pq})\wedge\T_M(s,t_{pq})ds
\cr
& = \frac{2\pi}{Mq}\sum_{j = 0}^{Mq-1}\frac{\X_M(s_j^+,t_{pq})\wedge\T_M(s_j^+,t_{pq}) + \X_M(s_{j+1}^-,t_{pq})\wedge\T_M(s_{j+1}^-,t_{pq})}{2}
\cr
& = \frac{2\pi}{Mq}\sum_{j = 0}^{Mq-1}\X_M(s_j,t_{pq})\wedge\T_M(s_j^+,t_{pq}).
\end{align*}

\noindent Observe that this last formula is valid for both $q$ even and odd.

By symmetry considerations, it is easy to see that the two first components of $\mathcal M_M(t)$ are zero. Moreover, the numerical simulations show immediately that the third component remains constant in time. Therefore, to determine $\eqref{e:lm}$, it is enough to compute it at $t = 0$, which can be done explicitly. Indeed, bearing in mind that $\X_M(s_j, 0, 0) = (\pi\sin(\pi(2j-1)/M) / (M\sin(\pi/M)), -\pi\cos(\pi(2j-1)/M) / (M\sin(\pi/M)), 0)$, $\T_M(s_j^+, 0) = (\cos(2\pi j/M), \sin(2\pi j/M), 0)$, we get immediately
\begin{align*}
\mathcal {M}_M(t) & = \left(0, 0, \frac{2\pi}{M}\sum_{j=0}^{M-1}\left[\frac{\pi \sin(\pi(2j-1)/M)}{M\sin(\pi / M)}\sin(2\pi j/M) + \frac{\pi \cos(\pi(2j-1)/M)}{M\sin(\pi / M)}\cos(2\pi j/M)\right]\right)^T
\cr
& = \left(0, 0, \frac{2\pi^2}{M\tan(\pi / M)}\right)^T, \quad \forall t.
\end{align*}

\noindent On the other hand, in the previous sections, we have given very strong evidence that the behavior of the solution at time close to a rational is dictated by the one-corner problem. Therefore, the result in \cite{BanicaVega2016} suggests a chaotic behavior of the quantity
$$
\int_0^{2\pi/M}\mathbf\rho_M(s, t)ds.
$$

\noindent Reasoning as above, this quantity can be computed in the rational times \emph{exactly} as
\begin{equation}
\label{e:lmloc}
\int_0^{2\pi/M}\mathbf\rho_M(s, \tpq)ds = \frac{2\pi}{Mq}\sum_{j = 0}^{q-1}\X_M(s_j,t_{pq})\wedge\T_M(s_j^+,t_{pq}),
\end{equation}

\noindent where $s_j = 2\pi j/(Mq)$, $j\in\{0,\ldots, q-1\}$. Again, this formula is valid for both $q$ even and odd; and the numerical simulations immediately show in this case that the first and third components of \eqref{e:lmloc} remain constant, with values respectively equal to zero and $2\pi^2/(M^2\tan(\pi/M))$; whereas the second component exhibits a behavior that strongly reminds us of Riemann's nondifferentiable function \cite{Ja}:
\begin{equation}
\label{e:riemann}
\phi(x) = \sum_{n = 1}^\infty\frac{\sin(\pi n^2 x)}{n^2},
\end{equation}

\noindent and also of the trajectory described by $\X_{M}(0, t)$. This latter fact can be easily guessed by plugging in \eqref{e:lmloc} the identity
$$
\X_M(s_j,t_{pq})=\X_M(0,t_{pq})+\int_0^{s_j}\T(s,t)\,ds.
$$

\noindent In this section, given $\tpq$, we have reconstructed algebraically $\X_M(s, \tpq)$. In Figure \ref{f:momentum}, on the left-hand side, we have plotted the second component of \eqref{e:lmloc} with respect to $p / q$, for $M \in \{3, \ldots, 10\}$, $q = 2^3\times3\times5\times7\times11 = 9240$, and $p\in\{0, \ldots, q\}$; and, on the right-hand side, we have plotted $-\phi(x)$, for $x \in\{0, 10^{-3}, 2\times10^{-3}, \ldots, 1\}$, taking $n\in\{1, \ldots, 10000\}$ in \eqref{e:riemann}. Although different scaled and not identical, the curves on the left-hand side of Figure \ref{f:momentum} are strikingly similar to $-\phi(x)$.

\begin{figure}[!htb]
	\centering
	\includegraphics[width=0.49\textwidth, clip=true]{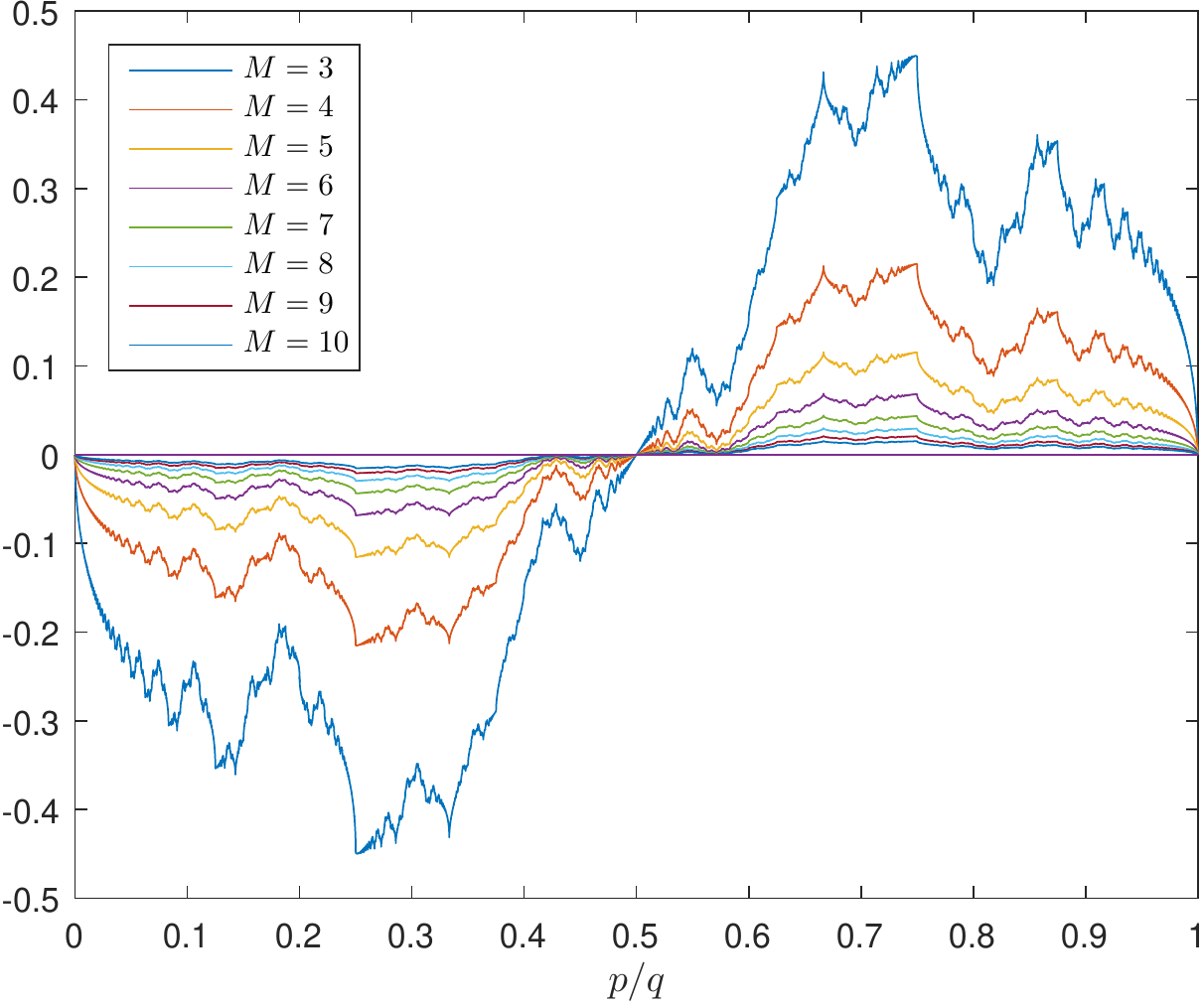}
	\includegraphics[width=0.49\textwidth, clip=true]{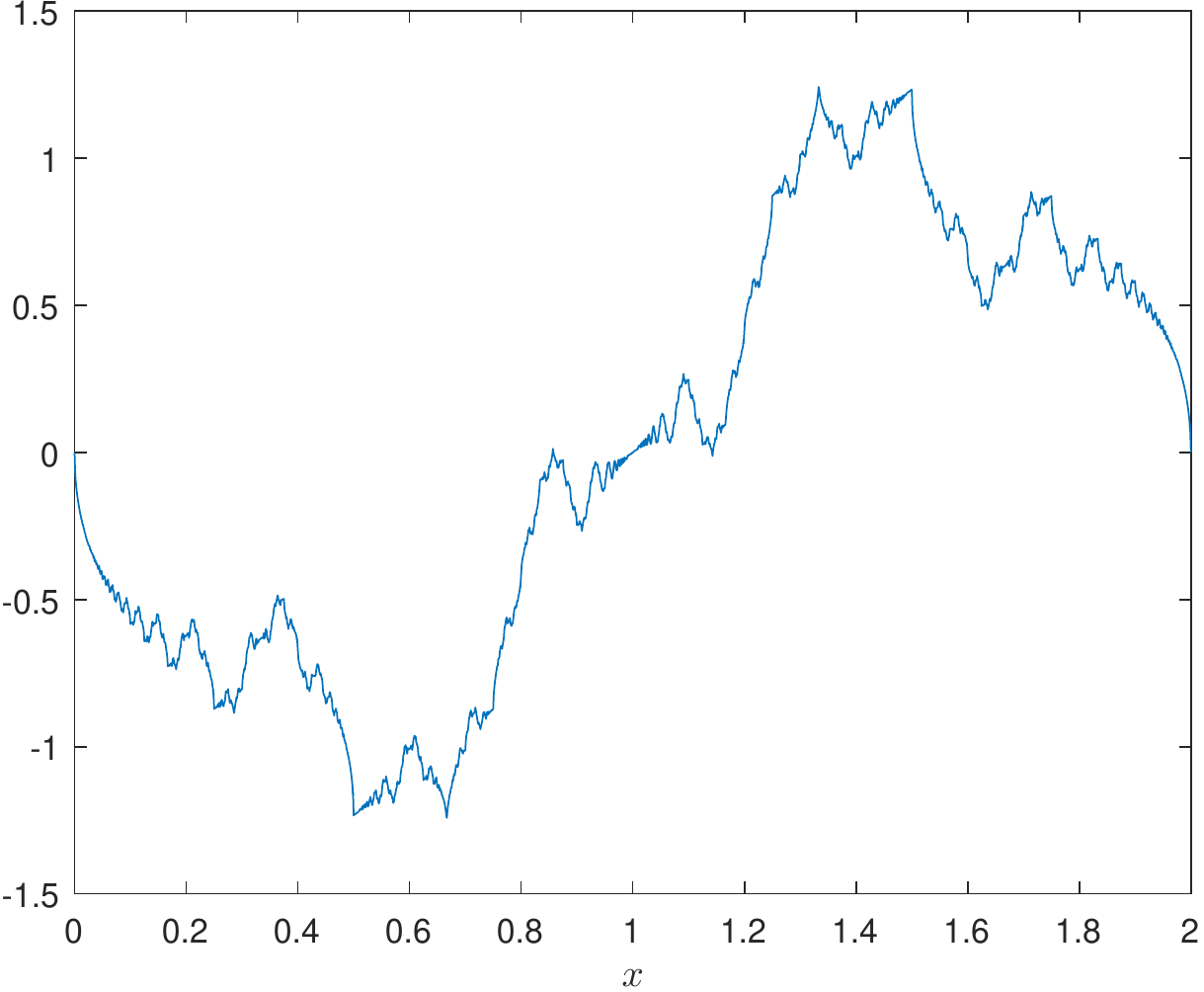}
	\caption{Left: Second component of \eqref{e:lm}, for $M \in \{3, \ldots, 10\}$, and $q = 2^3\times3\times5\times7\times11 = 9240$. We conjecture that this curve is a multifractal. Right: $-\phi(x)$, for $x \in\{0, 10^{-3}, 2\times10^{-3}, \ldots, 1\}$, where $\phi(x)$ is given by \eqref{e:riemann}, taking $n\in\{1, \ldots, 10000\}$. The curves on the left-hand side are strikingly similar to that on the right-hand side.} \label{f:momentum}
\end{figure}

In order to better understand the behavior of the second component of \eqref{e:lmloc}, we have expanded it into its sine expansion, 
\begin{equation}
\label{e:cn}
\int_0^{2\pi/M}[\X_M(s,t)\wedge\T_M(s,t)]_2ds = -\sum_{k = 1}^\infty c_k\sin(2\pi k\,t),
\end{equation}

\noindent which can be approximated by means of a discrete Fourier transform. In Figure \ref{f:fingerprintriemann}, we have plotted the approximations of $c_kk$, $k\in\{1, \ldots, 1800\}$, for $M = 3$. Although there are 1800 points, the dominating ones, marked with a star, are exactly those of the form $c_{n^2}n^2$, $n\in\{1, \ldots, 42\}$; this and the fact that these 42 values do not deviate largerly from a constant, shed light on the connection between \eqref{e:cn} and \eqref{e:riemann}. This topic deserves further research.
\begin{figure}[!htb]
	\centering
	\includegraphics[width=0.50\textwidth, clip=true]{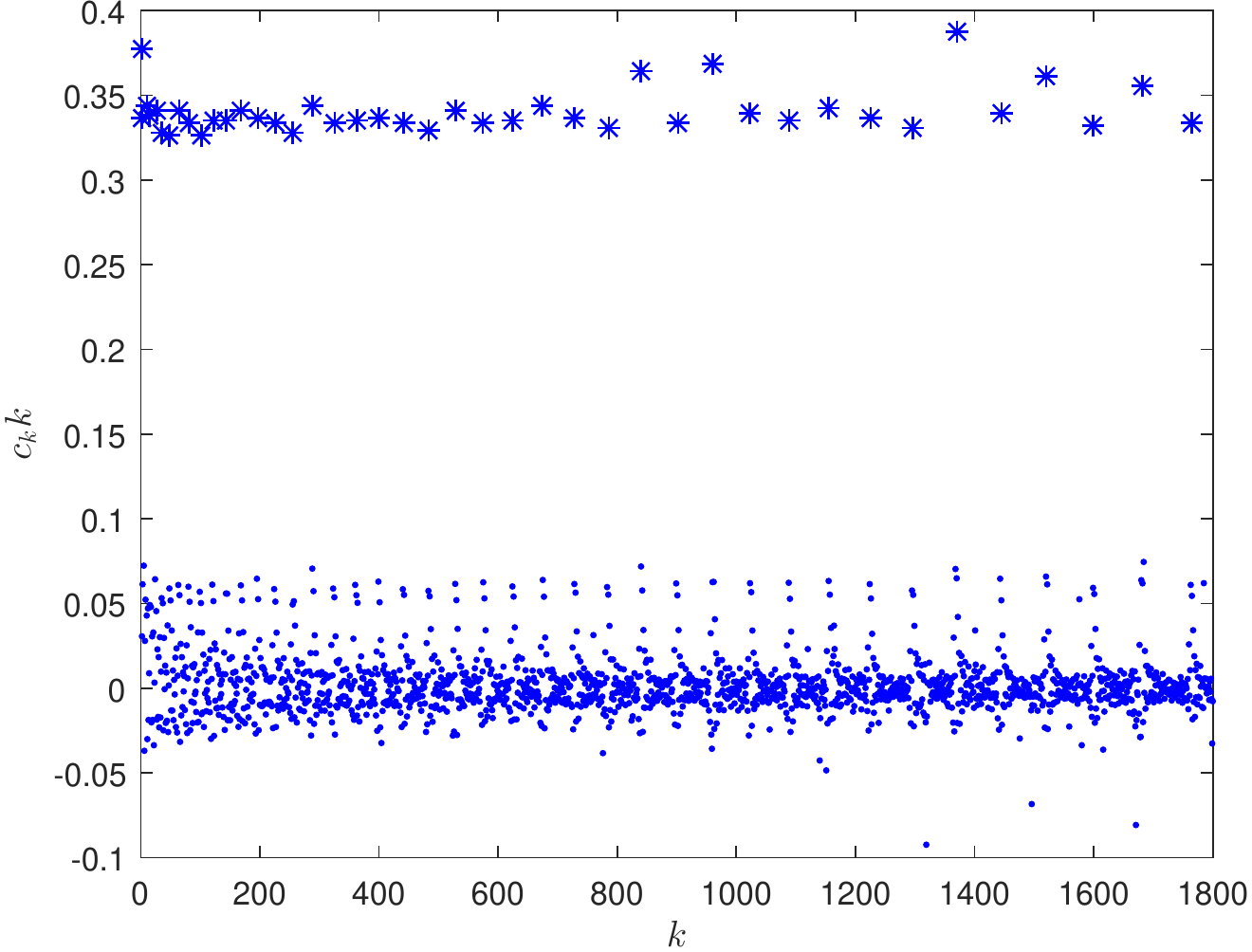}
	\caption{Approximations of $c_kk$, $k\in\{1, \ldots, 1800\}$, as a function of $k$, for $M = 3$. Although there are 1800 points, the dominating ones, marked with a star, are exactly those of the form $c_{n^2}n^2$, $n\in\{1, \ldots, 42\}$. This is in agreement with \eqref{e:riemann}.} \label{f:fingerprintriemann}
\end{figure}

\section{A couple of observations about more general polygons}

\label{s:observations}

Even if a complete study of the evolution of VFE for arbitrary polygons lies beyond the scope of this paper, we would like to point out a couple of observations that can be useful in order to extend the results in this article. The first one concerns the deduction and generalization of \eqref{e:cosrho}, which is equivalent to\begin{equation}
\label{e:cosrho3}
\cos^M\left(\frac{2\pi / M}{2}\right) =
\begin{cases}
\cos^{Mq}(\rho/2), & \mbox{if $q\equiv1\bmod2$},
 \cr
\cos^{Mq/2}(\rho/2), & \mbox{if $q\equiv0\bmod2$}.
\end{cases}
\end{equation}

\noindent At this point, bearing in mind that $2\pi/M$ is precisely the angle between two adjacent sides at time $t = 0$ and $t = t_{1,2}$, it follows that \eqref{e:cosrho3}, and hence \eqref{e:cosrho}, can be regarded as a consequence of
\begin{equation}
\label{e:cosrho4}
\prod_m \cos\left(\frac{\rho_m(t_{p,q})}{2}\right) = \mbox{constant}, \quad m \in
\{0, \ldots, \mathtt{number\_of\_sides} - 1 \},
\end{equation}

\noindent i.e., in a regular polygon, the product over a period $s\in[0,2\pi)$ of the cosines of the halves of the angles between adjacent sides is a constant. Note that this is a more unifying statement, because there is no more need to distinguish between even and odd values of $q$. Moreover, one wonders immediately whether \eqref{e:cosrho4} holds for any arbitrary polygon. We have performed some numerical experiments and the answer seems to be in the positive. For instance, we have computed the numerical evolution of VFE for an irregular planar quadrilateral whose tangent vector is given by
\begin{equation}
\label{e:irregularfourT0}
\T(s, 0) =
\begin{cases}
\T_0\equiv(1, 0, 0)^T, & s\in[0, 12\cdot\frac{2\pi}{32}),
 \\
\T_1\equiv(-12/13, 5/13, 0)^T, & s\in[12\cdot\frac{2\pi}{32}, 25\cdot\frac{2\pi}{32}),
 \\
\T_2\equiv(-4/5, -3/5, 0)^T, & s\in[25\cdot\frac{2\pi}{32}, 28\cdot\frac{2\pi}{32}),
 \\
\T_3\equiv(3/5, -4/5, 0)^T, & s\in[28\cdot\frac{2\pi}{32}, 2\pi),
\end{cases}
\end{equation}

\noindent i.e., the total length is $2\pi$, and the sides are proportional to $3$, $4$, $12$ and $13$, respectively. We have taken $N = 2^{11}\cdot3\cdot5 = 30720$ and $\Delta t = (2^{17}\cdot 3^4\cdot5^2)^{-1}\pi = \pi / 265420800$. Then, at $t = \pi / 32$, and $t = \pi / 16$, we have clearly skew polygons with exactly 32 equally-spaced sides, as is shown in Figures \ref{f:irregularfourX} and \ref{f:irregularfourT}.
\begin{figure}[!htb]
\centering
\includegraphics[width=0.5\textwidth, clip=true]{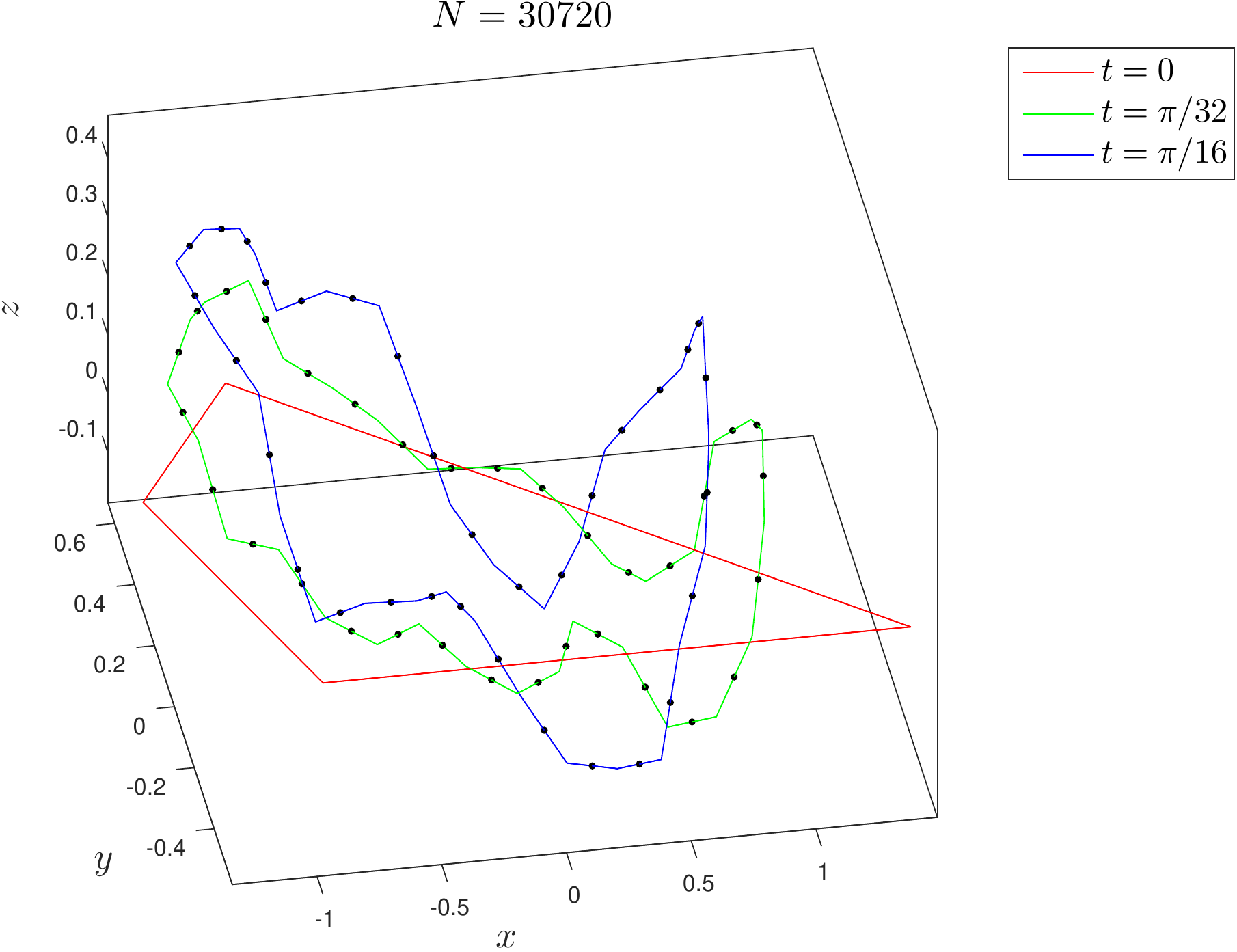}
\caption{Plots of $\X$ at $t = 0$ (irregular planar quadrilateral), and $t = \pi / 16$ and $t = \pi / 32$. At the two latter times, we have skew polygons with 32 equally-spaced sides; to ease the counting process, the middle points of the sides are indicated with a small black dot.} \label{f:irregularfourX}
\end{figure}

\begin{figure}[!htb]
\centering
\includegraphics[width=0.49\textwidth, clip=true]{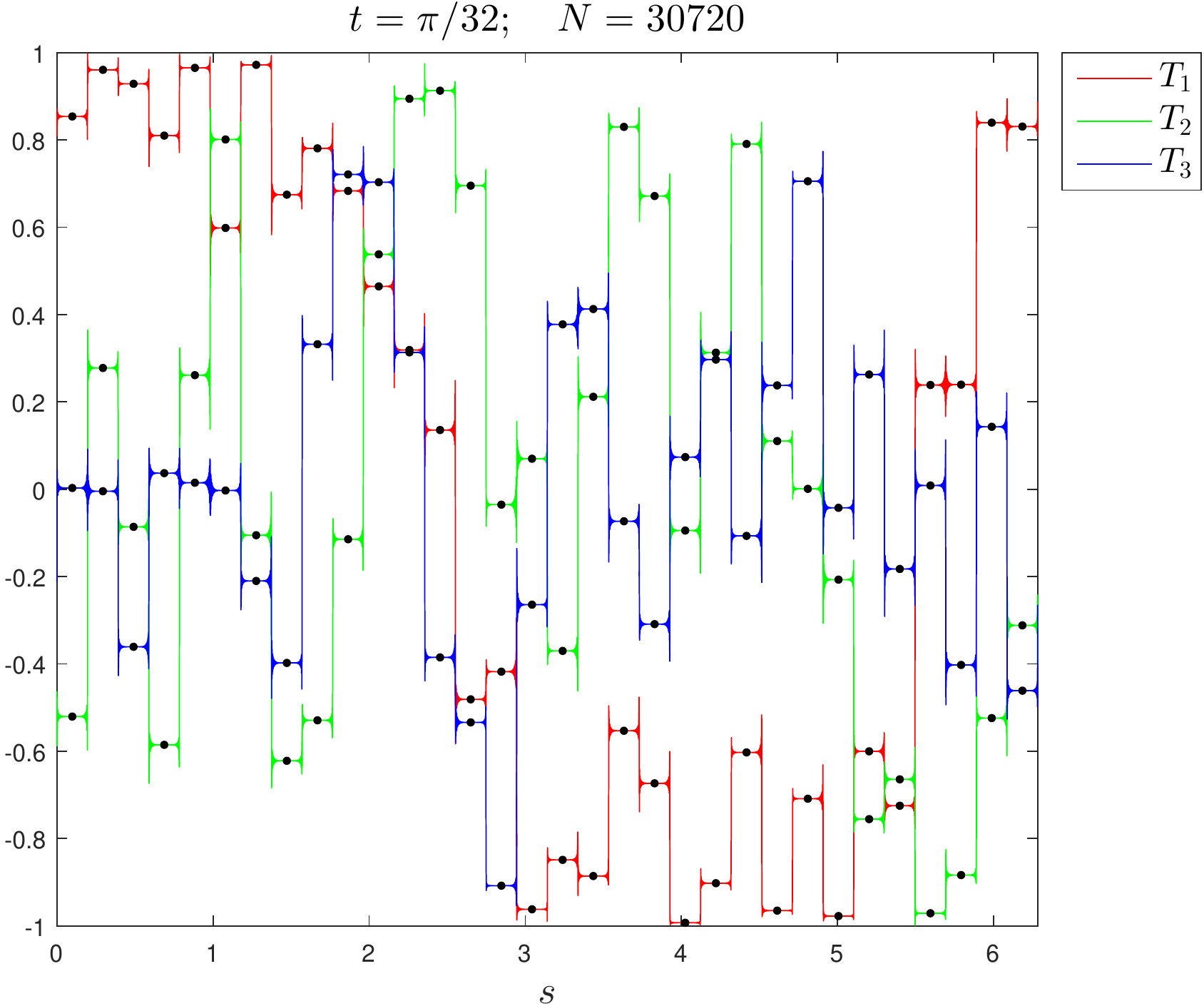}
\includegraphics[width=0.49\textwidth, clip=true]{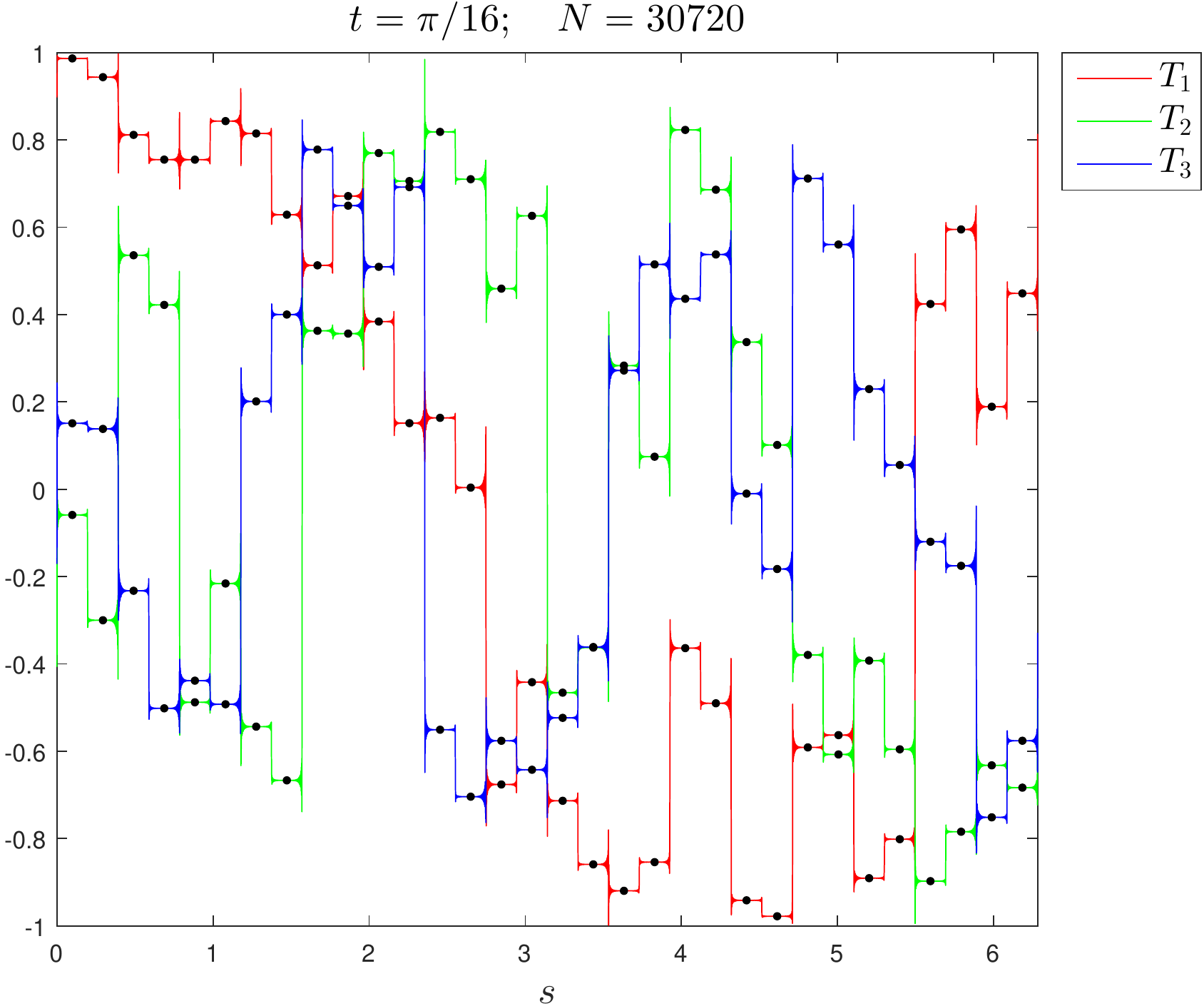}
\caption{Plots of the three components of $\T$ at $t = \pi / 32$ (left) and $t = \pi / 16$ (right), which clearly show the appearance of 32 equally spaced sides; to ease the counting process, the middle points of the sides are indicated with a small black dot. As in \cite[Figure 1]{HozVega2014}, the Gibbs phenomenon is present.}\label{f:irregularfourT}
\end{figure}

In order to test whether \eqref{e:cosrho4} holds for this example, we have to compute
\begin{equation}
\label{e:cosrho5}
P(t) \equiv \prod_m \cos\left(\frac{\rho_m(t)}{2}\right) = \left[\prod_m \frac{1 + \T_m(t)\cdot\T_{m+1}(t)}{2}\right]^{1/2},\quad m \in
\{0, \ldots, \mathtt{number\_of\_sides} - 1 \},
\end{equation}

\noindent where $t$ is a fractional multiple of $\pi$, i.e., a time when $\X$ exhibits the shape of a skew polygon. When $t = 0$, $P(0) = 7/65$ can be explicitly calculated from \eqref{e:irregularfourT0}. For other $t$, due to the Gibbs phenomenon, we choose $\T_m$, for a given side, to be equal to the mean of the inner half of the numerical values of the tangent vector at that side; for instance, in $s\in[0, 2\pi/32)$, $\T(s) = \T_0 \equiv \mean_{s\in[2\pi/128, 6\pi/128)}\T(s)$, etc. Following this procedure, we find that $|P(\pi / 32) - 7/65| = 5.5542\cdot10^{-8}$ and $|P(\pi / 16) - 7/65| = 1.1855\cdot10^{-6}$. In Table \ref{t:irregularfourerror}, we have computed $P(t)$ for this example for more times, obtaining quite satisfying results as well. A more careful computation of $\T_m$ may further improve the results.
\begin{table}[htb!]
\centering
\begin{tabular}{|c|c|c||c|c|c|}
\hline $t$ & No. of sides & $|P(t) - 7/65|$ & $t$ & No. of sides & $|P(t) - 7/65|$
\\
 \hline
$\pi/160$ & $160$ & $1.1865\cdot10^{-4}$ & $3\pi/80$ & $160$ & $1.2693\cdot10^{-4}$
\\
 \hline
$\pi/128$ & $128$ & $7.4416\cdot10^{-5}$ & $5\pi/128$ & $128$ & $7.0374\cdot10^{-6}$
\\
 \hline
$\pi/96$ & $96$ & $9.2262\cdot10^{-7}$ & $\pi/24$ & $96$ & $9.1060\cdot10^{-5}$
\\
 \hline
$\pi/80$ & $160$ & $2.1636\cdot10^{-4}$ & $7\pi/160$ & $160$ & $7.1699\cdot10^{-5}$
\\
 \hline
$\pi/64$ & $64$ & $1.3437\cdot10^{-5}$ & $3\pi/64$ & $64$ & $1.4202\cdot10^{-5}$
\\
 \hline
$3\pi/160$ & $160$ & $1.1259\cdot10^{-4}$ & $\pi/20$ & $160$ & $6.2725\cdot10^{-5}$
\\
 \hline
$\pi/48$ & $96$ & $4.5971\cdot10^{-5}$ & $5\pi/96$ & $96$ & $7.4601\cdot10^{-7}$
\\
 \hline
$3\pi/128$ & $128$ & $1.4293\cdot10^{-4}$ & $7\pi/128$ & $128$ & $1.3114\cdot10^{-4}$
\\
 \hline
$\pi/40$ & $160$ & $5.1102\cdot10^{-5}$ & $9\pi/160$ & $160$ & $7.8719\cdot10^{-5}$
\\
 \hline
$\pi/32$ & $32$ & $5.5542\cdot10^{-8}$ & $\pi/16$ & $32$ & $1.1855\cdot10^{-6}$
\\
\hline
\end{tabular}
\caption{Test of \eqref{e:cosrho5}, with $\T(s, 0)$ given by \eqref{e:irregularfourT0}. As explained, in order to compute $P(t)$, we have chosen each $\T_m$ to be equal to the mean of the inner half of the numerical values of the tangent vector at that side. The results are quite satisfying.
}\label{t:irregularfourerror}
\end{table}

Obviously, the previous arguments do not constitute a proof, not even a numerical one. However, they tell us that it is very reasonable to conjecture \eqref{e:cosrho5} to be valid for any arbitrary polygon. Moreover, if $P(t)$ is a conserved quantity, so is $\log(P(t))$, i.e.,
\begin{equation*}
\log(P(t)) = \frac{1}{2}\sum_m\log\left(\frac{1 + \T_m(t)\cdot\T_{m+1}(t)}{2}\right),
\end{equation*}

\noindent which is in agreement with \cite{Ishimori1982} and \cite{Lakshmanan2011}. This deserves further study.

The second observation is that our claiming that the corners do not \textit{see} one another at infinitesimal times is valid for nonregular polygons as well. In Figure \ref{f:comparison}, we have plotted simultaneously, at $t = (2^8\cdot3^4\cdot5)^{-1}\pi = \pi / 103680$, $\T_{num}$ (black) corresponding to \eqref{e:irregularfourT0}, with $N = 30720$, and $\T_{rot}$ (thick red) corresponding to the corner at $s = 3\cdot2\pi/32$, i.e., with inner angle $\theta = \pi - \arccos(-12/13)$ in \eqref{e:thetaA1}. As in Figure \ref{f:comparison}, except for the thicker stroke, the red curve is visually undistinguishable from the black one.
\begin{figure}[!htb]
\centering
	\includegraphics[width=0.5\textwidth, clip=true]{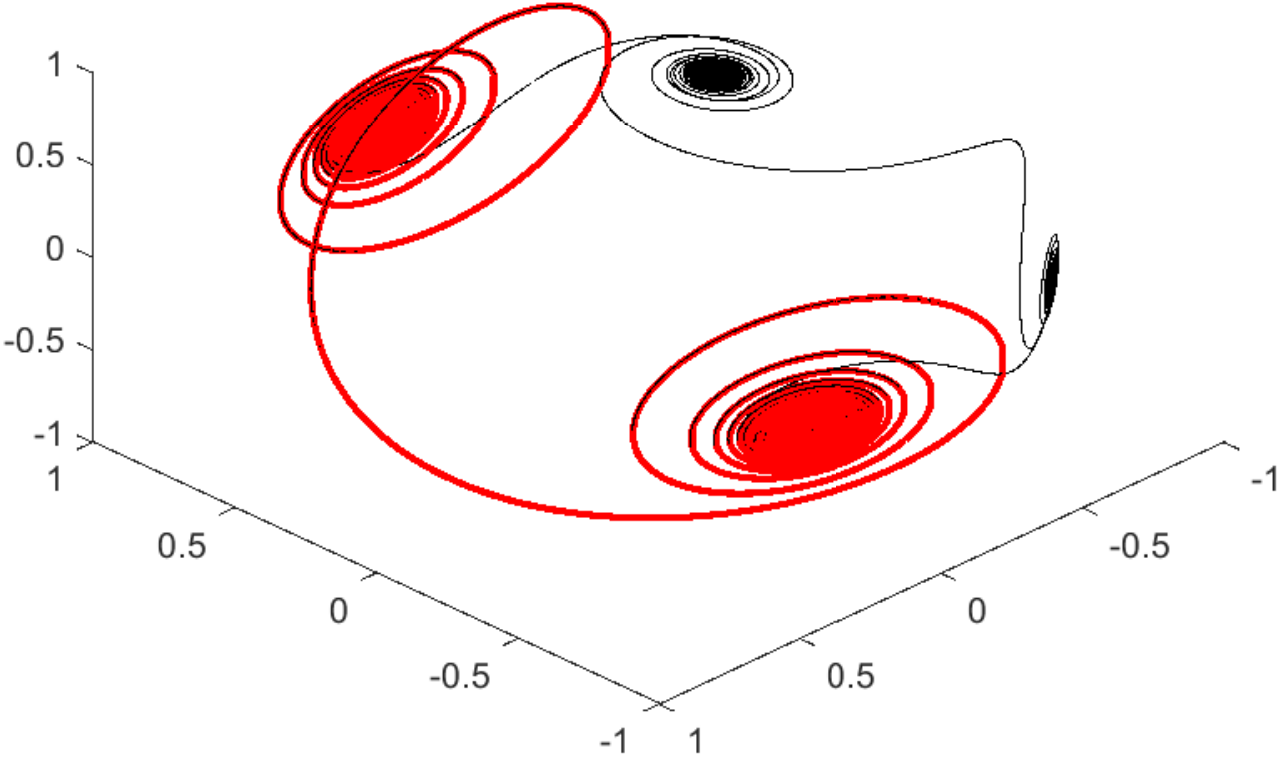}
\caption{Comparison at $t = \pi / 103680$ between the evolution of $\T_{num}$ (black) corresponding to \eqref{e:irregularfourT0}, with $N = 30720$, and $\T_{rot}$ (thick red) corresponding to the corner at $s = 3\cdot2\pi/32$. As in Figure \ref{f:comparison}, except for the thicker stroke, the red curve is visually undistinguishable from the black one.}\label{f:comparison2}
\end{figure}

\section{Conclusions}

\label{s:Conclusions}

In this paper, we have studied the evolution of a regular polygon according to the vortex filament equation (VFE). This equation is a geometric flow also known as the binormal flow or the localized induction approximation (LIA). The latter name is the one more frequently used in the literature of fluid mechanics, and refers to the velocity that an isolated vortex filament induces on itself as an effect of curvature. This velocity is obtained from the Biot-Savart law using a cut-off procedure to compute the integral, so just local effects are considered. LIA has been also extensively used as a model to describe vortices in superfluids. At this respect, see, for example, \cite{FMOHL}, where the direct observation of Kelvin waves by quantized reconnection is reported.

In the first part of this paper (Sections 2-4), we have given numerical, algebraic and analytic arguments to prove that the evolution at infinitesimal times (i.e., $t=0^+$) of one corner is independent from the other ones. As a consequence, we can understand the dynamics of a regular polygon as the nonlinear interaction of different filaments, one for each corner. In fact, the mathematical description of the regular polygon is given by its curvature, which is written as a sum of periodic deltas that have an amplitude determined by the Gau{\ss}-Bonet theorem (i.e., $2\pi/M$, if $M$ is the number of sides). Therefore, from this point of view, we are considering the interaction of infinitely many filaments. Let us recall that the case of a filament with just one corner has been extensively studied from a theoretical point of view in \cite{BV0,BV1,BV2,BV3,BanicaVega2016,delahoz2007,GutierrezRivasVega2003}, and, from a numerical point of view, in \cite{buttke87,DelahozGarciaCerveraVega09}, so we can therefore say that it is fairly well understood. In particular, it is known that corners can appear and disappear in a stable way.

As it was observed in \cite{HozVega2014}, the dynamics for later times exhibits a Talbot effect, so that, at rational multiples of the period (i.e., $(p/q)(2\pi/M^2)$), new polygons appear. The number of sides of the new polygons depend on $q$, and they behave in a random way (see \cite{HozVega2014b}). As a consequence, an intermittent phenomenon of creation/annihilation of corners is observed. One of the simplest examples at the linear level of this intermittency generated by the Talbot effect is the so-called Riemann's nondifferentiable function \cite{Ja}:
\begin{equation}
\label{e:riemann2}
\phi(x) = \sum_{n = 1}^\infty\frac{\sin(\pi n^2 x)}{n^2}.
\end{equation}

\noindent At this respect,  it is proved in \cite{Ja} that the set of times that have the same regularity (measured in terms of their H\"older exponent) is a fractal which has a dimension that fits within the so-called Frisch-Parisi conjecture; we refer the reader to \cite{Ja} for the details. The Talbot effect in nonlinear dispersive equations has been studied in \cite{CO1,CO2,ET2,Olver}. All these results are at the subcritical level of regularity, so the complex dynamics is the one inherited by the free evolution.

In Section \ref{s:energy} of this paper, we have given very strong numerical evidence that intermittency and multifractality are also present in the evolution of a regular polygon according to VFE. Although the dynamics is very similar at the qualitative level, it depends nonetheless on $M$, as can be inferred from Figures \ref{f:momentum} and \ref{f:fingerprintriemann}, corresponding to the time plot of a truncated linear momentum. In Section \ref{s:momentum}, we have shown that this Talbot effect is nevertheless purely nonlinear due to the existence of a phenomenon of transfer of energy that can not be present at the linear level. We believe that this latter property is a far reaching result. 

One could wonder if the above properties are still true if a general polygon instead of a regular one is considered. This is studied in the last section of this paper, where, besides proving at the numerical level the stability of the previous results, we observe that the periodicity in time of the dynamics is lost. This, of course, opens the way to create much more complicated dynamics, by choosing the sides of the polygons in an appropriate way.

The final conclusion is that we have exhibited a nonlinear geometric flow, obtained as an approximation of the evolution of vortex filaments, and which is amenable to having a nonlinear Talbot effect that, besides the usual properties of randomness, multifractality, and intermittency, has also transfer of energy.

\section*{Acknowledgements}

We want to thank V. Banica and C. Garc\'ia-Cervera for very enlightening conversations concerning the last two sections of this paper. Part of this work was started while the second author was visiting MSRI, within the New Challenges in PDE 2015 program.


\begin{thebibliography}{10}
	
	\bibitem{arms}
	R.~J. Arms and F.~R. Hama.
	\newblock {Localized-Induction Concept on a Curved Vortex and Motion of an
		Elliptic Vortex Ring}.
	\newblock {\em Physics of Fluids}, 8(4):553--559, 1965.
	
	\bibitem{BV0}
	V.~Banica and L.~Vega.
	\newblock {On the Stability of a Singular Vortex Dynamics}.
	\newblock {\em Comm. Math. Phys.}, 286(2):593--627, 2009.
	
	\bibitem{BV1}
	V.~Banica and L.~Vega.
	\newblock {Scattering for 1D cubic NLS and singular vortex dynamics}.
	\newblock {\em J. Eur. Math. Soc. (JEMS)}, 14(1):209--253, 2012.
	
	\bibitem{BV2}
	V.~Banica and L.~Vega.
	\newblock {Stability of the Self-similar Dynamics of a Vortex Filament}.
	\newblock {\em Archive for Rational Mechanics and Analysis}, 210(3):673--712,
	2013.
	
	\bibitem{BV3}
	V.~Banica and L.~Vega.
	\newblock {The initial value problem for the Binormal Flow with rough data}.
	\newblock {\em Annales scientifiques de l'ENS}, 48(6):1423--1455, 2015.
	
	\bibitem{BanicaVega2016}
	V.~Banica and L.~Vega.
	\newblock {Singularity formation for the 1-D cubic NLS and the Schr\"odinger
		map on $\mathbb{S}^2$}.
	\newblock {\em arXiv:1702.01947}, 2016.
	
	\bibitem{Berry}
	M.~V. Berry and S.~Klein.
	\newblock {Integer, fractional and fractal Talbot effects}.
	\newblock {\em J. Mod. Optics}, 43:2139--2164, 1996.
	
	\bibitem{buttke87}
	T.~F. Buttke.
	\newblock {A Numerical Study of Superfluid Turbulence in the Self-Induction
		Approximation}.
	\newblock {\em J. Comput. Phys.}, 76(2):301--326, 1998.
	
	\bibitem{CO1}
	G.~Chen and P.~J. Olver.
	\newblock {Dispersion of discontinuous periodic waves}.
	\newblock {\em Proc. R. Soc. Lond. A}, 2012.
	
	\bibitem{CO2}
	G.~Chen and P.~J. Olver.
	\newblock {Numerical simulation of nonlinear dispersive quantization}.
	\newblock {\em Discrete Contin. Dyn. Syst.}, 2014.
	
	\bibitem{ET1}
	V.~Chousionis, M.~B. Erdo\u{g}an, and N.~Tzirakis.
	\newblock {Fractal solutions of linear and nonlinear dispersive partial
		differential equations}.
	\newblock {\em Proc. Lond. Math. Soc.}, 110(3):543--564, 2015.
	
	\bibitem{delahoz2007}
	F.~de~la Hoz.
	\newblock {Self-similar solutions for the 1-D Schr\"odinger map on the
		hyperbolic plane}.
	\newblock {\em Math. Z.}, 257(1):61--80, 2007.
	
	\bibitem{DelahozGarciaCerveraVega09}
	F.~de~la Hoz, C.~J. {Garc\'\i a-Cervera}, and L.~Vega.
	\newblock {A Numerical Study of the Self-Similar Solutions of the Schr\"odinger
		Map}.
	\newblock {\em SIAM J. Appl. Math.}, 70(4):1047--1077, 2009.
	
	\bibitem{HozVega2014}
	F.~de~la Hoz and L.~Vega.
	\newblock {Vortex filament equation for a regular polygon}.
	\newblock {\em Nonlinearity}, 27(12):3031--3057, 2014.
	
	\bibitem{HozVega2014b}
	F.~de~la Hoz and L.~Vega.
	\newblock {The Vortex Filament Equation as a Pseudorandom Generator}.
	\newblock {\em Acta Appl. Math.}, 138(1):135--151, 2015.
	
	\bibitem{ET2}
	M.~B. Erdo\u{g}an and N.~Tzirakis.
	\newblock {Talbot effect for the cubic nonlinear Schrödinger equation on the
		torus}.
	\newblock {\em Math. Res. Lett.}, 20(6):1081--1090, 2013.
	
	\bibitem{FMOHL}
	E.~Fonda, D.~P. Meichle, N.~T. Ouellette, S.~Hormoz, and D.~P. Lathrop.
	\newblock {Direct observation of Kelvin waves excited by quantized vortex
		reconnection}.
	\newblock {\em PNAS}, 111:4707--4710, 2014.
	\newblock Suppl. 1.
	
	\bibitem{GG}
	F.~F. Grinstein and E.~J. Gutmark.
	\newblock {Flow control with noncircular jets}.
	\newblock {\em Ann. Rev. Fluid Mech.}, 31:239--272, 1999.
	
	\bibitem{GGP}
	F.~F. Grinstein, E.~J. Gutmark, and T.~Parr.
	\newblock {Nearfield dynamics of subsonic, free square jets. A computational
		and experimental study}.
	\newblock {\em Phys. Fluids}, 7:1483--1497, 1995.
	
	\bibitem{GutierrezRivasVega2003}
	S.~Guti\'errez, J.~Rivas, and L.~Vega.
	\newblock {Formation of singularities and self-similar vortex motion under the
		localized induction approximation}.
	\newblock {\em Comm. PDE}, 28(5--6):927--968, 2003.
	
	\bibitem{hasimoto}
	H.~Hasimoto.
	\newblock {A soliton on a vortex filament}.
	\newblock {\em J. Fluid Mech.}, 51(3):477--485, 1972.
	
	\bibitem{Ishimori1982}
	Y.~Ishimori.
	\newblock {An Integrable Classical Spin Chain}.
	\newblock {\em J. Phys. Soc. Jpn.}, 51(11):3417--3418, 1982.
	
	\bibitem{Ja}
	S.~Jaffard.
	\newblock {The spectrum of singularities of Riemann's function}.
	\newblock {\em Rev. Mat. Iberoamericana}, 12(2):441--460, 1996.
	
	\bibitem{JS}
	R.~L. Jerrard and C.~Seis.
	\newblock {On the vortex filament conjecture for Euler flows}.
	\newblock {\em Arch. Ration. Mech. Anal.}, 224(1):135--172, 2017.
	
	\bibitem{Didier}
	R.~L. Jerrard and D.~Smets.
	\newblock {On Schr\"odinger maps from $T^1$ to $S^2$}.
	\newblock {\em Ann. Sci. \'Ec. Norm. Sup\'er. (4)}, 45(4):637--680, 2013.
	
	\bibitem{Didier2}
	R.~L. Jerrard and D.~Smets.
	\newblock {On the motion of a curve by its binormal curvature}.
	\newblock {\em Jour. Eur. Math. Soc.}, 17(6):1487--1515, 2015.
	
	\bibitem{Lakshmanan2011}
	M.~Lakshmanan.
	\newblock {The fascinating world of the Landau-Lifshitz-Gilbert equation: an
		overview}.
	\newblock {\em Phil. Trans. R. Soc. A}, 369:1280--1300, 2011.
	
	\bibitem{Bertozzi-Majda}
	A.~J. Majda and A.~L. Bertozzi.
	\newblock {\em {Vorticity and Incompressible Flows}}.
	\newblock Cambridge Texts in Applied Mathematics. Cambridge University Press,
	2002.
	
	\bibitem{Olver}
	P.~J. Olver.
	\newblock {Dispersive quantization}.
	\newblock {\em Amer. Math. Monthly}, 117(7):599--610, 2010.
	
	\bibitem{Ricca}
	R.~L. Ricca.
	\newblock {Physical interpretation of certain invariants for vortex filament
		motion under LIA}.
	\newblock {\em Phys. Fluids A}, 4:938--944, 1992.
	
	\bibitem{darios}
	L.~S.~Da Rios.
	\newblock {Sul moto d'un liquido indefinito con un filetto vorticoso di forma
		qualunque}.
	\newblock {\em Rend. Circ. Mat. Palermo}, 22(1):117--135, 1906.
	\newblock In Italian.
	
	\bibitem{Saffman}
	P.~G. Saffman.
	\newblock {\em {Vortex Dynamics}}.
	\newblock Cambridge Monographs on Mechanics. Cambridge University Press, 1995.
	
	\bibitem{ZWZX}
	Y.~Zhang, J.~Wen, S.~N. Zhu, and Xiao.
	\newblock {Nonlinear Talbot effect}.
	\newblock {\em M. Phys. Rev. Lett.}, 2010.
	
\end{thebibliography}
\end{document}